\documentclass[10pt,a4paper,reqno]{amsart}
\usepackage{mathtools}
\usepackage{amstext}
\usepackage{amsthm}
\usepackage{amssymb}
\usepackage[unicode=true,
 bookmarks=false,
 breaklinks=false,pdfborder={0 0 1},backref=false,colorlinks=false]
 {hyperref}

\makeatletter

\numberwithin{equation}{section}
\numberwithin{figure}{section}
\theoremstyle{plain}
\newtheorem{thm}{\protect\theoremname}[section]
\theoremstyle{plain}
\newtheorem{lem}[thm]{\protect\lemmaname}
\theoremstyle{plain}
\newtheorem{prop}[thm]{\protect\propositionname}
\theoremstyle{remark}
\newtheorem{rem}[thm]{\protect\remarkname}

\usepackage{enumitem}
\usepackage{graphicx}
\usepackage{comment}
\usepackage{xcolor}

\numberwithin{equation}{section}

\providecommand{\theoremname}{Theorem}

\providecommand{\lemmaname}{Lemma}
\providecommand{\remarkname}{Remark}
\providecommand{\propositionname}{Proposition}

\makeatother

\begin{document}
\global\long\def\R{\mathbf{\mathbb{R}}}%
\global\long\def\C{\mathbf{\mathbb{C}}}%
\global\long\def\Z{\mathbf{\mathbb{Z}}}%
\global\long\def\N{\mathbf{\mathbb{N}}}%
\global\long\def\T{\mathbb{T}}%
\global\long\def\Im{\mathrm{Im}}%
\global\long\def\Re{\mathrm{Re}}%
\global\long\def\Hc{\mathcal{H}}%
\global\long\def\M{\mathbb{M}}%
\global\long\def\P{\mathbb{P}}%
\global\long\def\L{\mathcal{L}}%
\global\long\def\F{\mathcal{\mathcal{F}}}%
\global\long\def\s{\sigma}%
\global\long\def\Rc{\mathcal{R}}%
\global\long\def\W{\tilde{W}}%
\global\long\def\G{\mathcal{G}}%
\global\long\def\d{\partial}%
\global\long\def\mc#1{\mathcal{\mathcal{#1}}}%
\global\long\def\Right{\Rightarrow}%
\global\long\def\Left{\Leftarrow}%
\global\long\def\les{\lesssim}%
\global\long\def\hook{\hookrightarrow}%
\global\long\def\D{\mathbf{D}}%
\global\long\def\rad{\mathrm{rad}}%
\global\long\def\d{\partial}%
\global\long\def\jp#1{\langle#1\rangle}%
\global\long\def\norm#1{\|#1\|}%
\global\long\def\ol#1{\overline{#1}}%
\global\long\def\wt#1{\widehat{#1}}%
\global\long\def\tilde#1{\widetilde{#1}}%
\global\long\def\br#1{(#1)}%
\global\long\def\Bb#1{\Big(#1\Big)}%
\global\long\def\bb#1{\big(#1\big)}%
\global\long\def\lr#1{\left(#1\right)}%
\global\long\def\la{\lambda}%
\global\long\def\al{\alpha}%
\global\long\def\be{\beta}%
\global\long\def\ga{\gamma}%
\global\long\def\La{\Lambda}%
\global\long\def\De{\Delta}%
\global\long\def\na{\nabla}%
\global\long\def\fl{\flat}%
\global\long\def\sh{\sharp}%
\global\long\def\calN{\mathcal{N}}%
\global\long\def\avg{\mathrm{avg}}%
\global\long\def\bbR{\mathbf{\mathbb{R}}}%
\global\long\def\bbC{\mathbf{\mathbb{C}}}%
\global\long\def\bbZ{\mathbf{\mathbb{Z}}}%
\global\long\def\bbN{\mathbf{\mathbb{N}}}%
\global\long\def\bbT{\mathbb{T}}%
\global\long\def\bfD{\mathbf{D}}%
\global\long\def\calC{\mathcal{C}}%
\global\long\def\calE{\mathcal{E}}%
\global\long\def\calF{\mathcal{F}}%
\global\long\def\calG{\mathcal{G}}%
\global\long\def\calH{\mathcal{H}}%
\global\long\def\calL{\mathcal{L}}%
\global\long\def\calO{\mathcal{O}}%
\global\long\def\calR{\mathcal{R}}%
\global\long\def\calZ{\mathcal{Z}}%
\global\long\def\Lmb{\Lambda}%
\global\long\def\eps{\varepsilon}%
\global\long\def\lmb{\lambda}%
\global\long\def\gmm{\gamma}%
\global\long\def\rd{\partial}%
\global\long\def\chf{\mathbf{1}}%
\global\long\def\td#1{\widetilde{#1}}%
\global\long\def\sgn{\mathrm{sgn}}%

\title[Soliton Resolution for CM--DNLS]{Soliton resolution for Calogero--Moser derivative nonlinear Schrödinger
equation}
\begin{abstract}
We consider soliton resolution for the Calogero--Moser derivative
nonlinear Schrödinger equation (CM-DNLS). A rigorous PDE analysis
of (CM-DNLS) was recently initiated by Gérard and Lenzmann, who demonstrated
its Lax pair structure. Additionally, (CM-DNLS) exhibits several symmetries,
such as mass-criticality with pseudo-conformal symmetry and a self-dual
Hamiltonian. Despite its integrability, finite-time blow-up solutions
have been constructed.

The purpose of this paper is to establish soliton resolution for both
finite-time blow-up solutions and global solutions in a fully general
setting, \emph{without imposing radial symmetry or size constraints}.
To our knowledge, this is the first proof of full soliton resolution for a Schrödinger-type equation that does not rely on complete integrability. A key aspect of our proof
is the control of the energy of the outer radiation after extracting
a soliton, referred to as the \emph{energy bubbling} estimate. This
benefits from two levels of conservation laws, mass and energy, and
self-duality. This approach allows us to directly prove continuous-in-time
soliton resolution, bypassing time-sequential soliton resolution.
Importantly, our proof does not rely on the integrability of the equation,
potentially offering insights applicable to other non-integrable models.
\end{abstract}

\author{Taegyu Kim}
\email{k1216300@kias.re.kr}
\address{(current) Korea Institute for Advanced Study, 80 Hoegi-ro, Dongdaemun-gu, Seoul 02455, Korea (previous) Department of Mathematical Sciences, Korea Advanced Institute of Science and Technology, 291 Daehak-ro, Yuseong-gu, Daejeon 34141, Korea}
\author{Soonsik Kwon}
\email{soonsikk@kaist.edu}
\address{Department of Mathematical Sciences, Korea Advanced Institute of Science and Technology, 291 Daehak-ro, Yuseong-gu, Daejeon 34141, Korea}

\maketitle
\tableofcontents{}


\section{Introduction}

In this paper, we study the global behavior of solutions, \emph{soliton resolution}, for the Calogero--Moser derivative nonlinear
Schrödinger equation (CM-DNLS): 
\begin{align}
	\begin{cases}
		i\partial_{t}u+\partial_{xx}u+2D_{+}(|u|^{2})u=0,\qquad u:[0,T)\times\bbR\to\bbC\\
		u(0)=u_{0}\in H^{s}(\R).
	\end{cases}\tag{CM-DNLS}\label{CMdnls}
\end{align}
Here, $D_{+}=D\Pi_{+}=-i\partial_{x}\Pi_{+}$ is a nonlocal derivative
with the projection to positive frequencies, i.e. $\Pi_{+}$ is the
Cauchy--Szeg\H{o} projection with the Fourier symbol, $\chf_{\xi>0}$.
\eqref{CMdnls} is a recently introduced nonlinear Schrödinger-type
equation. \eqref{CMdnls} draws attention as it enjoys rich and interesting
mathematical structures. To name a few, \eqref{CMdnls} is mass-critical
(with pseudo-conformal invariance), completely integrable, and has
a self-dual Hamiltonian. The purpose of this work is to show soliton
resolution theorem for \eqref{CMdnls}. We show soliton resolution in fully general setting, \emph{without size constraints or radial symmetry}. Although \eqref{CMdnls} is completely integrable, we do
not use the integrability features, and our proof may give insights
toward other non-integrable models.

\eqref{CMdnls} was introduced in \cite{AbanovBettelheimWiegmann2009FormalContinuum}
as a continuum model of the classical Calogero--Moser Hamiltonian
system, which is known to be completely integrable \cite{Calogero1971ClassicCalogerMoser,CalogeroMarchioro1974ClassicCalogerMoser,Moser1975ClassicCalogerMoser,OlshanetskyPerelomov1976Invent}.
It is also known as continuum Calogero--Moser model (CCM) or Calogero--Moser
NLS (CM-NLS). There is also a periodic counterpart, Calogero--Sutherland
model \cite{Sutherland1971ClassicCalogerMoser,Sutherland1972ClassicCalogerMoser}.

A rigorous mathematical analysis was initiated by Gérard and Lenzmann
\cite{GerardLenzmann2022}. They verified that \eqref{CMdnls} is
completely integrable by discovering the Lax pairs on the \emph{Hardy--Sobolev
	space} $H_{+}^{s}(\mathbb{R})$: 
\[
H_{+}^{s}(\mathbb{R})\coloneqq H^{s}(\R)\cap L_{+}^{2}(\R),\qquad L_{+}^{2}(\R)\coloneqq\{f\in L^{2}(\mathbb{R}):\text{supp}\,\widehat{f}\subset[0,\infty)\}.
\]
The flow of \eqref{CMdnls} preserves the positive frequency condition, namely that $\text{supp}\,\widehat{u}(t)\allowbreak\subset[0,\infty)$, and this distinctive feature is called \emph{chirality}. For chiral solutions $u(t)\in H_{+}^{s}(\R)$
to \eqref{CMdnls}, the equation admits the Lax pair structure: 
\begin{equation}
	\frac{d}{dt}\mathcal{L}_{\textnormal{Lax}}=[\mathcal{P}_{\textnormal{Lax}},\mathcal{L}_{\textnormal{Lax}}]\label{eq:LaxPair-evol}
\end{equation}
with the $u$-dependent operators $\mathcal{L}_{\textnormal{Lax}}$
and $\mathcal{P}_{\textnormal{Lax}}$ defined by\footnote{\eqref{eq:LaxPair} is slightly different but equivalent to what was
	presented in \cite{GerardLenzmann2022}. This formulation was presented
	in the introduction of \cite{KLV2025CAMS}.} 
\begin{align}
	\mathcal{L}_{\textnormal{Lax}}=-i\partial_{x}-u\Pi_{+}\overline{u},\quad\text{and}\quad\mathcal{P}_{\textnormal{Lax}}=i\partial_{xx}+2uD_{+}\overline{u}.\label{eq:LaxPair}
\end{align}
There is a similar Lax pair in the full intermediate NLS \cite{PelinovskyGrimshaw1995IntermediateDefocusingCMDNLSLaxPair}.
Another feature of integrability on the Hardy space is an explicit
formula \cite{KLV2025CAMS}
given by a holomorphic function on the upper half-plane. This is followed
by the works of Gérard and collaborators in relevant models \cite{GerardGrellier2015ExplictFormulaCubicSzego1,Gerard2023BOequexplicitFormula,GerardPushnitski2024CMP}.\footnote{In fact, \eqref{eq:LaxPair-evol} holds true regardless of chirality. See \cite[Proposition 3.5]{KimKimKwon2024arxiv}. However, the chirality is required for the explicit formula.} 
The solution given by the explicit formula might become singular and may not be a \emph{strong} solution (say, in $C_{t}H_{+}^{s}$) to \eqref{CMdnls}. Hence, the formula says that if a strong solution exists on a time interval, then the solution should follow the explicit formula on that interval.

We briefly recall symmetries and conservation laws. \eqref{CMdnls}
enjoys time and space translation, and phase rotation symmetries.
They are associated to the conservation laws of energy, mass, and
momentum: 
\begin{gather*}
	\widetilde{E}(u)=\frac{1}{2}\int_{\mathbb{R}}\left|\partial_{x}u-i\Pi_{+}(|u|^{2})u\right|^{2}dx,\, \text{(Energy)}\\
	M(u)=\int_{\mathbb{R}}|u|^{2}dx,\,\text{(Mass)}\quad\widetilde{P}(u)=\Re\int_{\mathbb{R}}(\overline{u}Du-\frac{1}{2}|u|^{4})dx.\,\text{(Momentum)}
\end{gather*}
Moreover, it has \textit{Galilean invariance} 
\begin{align*}
	u(t,x)\mapsto[\textnormal{Gal}_{c}u](t,x)\coloneqq e^{icx-ic^{2}t}u(t,x-2ct),\quad(c\in\mathbb{R}).
\end{align*}
Of particular importance are \textit{$L^{2}$-scaling symmetry} 
\begin{align*}
	u(t,x)\mapsto\lambda^{-\frac{1}{2}}u(\lambda^{-2}t,\lambda^{-1}x),\quad(\lambda>0),
\end{align*}
and\textit{ pseudo-conformal symmetry} 
\begin{equation}
	u(t,x)\mapsto[\calC u](t,x)\coloneqq\frac{1}{|t|^{1/2}}e^{i\frac{x^{2}}{4t}}u\left(-\frac{1}{t},\frac{x}{|t|}\right).\label{eq:pseudo-conf-transf}
\end{equation}
Associated identities to scaling and pseudo-conformal symmetries are
the virial identities \eqref{eq:Virial identity}. We note that those
symmetries are shared with the mass-critical NLS. But, we remark that
$\td E(u)$ is a complete square of first-order form. Regarding the
chirality, the Galilean invariance preserves the chirality only if
$c\geq0$. The pseudo-conformal symmetry is valid for $H^{1,1}$-solutions
but does not preserve the chirality.

It is well-known that stationary or static solutions play a pivotal
role in the dynamics. Here, a solution to \eqref{CMdnls} is static
if and only if it has zero energy. In \cite{GerardLenzmann2022},
the authors showed that $\calR(x)$, called the \emph{ground state}
or \emph{soliton}, 
\[
\mathcal{R}(x)=\frac{\sqrt{2}}{x+i}\in H_{+}^{1}(\mathbb{R})\quad\text{with}\quad M(\mathcal{R})=2\pi\quad\text{and}\quad\widetilde{E}(\mathcal{R})=0,
\]
is the unique zero energy solution (and thus static) up to scaling,
phase rotation, and translation symmetries. Note that this $\calR(x)$
is a chiral solution. More generally, any nonzero $H^{1}(\bbR)$ traveling
wave solutions (i.e., solutions of the form $u(t,x)=e^{i\omega t}\calR_{c,\omega}(x-2ct)$
for some $\omega,c\in\bbR$) are given by $\textnormal{Gal}_{c}\calR$
up to scaling, phase rotation, and translation symmetries \cite{GerardLenzmann2022}.
Applying the pseudo-conformal transform \eqref{eq:pseudo-conf-transf}
to $\calR$, one obtains an explicit finite-time blow-up solution:
\[
S(t,x)\coloneqq\frac{1}{t^{1/2}}e^{ix^{2}/4t}\mathcal{R}\left(\frac{x}{t}\right)\in L^{2}(\bbR),\qquad\forall t>0.
\]
It is important to note that $S(t)$ is neither of finite energy ($S(t)\notin H^{1}(\bbR)$),
nor \emph{chiral} ($S(t)\notin L_{+}^{2}(\bbR)$).

Let us briefly discuss previously known results on \eqref{CMdnls}.
De Moura and Pilod \cite{MouraPilod2010CMDNLSLocal} proved the local
well-posedness in $H^{s}(\mathbb{R})$ for all $s>\frac{1}{2}$. It
is also locally well-posed in $H_{+}^{s}(\R)$ for $s>\frac{1}{2}$
\cite{GerardLenzmann2022}. The ground state $\mathcal{R}$ is a threshold
for global regularity, i.e. global existence of strong solutions.
To this regard, Lax pairs provide significant information for the
subthreshold $M(u)<M(\mathcal{R})=2\pi$. If $M(u)<M(\mathcal{R})$
and $u_0\in H_{+}^{1}(\mathbb{R})$, then the solution is global and moreover,
$\sup_{t\in\R}\norm{u(t)}_{H_{+}^{k}(\R)}\lesssim_{k}\norm{u_{0}}_{H_{+}^{k}(\R)}$
for all $k\in\mathbb{N}_{\geq1}$ \cite{GerardLenzmann2022}. Subsequently,
the local well-posedness on $M(u)<M(\mathcal{R})$ was improved to
$L_{+}^{2}(\mathbb{R})$ by Killip, Laurens, and Vi\c{s}an \cite{KLV2025CAMS}.
At the threshold ($M(u_{0})=M(\mathcal{R})$), by adapting \cite{Merle1993Duke},
$H^{1}$-solutions are global, as $S(t)\notin H^{1}$ \cite{GerardLenzmann2022}.
The dynamics above the threshold ($M(u)>M(\mathcal{R}))$ were also
studied. Gérard and Lenzmann \cite{GerardLenzmann2022} employed the
Lax pair structure to construct $N$-soliton solutions of the form
\begin{align*}
	u(t,x)=\sum_{k=1}^{N}\frac{a_{k}(t)}{x-z_{k}(t)}\in H_{+}^{1}(\mathbb{R}),\quad M(u_{0})=2\pi N,\quad\forall N\geq2,
\end{align*}
where the residues $a_{k}(t)\in\mathbb{C}$ and the pairwise distinct
poles $z_{k}(t)\in\mathbb{C}_{-}=\{z\in\bbC:\Im(z)<0\}$ for $1\leq k\leq N$
solve a complexified version of the classical Calogero--Moser system.
These $N$-soliton solutions blow up in infinite-time with $\|u(t)\|_{H^{s}}\sim|t|^{2s}$
as $|t|\to\infty$ for any $s>0$. Hogan and Kowalski \cite{HoganKowalski2024PAA},
using the explicit formula, showed the existence of possibly infinite
time blow-up solutions with mass arbitrarily close to the threshold
$M(\calR)$. The first finite-time blow-up solutions were constructed
by the authors and K. Kim \cite{KimKimKwon2024arxiv}, arising from
smooth chiral data. This result says that $\mathcal{R}$ is a threshold for global existence of strong solutions. Moreover, it is remarkable that although \eqref{CMdnls}
is completely integrable, it admits finite-time blow-up solutions.
Additionally, the zero dispersion limit of \eqref{CMdnls} was investigated
by Badreddine \cite{Badreddine2024SIMA}. We also refer
to \cite{Matsuno2023,Badreddine2024PAA,Badreddine2023CMDNLSTorusDefocusingTravelingwavearxiv,BCD2024arxivNumerical}
for the periodic model.

Our main theorem concerns the global dynamics beyond the threshold,
so-called \emph{soliton resolution}. We show soliton resolution
in a fully general setting, \emph{without radial symmetry and size
	constraints. }We use notation for modulated functions $[f]_{\la,\ga,x}$
in \eqref{eq:modulation notation}. 
\begin{thm}[Soliton resolution for \eqref{CMdnls}]
	\label{thm:Soliton resolution} Let $u\in C_{t}H^{1}([0,T)\times\R)$
	be a solution to \eqref{CMdnls} with initial data $u_{0}\in H^{1}$,
	where $[0,T)$ is its maximal forward interval of existence.
	
	\noindent \textbf{(Finite-time blow-up solutions)} If $T<\infty$,
	there exist an integer $N\in\bbN$ with $1\leq N\leq\frac{M(u_{0})}{M(\calR)}$,
	a time $0<\tau<T$, modulation parameters $(\lambda_{j}(t),\gamma_{j}(t),x_{j}(t)):C^{1}([\tau,T))\to\bbR^{+}\times\bbR/2\pi\bbZ\times\bbR$
	for $j=1,2,\cdots,N$, and asymptotic profile $z^{\ast}\in L^{2}$
	so that $u(t)$ admits the decomposition 
	\begin{align}
		u(t)-\sum_{j=1}^{N}[\calR]_{\lambda_{j}(t),\gamma_{j}(t),x_{j}(t)}\to z^{\ast}\text{ in }L^{2}\text{ as }t\to T,\label{eq:soliton resol ungauge}
	\end{align}
	and satisfies the following properties: 
	\begin{itemize}
		\item (Asymptotic orthogonality, no bubble tree) For all $1\leq i\neq j\leq N$,
		\begin{align}
			\left|\frac{x_{i}(t)-x_{j}(t)}{\lambda_{i}(t)}\right|\to\infty\quad\text{as}\quad t\to T.\label{eq:asymptotic orthogonality}
		\end{align}
		\item (Convergence of translation parameters) $\lim_{t\to T}x_{j}(t)\eqqcolon x_{j}(T)$
		exist for all $1\leq j\leq N$. 
		\item (Further information about $z^{\ast}$) We have $M(z^{\ast})=M(u_{0})-N\cdot M(\mathcal{R})$.
		In addition, if $u_{0}\in H^{1,1}$, then $xz^{\ast}\in L^{2}$. 
		\item (Bound on the blow-up speed) We have $\|u(t)\|_{\dot{H}^{1}}\sim\max_{j}(\la_{j}(t)^{-1})$,
		and 
		\begin{align*}
			\lambda_{j}(t)\lesssim T-t\text{ as }t\to T & \text{ for all }1\le j\le N.
		\end{align*}
		\item (Chiral solution) If $u_{0}\in L_{+}^{2}$, then each component in
		the decomposition is chiral, i.e. $z^{\ast}\in L_{+}^{2}$. 
	\end{itemize}
	\textbf{(Global solutions)} If $T=\infty$ and $u\in H^{1,1}$, then
	either $u(t)$ scatters forward in time, 
	\begin{align*}
		\lim_{t\to\infty}\|u(t)-e^{it\partial_{xx}}u^{\ast}\|_{L^{2}}=0,
	\end{align*}
	or there exist an integer $N\in\bbN$ with $1\leq N\leq\frac{M(u_{0})}{M(\calR)}$,
	a time $\tau>0$, modulation parameters $(\lambda_{j}(t),\gamma_{j}(t),c_{j}(t)):C^{1}([\tau,\infty))\to\bbR^{+}\times\bbR/2\pi\bbZ\times\bbR$
	for $j=1,2,\cdots,N$, and $u^{\ast}\in L^{2}$ so that $u(t)$ admits
	the decomposition 
	\begin{align*}
		u(t)-\sum_{j=1}^{N}\textnormal{Gal}_{c_{j}(t)}([\calR]_{\lambda_{j}(t),\gamma_{j}(t),0})-e^{it\partial_{xx}}u^{\ast}\to0\text{ in }L^{2}\text{ as }t\to\infty,
	\end{align*}
	and satisfies the following properties: 
	\begin{itemize}
		\item (Asymptotic orthogonality, no bubble tree) For all $1\leq i\neq j\leq N$,
		denoting the translation parameters $x_{j}(t)\coloneqq2tc_{j}(t)$
		of $\textnormal{Gal}_{c_{j}(t)}[\calR]$, we have 
		\begin{align*}
			\left|\frac{x_{i}(t)-x_{j}(t)}{\lambda_{i}(t)}\right|\to\infty\quad\text{as}\quad t\to\infty.
		\end{align*}
		\item (Convergence of velocity) $\lim_{t\to\infty}c_{j}(t)\eqqcolon c_{j}(\infty)$
		exist for all $1\leq j\leq N$. 
		\item (Further information about $u^{\ast}$) We have $M(u^{\ast})=M(u_{0})-N\cdot M(\calR)$.
		We also have a further regularity, $\partial_{x}u^{\ast}\in L^{2}$. 
		\item (Bound on the scale) We have $\|u(t)\|_{\dot{H}^{1}}\sim\max_{j}(\la_{j}(t)^{-1})$,
		and 
		\begin{align*}
			\lambda_{j}(t)\lesssim 1 \quad \text{for all}\quad 1\le j\le N.
		\end{align*}
		\item (Chiral solution) If $u_{0}\in L_{+}^{2}$, then we can choose each
		component in the decomposition to be chiral. 
	\end{itemize}
\end{thm}

Soliton resolution is widely believed to occur in many dispersive
equations. It suggests that any \emph{generic} global-in-time solution
asymptotically decouples into a sum of solitons (or similar solutions
such as breathers) and a radiation term that goes to zero in some
sense. For blow-up solutions, in various models, it is believed that
each blow-up profile is a sum of modulated solitons. This type of
results were conjectured for KdV equation in \cite{FPU1955numerical1,ZabuskyKruskal1965numerical2}
from the numerical simulations. The rigorous proof of soliton resolution
was first demonstrated in various integrable PDEs via the inverse
scattering method. To refer to just a few, see \cite{EckhausSchuur1983KdV1,Eckhaus1986KdV2}
for KdV, \cite{Schuur1986mKdV1} for mKdV, and \cite{ZakharovShabat19721dNLS1,SegurAblowitz19761dNLS2,Segur19761dNLS3,Novokvsenov19801dNLS4,BJM20181dNLS5}
for 1-dimensional cubic NLS. Also refer to \cite{Jenkinsetal2018CMPDNLS,Jenkinsatel2020QAMDNLS}
for Derivative NLS.

For non-integrable dispersive and wave equations, soliton resolution
has been studied in several models. In wave equations, such as the
energy-critical nonlinear wave equation (in various dimensions) and
energy-critical equivariant wave maps \cite{DuyckaertsKenigMerle2013Camb,Cote2015CPAMsolitonResol,JiaKenig2017AJMwaveSolResol,DJKM2017GaFASolresolSequence,DuyckaertsKenigMerle2023Acta,DKMM2022CMP,jendrejLawrie2025JAMS,CDKM2024vietnam,JendrejLawrie2023AnnPDESolResol}.
For damped Klein-Gordon equations, soliton resolution for global solutions
was established in \cite{BurqRaugelSchlag2017ASENS,CMY2021ARMA,Ishizuka2025NonAnal}.
For Schrödinger-type equations, soliton resolution was established
for the equivariant self-dual Chern--Simon--Schrödinger equation
(CSS) in \cite{KimKwonOh2025AJM}. In the context of parabolic
equations, soliton resolution has been studied by several authors
for the harmonic map heat flow \cite{JendrejLawrie2023CVPDEsolResolHMHF,JLS2025ForumPi}
(or references therein) and the energy-critical semilinear heat equation
\cite{Aryan2024solResolHeat}. Among others, the authors in \cite{JLS2025ForumPi}
proved a version of continuous-in-time soliton resolution without
symmetry. As seen from the list, most results in non-integrable models
were achieved under symmetry constraints, excluding moving solitons.

\vspace{5bp}
\noindent \textbf{Gauge transform.}

\eqref{CMdnls} has an extra structure, so-called \emph{gauge transform}.
This property is shared with DNLS. Define the gauge transform 
\begin{align*}
	v(t,x)=\mathcal{G}(u)(t,x):=-u(t,x)e^{-\frac{i}{2}\int_{-\infty}^{x}|u(t,y)|^{2}dy}.
\end{align*}
Then, new variable $v(t,x)$ solves 
\begin{align}
	\begin{cases}
		i\partial_{t}v+\partial_{xx}v+|D|(|v|^{2})v-\frac{1}{4}|v|^{4}v=0,\quad(t,x)\in\mathbb{R}\times\mathbb{R}\\
		v(0)=v_{0}.
	\end{cases} & \tag{\ensuremath{\mathcal{G}}-CM}\label{CMdnls-gauged}
\end{align}
\eqref{CMdnls-gauged} is a gauge transformed Calogero-Moser derivative
NLS. All symmetries are transferred accordingly. The conservation
laws of energy, mass, and momentum are given by 
\begin{align}
	\begin{gathered}E(v)=\frac{1}{2}\int_{\mathbb{R}}\left|\partial_{x}v+\frac{1}{2}\mathcal{H}(|v|^{2})v\right|^{2}dx,\\
		M(v)=\int_{\mathbb{R}}|v|^{2}dx,\quad P(v)=\int_{\mathbb{R}}\Im(\overline{v}\partial_{x}v)dx,
	\end{gathered}
	\label{eq:intro energy gauge}
\end{align}
where $\mathcal{H}$ is the Hilbert transform. The virial identities are 
\begin{align}
	\begin{split}\frac{d}{dt}\int_{\bbR}|x|^{2}|v(t,x)|^{2}dx & =4\int_{\bbR}x\cdot\Im(\ol v\rd_{x}v)\,dx,\\
		\frac{d}{dt}\int_{\bbR}x\cdot\Im(\ol v\rd_{x}v)\,dx & =4E(v).
	\end{split}
	\label{eq:Virial identity}
\end{align}
It is noteworthy that \eqref{CMdnls-gauged} admits the Hamiltonian
formulation 
\begin{align*}
	\partial_{x}v=-i\nabla E(v),
\end{align*}
where $\nabla E(v)$ is a functional derivative with respect to $(f,g)_{r}=\Re\int f\ol g$.
In other words, $-i\nabla E(v)$ is a symplectic derivative with respect
to the standard symplectic form $\omega(f,g)=\Im\int f\ol g$. Moreover,
as $E(v)$ is a complete square, \eqref{CMdnls-gauged} is a self-dual
Hamiltonian equation. See more details in \cite{KimKimKwon2024arxiv}.
The static solution $\calR$ of \eqref{CMdnls} is transformed as
a static solution to \eqref{CMdnls-gauged} 
\begin{align*}
	Q(x)\coloneqq-\mathcal{G}(\mathcal{R})(x)=\frac{\sqrt{2}}{\sqrt{1+x^{2}}}\in H^{1}(\mathbb{R}),\quad M(Q)=2\pi,\quad E(Q)=0.
\end{align*}
Note that we chose the minus sign in the transform $v=-\mathcal{G}(u)$
to ensure that $Q$ is positive. $\mathcal{R}$ is not real-valued,
and $\Re\mathcal{R}(x)$ and $\Im\mathcal{R}(x)$ exhibit different
decays. However $Q$ is positive real-valued, which is a technical
benefit of working with \eqref{CMdnls-gauged}. In the main body of
analysis, we will prove soliton resolution for \eqref{CMdnls-gauged},
and then, using the gauge transform and its inverse, we will obtain
Theorem~\ref{thm:Soliton resolution}. 
\begin{thm}[Soliton resolution for \eqref{CMdnls-gauged}]
	\label{thm:Soliton resolution gauged} Let $v\in C_{t}H^{1}([0,T)\times\R)$
	be a solution to \eqref{CMdnls-gauged} with initial data $v_{0}\in H^{1}$,
	where $[0,T)$ is its maximal forward interval of existence.
	
	\noindent \textbf{(Finite-time blow-up solutions)} If $T<\infty$,
	there exist an integer $N\in\bbN$ with $1\leq N\leq\frac{M(v_{0})}{M(Q)}$,
	a time $0<\tau<T$, modulation parameters $(\lambda_{j}(t),\gamma_{j}(t),\allowbreak x_{j}(t)):C^{1}([\tau,T))\to\bbR^{+}\times\bbR/2\pi\bbZ\times\bbR$
	for $j=1,2,\cdots,N$, and asymptotic profile $z^{\ast}\in L^{2}$
	so that $v(t)$ admits the decomposition 
	\begin{align*}
		v(t)-\sum_{j=1}^{N}[Q]_{\lambda_{j}(t),\gamma_{j}(t),x_{j}(t)}\to z^{\ast}\text{ in }L^{2}\text{ as }t\to T,
	\end{align*}
	and satisfies the following properties: 
	\begin{itemize}
		\item (Asymptotic orthogonality, no bubble tree) For all $1\leq i\neq j\leq N$,
		\begin{align}
			\left|\frac{x_{i}(t)-x_{j}(t)}{\lambda_{i}(t)}\right|\to\infty\quad\text{as}\quad t\to T.\label{eq:asymptotic orthogonality gauged}
		\end{align}
		\item (Convergence of translation parameters) $\lim_{t\to T}x_{j}(t)\eqqcolon x_{j}(T)$
		exist for all $1\leq j \leq N$. 
		\item (Further information about $z^{\ast}$) We have $M(z^{\ast})=M(v_{0})-N\cdot M(Q)$.
		We also have a further regularity, $\partial_{x}z^{\ast}\in L^{2}$.
		In addition, if $v\in H^{1,1}$, then we have $xz^{\ast}\in L^{2}$. 
		\item (Bound on the blow-up speed) We have $\|v(t)\|_{\dot{H}^{1}}\sim\max_{j}(\la_{j}(t)^{-1})$,
		and 
		\begin{align*}
			\lambda_{j}(t)\lesssim T-t\text{ as }t\to T & \text{ for all }1\le j\le N.
		\end{align*}
	\end{itemize}
	\textbf{(Global solutions)} If $T=\infty$ and $v\in H^{1,1}$, then
	either $v(t)$ scatters forward in time, 
	\begin{align*}
		\lim_{t\to\infty}\|v(t)-e^{it\partial_{xx}}v^{\ast}\|_{L^{2}}=0,
	\end{align*}
	or there exist an integer $N\in\bbN$ with $1\leq N\leq\frac{M(v_{0})}{M(Q)}$,
	a time $\tau>0$, modulation parameters $(\lambda_{j}(t),\gamma_{j}(t),c_{j}(t)):C^{1}([\tau,\infty))\to\bbR^{+}\times\bbR/2\pi\bbZ\times\bbR$
	for $j=1,2,\cdots,N$, and $v^{\ast}\in L^{2}$ so that $v(t)$ admits
	the decomposition 
	\begin{align}
		v(t)-\sum_{j=1}^{N}\textnormal{Gal}_{c_{j}(t)}([Q]_{\lambda_{j}(t),\gamma_{j}(t),0})-e^{it\partial_{xx}}v^{\ast}\to0\text{ in }L^{2}\text{ as }t\to\infty,\label{eq:soliton configu gauged}
	\end{align}
	and satisfies the following properties: 
	\begin{itemize}
		\item (Asymptotic orthogonality, no bubble tree) For all $1\leq i\neq j\leq N$,
		denoting the translation parameters $x_{j}(t)\coloneqq2tc_{j}(t)$
		of $\textnormal{Gal}_{c_{j}(t)}[Q]$, we have 
		\begin{align*}
			\left|\frac{x_{i}(t)-x_{j}(t)}{\lambda_{i}(t)}\right|\to\infty\quad\text{as}\quad t\to\infty.
		\end{align*}
		\item (Convergence of velocity) $\lim_{t\to\infty}c_{j}(t)\eqqcolon c_{j}(\infty)$
		exist for all $1\leq j\leq N$. 
		\item (Further information about $v^{\ast}$) We have $M(v^{\ast})=M(v_{0})-N\cdot M(Q)$.
		We also have a further regularity, $\partial_{x}v^{\ast},xv^{\ast}\in L^{2}$. 
		\item (Bound on the scale) We have $\|v(t)\|_{\dot{H}^{1}}\sim\max_{j}(\la_{j}(t)^{-1})$,
		and 
		\begin{align*}
			\lambda_{j}(t)\lesssim 1 \quad\text{for all}\quad 1\le j\le N.
		\end{align*}
	\end{itemize}
\end{thm}

\noindent \textbf{Comments on Theorem~\ref{thm:Soliton resolution} and~\ref{thm:Soliton resolution gauged}.}

\emph{1. Novelty and Method}. Our proof does not rely on
the complete integrability. To our knowledge, this is the first proof of soliton resolution for Schrödinger-type equations without radial symmetry or size constraints that does not use complete integrability techniques. We bypass time-sequential soliton
resolution and directly prove continuous-in-time soliton resolution.
The nonnegativity of energy is crucial. We believe our argument is
applicable to other models with nonnegative energy, such as, wave
or Schrödinger map, Chern-Simons-Schrödigner, and so on. A similar
idea was used in NLS under threshold condition by Merle \cite{Merle1993Duke}
or Dodson \cite{Dodson2024AnalPDEd1}. 

\emph{2. No bubble tree}. \eqref{eq:asymptotic orthogonality} and
\eqref{eq:asymptotic orthogonality gauged} indicate that there is
no bubble tree. i.e. any two bubbles maintain a distance larger than
the scales of both. We believe this is natural. If a bubble tree existed,
there would be a discontinuity in the soliton configuration in \eqref{CMdnls}
along with the gauge transform. See more detail in Remark~\ref{rem:phase correction}.
As a similar result, in 1D harmonic map heat flow, there is no finite
time bubble tree \cite{Hout2003JDE}. So far, finite-time bubble trees
are not yet constructed in any model, while there are several results
of infinite time bubble tree construction \cite{Topping2000,delPinoMussoWei2021AnalPDE,Jendrej2017AnalPDE,Jendrej2019AJM,JendrejLawrie2018Invent}.

\emph{3. Global solutions}. We prove finite-time blow-up cases first,
and then take the pseudo-conformal transform to obtain results for
global solutions. This is why we need to assume $u(t)\in H^{1,1}$
for global solutions, while $u(t)\in H^{1}$ suffices for finite-time
blow-up solutions. After taking the pseudo-conformal transform, the
translation parameter $x_{j}(t)$ becomes the velocity of the Galilean
boost, $\text{Gal}_{c_{j}(t)}$, and the scaling parameter becomes
$\la(t)\to\mathring{\la}(t)\coloneqq t\la(-t^{-1})\lesssim1$. This
results a multi-soliton configuration with moving solitons at constant
velocities.

For $H^{1,1}$-solutions, thanks to the pseudo-conformal transform,
the linear scattering of radiation part is easily obtained. In particular,
this adresses the subthreshold problem, i.e. when $M(u)<M(\mathcal{R})$
and $u\in H^{1,1}$, then $u(t)$ has to scatter. A remaining question
is whether $H^{1}$-global solutions may exhibit different dynamics
other than $H^{1,1}$-solutions. At current status, even for small
solution in $H^{1}$ or $H^{1}_{+}$, the (linear or modified) scattering is not known.
In view of results of other cubic equations in 1D, it is unclear whether
the linear scattering occurs for $H^{1}$-solutions.

We recall that in the $H^{k}_{+}$ setting, the global behavior of subthreshold solutions is known: if $M(u)<M(\mathcal{R})$ and $u_0 \in H^{k}_{+}$, then the solution is global and satisfies uniform-in-time bounds in $H^{k}_{+}$ for all $k\geq 1$ \cite{GerardLenzmann2022}. Moreover, under the subthreshold condition $M(u)<M(\mathcal{R})$ \cite{KLV2025CAMS} established the explicit formula and proved local well-posedness in $L^2_+$ (i.e., $k=0$). Nevertheless, scattering in this setting remains an open problem.

A multi-soliton example constructed by Gérard--Lenzmann \cite{GerardLenzmann2022}
does not belong to $H^{1,1}$ due to the slow decay of solitons. Still,
their examples meet our criteria in Theorem~\ref{thm:Soliton resolution}.
On the other hand, even if $v(t)\in H^{1,1}$, each soliton component does not belong to $H^{1,1}$ in the multi-soliton configuration \eqref{eq:soliton configu gauged}. Consequently, the radiation $\eps(t)\coloneqq v(t)-\sum_{j=1}^{N}\textnormal{Gal}_{c_{j}(t)}([Q]_{\lambda_{j}(t),\gamma_{j}(t),0})$ does not belong to $H^{1,1}$.
One might find this decomposition unsatisfactory. If one wants all
components to belong to $H^{1,1}$, then one can simply truncate tails
of solitons. More precisely, one can replace $[Q]_{\lambda_{j}(t),\gamma_{j}(t),0}$
with ($[Q]_{\lambda_{j}(t),\gamma_{j}(t),0})\chi_{t^{\delta}}$ such
that $\la_{j}(t)\ll t^{\delta}$.

\vspace{5bp}
\noindent \textbf{Outline of the proof.} 

As mentioned above, we prove Theorem~\ref{thm:Soliton resolution gauged}
for the finite-time blow-up solutions and then use the pseudo-conformal
transform to obtain result for global solutions case. And then we
use the gauge transform $\mathcal{G}$ to obtain Theorem~\ref{thm:Soliton resolution}.
Thus, in main analysis we consider a finite-time blow-up solution
to \eqref{CMdnls-gauged} $v(t)\in H^{1}$. 

The first ingredient is the variational characterization
of $Q$. In fact, $Q$ is the unique zero energy solution up to symmetries,
and thus also a static solution. Furthermore, this also tells us the
proximity of a small energy function to $Q$, as stated in Lemma~\ref{lem:Proximity}.
More specifically, if $\sqrt{E(v)}<\delta\norm v_{\dot{H}^{1}}$,
then $v=[Q+\wt{\eps}]_{\la,\ga,x}$ with $\norm{\wt{\eps}}_{\dot{H}^{1}}<\eta$.
This is due to the nonnegativity of the energy and achieved near blow-up
time. Also note that this is a stronger variational feature than that
of mass-critical NLS. This allows us to extract a soliton from $v(t)$. 

The second ingredient, one of our main novelty, is the \emph{energy
	bubbling} estimate. After the first decomposition with orthogonality
conditions on $\wt{\eps}$, we have the following improved estimate
(Lemma~\ref{lem:Nonlinear Coercivity}): 
\begin{align}
	E(Q+\wt{\eps})\gtrsim_{\|\wt{\eps}\|_{L^{2}}}\|\partial_{x}(\chi_{R}\wt{\eps})\|_{L^{2}}^{2}+\|Q\wt{\eps}\|_{L^{2}}^{2}+E((1-\chi_{R})\wt{\eps}),\label{eq:intro energy bubbling}
\end{align}
where $R>1$ is a large parameter depending on $\|\wt{\eps}\|_{L^{2}}$
and $\chi_{R}$ is a smooth cut off on $[-R,R]$. \eqref{eq:intro energy bubbling}
is motivated from a \emph{nonlinear coercivity estimate} for (CSS)
\cite{KimKwonOh2025AJM}. In (CSS), due to a special property of
its nonlinearity they have a stronger estimate on the outer radiation.
Then authors in \cite{KimKwonOh2025AJM} obtain soliton resolution
consisting of a single soliton. \eqref{eq:intro energy bubbling}
benefits from the nonnegativity of energy. More importantly, we will
take advantage of the energy conservation, which is located above
the critical scaling.

For blow-up solutions, we recall that $v=[Q+\wt{\eps}]_{\la,\ga,x}$, so that $E(Q+\wt{\eps})=\la^{2}E(v)$ with $\la(t)\to0$. Hence, the right-hand side of \eqref{eq:intro energy bubbling} is not just small but goes to zero. In particular, we have a quantitative
control of the inner part of radiation $\chi_{R}\wt{\eps}$. However,
in general, the outer part of radiation $\|(1-\chi_{R})\wt{\eps}\|_{\dot{H}^{1}}$
may not go to zero. Also, $M((1-\chi_{R})\wt{\eps})$ can be large.
Instead, we have good control of $E((1-\chi_{R})\wt{\eps})$. Indeed,
we will have a dichotomy: either 
\[
E((1-\chi_{R})\wt{\eps})\gtrsim\|(1-\chi_{R})\wt{\eps}\|_{\dot{H}^{1}}^{2}\quad\text{for all }t<T
\]
or 
\[
E((1-\chi_{R})\wt{\eps})\ll\|(1-\chi_{R})\wt{\eps}\|_{\dot{H}^{1}}^{2}\quad\text{for all }t<T.
\]
In fact, if the former is false, then the latter holds true sequentially
in time. However, we will prove the latter holds for all $t<T$. We
refer to this as the \emph{no-return property},~\ref{state:no-return property}
in Lemma~\ref{lem:Induction}. We prove~\ref{state:no-return property}
by observing the difference in exterior mass between sequences depending
on whether $E((1-\chi_{R})\wt{\eps})\ll\|(1-\chi_{R})\wt{\eps}\|_{\dot{H}^{1}}^{2}$
holds or not. Hence, extracting multi-soliton configuration benefits
from two levels of conservation laws, mass and energy. Thanks to no-return
property, we do not get through time-sequential soliton resolution,
but directly prove continuous-in-time soliton resolution. 

In the former case of the dichotomy, we have a good quantitative
estimate for $\|(1-\chi_{R})\wt{\eps}\|_{\dot{H}^{1}}$ and so we
can verify that $[\wt{\eps}]_{\la,\ga,x}$ converges to an asymptotic
profile. This ends the soliton decomposition. In the latter case,
we can reapply the decomposition to extract the second soliton from
$(1-\chi_{R})\wt{\eps}$ and arrive at the above dichotomy for the
outer radiation part. We can iterate this procedure. Since the mass
drops by $M(Q)$ at each step, it ends at finitely many steps. At
last, we arrive at the multi-soliton configuration: 
\[
v(t)=\sum_{i=1}^{N}[Q]_{\lambda_{i},\gamma_{i},x_{i}}(t)+\eps_{N}(t),\quad\|\eps_{N}(t)\|_{H^{1}}\lesssim1.
\]
The radiation $\eps_{N}(t)$ with a uniform bound $\|\eps_{N}(t)\|_{H^{1}}\lesssim1$
converges to an asymptotic profile $z^{*}$ in $L^{2}$. Along the
way, we also prove behaviors of the modulation parameters, $\la_{j}(t)\lesssim T-t$
and $\lim_{t\to T}x_{j}(t)=x_{j}(T)<\infty$. Finally, it is noteworthy
that we verify a nonradial version of \emph{no bubble-tree} condition,
Proposition~\ref{prop:no bubble tree}: 
\[
\frac{|x_{i}(t)-x_{j}(t)|}{\max(\la_{i}(t),\la_{j}(t))}\to\infty,\qquad\text{as }t\to T.
\]
This is a consequence of $\|Q\wt{\eps}(t)\|_{L^{2}}\lesssim E(Q+\wt{\eps})$
in \eqref{eq:intro energy bubbling}, which provides a quantitative
bound on the interaction between soliton and $\wt{\eps}(t)$. We are
not sure if such no-bubble tree property is a special feature of this
model. One can observe this no-bubble tree condition is consistent
with the gauge transform between \eqref{CMdnls} and \eqref{CMdnls-gauged} (Remark
\ref{rem:phase correction}).

\vspace{5bp}
\noindent \textbf{Organization of the paper.} 

In Section~\ref{sec:preliminaries}, we introduce notation
and preliminaries for our analysis. In Section~\ref{sec:bubbling},
we review a standard variational argument and prove the energy bubbling
estimate, which is a core proposition of our analysis. In Section~\ref{sec:multisoliton},
we prove the multi-soliton configuration via an induction argument.
In this step, we also show the no-return property and confirm that
the multi-soliton configuration holds true continuously in time. In
Section~\ref{sec:prooffinish}, we complete the proofs of main theorems
by applying the pseudo-conformal transform and the gauge transform.

\vspace{5bp}
\noindent \textbf{Acknowledgments.} 

The authors are partially supported by the National Research
Foundation of Korea, RS-2019-NR040050 and RS-2022-NR069873.
The authors appreciate Kihyun Kim for helpful comments on the first
manuscript.

\section{Notation and preliminaries}

\label{sec:preliminaries}

In this section, we collect notations and frequently used formulas.
For quantities $A\in\bbC$ and $B\geq0$, we write $A\lesssim B\Leftrightarrow A=O(B)$
if $|A|\leq CB$ holds for some implicit constant $C$. For $A,B\geq0$,
we write $A\sim B$ if $A\lesssim B$ and $B\lesssim A$. If $C$
depends on some parameters, then we write them as a subscript. For
$A=A(t)\in\bbC,B=B(t)>0$, we write $A=o_{t\to T}(B)$ if for any
$\epsilon>0$ there exists $\delta>0$ such that if $0<T-t<\delta$,
then $|A|\leq\epsilon B$. We also write simply $A\ll B$. When the
quantities are function of time, the estimates are uniform in time,
unless stated otherwise.

Denote a smooth cut-off function by $\chi_{R}=\chi(\frac{x}{R})$
where $\chi\in C^{\infty}$ with $\text{supp}\chi\in[-2,2]$ and $\chi\equiv1$
on $[-1,1]$. It is possible to choose $\chi$ such that $|\chi^{\prime}|^{2}\leq C_{\chi}\chi$
for some $C_{\chi}>0$. We will also frequently use the outer cut-off
\[
\varphi_{R}\coloneqq1-\chi_{R},
\]
and a truncation on the centers of solitons, $\Phi_{R}=\prod_{j=1}^{N}\varphi_{R}(\cdot-x_{j})$.
We also use the sharp cut-off on a set $A$ by ${\textbf 1}_{A}$. We
write the inhomogeneous weight by $\langle x\rangle\coloneqq(1+x^{2})^{\frac{1}{2}}$. We denote $|f|_{-1}$ by
\begin{align*}
	|f|_{-1}(x)\coloneqq \sup_{0\leq j \leq 1} |x|^{-j}|\partial_x^{1-j} f(x)|. 
\end{align*}

The Fourier transform (on $\bbR$) is denoted by 
\begin{align*}
	\mathcal{F}(f)(\xi)=\widehat{f}(\xi)\coloneqq\int_{\mathbb{R}}f(x)e^{-ix\xi}dx,
\end{align*}
with its inverse $\mathcal{F}^{-1}(f)(x)\coloneqq\tfrac{1}{2\pi}\int_{\R}\wt f(\xi)e^{ix\xi}d\xi$.
We denote $|D|$ by the Fourier multiplier operator with symbol $|\xi|$,
that is, $|D|\coloneqq\mathcal{F}^{-1}|\xi|\mathcal{F}$, and then
$|D|=\partial_{x}\mathcal{H}=\mathcal{H}\partial_{x}$ where $\mathcal{H}$
is the Hilbert transform: 
\begin{align*}
	\mathcal{H}f\coloneqq\left(\frac{1}{\pi}\text{p.v.}\frac{1}{x}\right)*f=\mathcal{F}^{-1}(-i\text{sgn}(\xi))\mathcal{F}f.
\end{align*}
We note that $[\calH,i]=0$. We denote by $\Pi_{+}$ the Cauchy--Szeg\H{o}
projection from $L^{2}(\mathbb{R})$ onto the Hardy space $L_{+}^{2}(\mathbb{R})=\{f\in L^{2}(\bbR):\mathrm{supp}\wt f\subset[0,\infty)\}$:
\begin{align*}
	\Pi_{+}f=\mathcal{F}^{-1}\textbf{1}_{\xi>0}\mathcal{F}f.
\end{align*}
Then, we have $\Pi_{+}=\frac{1}{2}(1+i\mathcal{H}).$

We use the real inner product $(\cdot,\cdot)_{r}$ given by $(f,g)_{r}=\int_{\R}\text{Re}(\overline{f}g)dx$.
For given modulation parameters, $(\la,\ga,x)\in\R_{+}\times\R/2\pi\Z\times\R$,
we denote a (inversely) modulated function by 
\begin{equation}
	[f]_{\la,\ga,x}\coloneqq\frac{e^{i\ga}}{\la^{1/2}}f\lr{\frac{\cdot-x}{\la}},\quad[g]_{\la,\ga,x}^{-1}\coloneqq e^{-i\ga}\la^{1/2}g\lr{\la\cdot+x},\label{eq:modulation notation}
\end{equation}
and we have $[[f]_{\la,\ga,x}]_{\la,\ga,x}^{-1}=[[f]_{\la,\ga,x}^{-1}]_{\la,\ga,x}=f$
and $[f]_{\la,\ga,x}^{-1}=[f]_{\la^{-1},-\ga,-\la^{-1}x}$. We also
note that 
\begin{align*}
	[\calH f]_{\la,\ga,x}=\calH[f]_{\la,\ga,x},\quad[\calG(f)]_{\la,\ga,x}=\calG([f]_{\la,\ga,x}).
\end{align*}

We have some algebraic identities with respect to $\calH$, 
\begin{align*}
	\begin{gathered}\mathcal{H}(Q^{2})=yQ^{2},\qquad|D|(Q^{2})=\partial_{y}\calH(Q^{2})=\tfrac{2(1-y^{2})}{(1+y^{2})^{2}}.\end{gathered}
\end{align*}

\begin{lem}[Formulas for Hilbert transform]
	\label{LemmaHilbertUsefulEquation} We have the following: 
\end{lem}

\begin{enumerate}
	\item For $f,g\in H^{\frac{1}{2}+}$, we have 
	\begin{align}
		fg=\mathcal{H}f\cdot\mathcal{H}g-\mathcal{H}(f\cdot\mathcal{H}g+\mathcal{H}f\cdot g)\label{eq:HilbertProductRule}
	\end{align}
	in a pointwise sense. 
	\item For $f\in\langle x\rangle^{-1}L^{2}$, we have 
	\begin{align}
		[x,\mathcal{H}]f(x)=\frac{1}{\pi}\int_{\mathbb{R}}f(y)dy.\label{eq:CommuteHilbert}
	\end{align}
	\item If $f\in H^{1}$ and $\partial_{x}f\in\langle x\rangle^{-1}L^{2}$,
	then we have 
	\begin{align}
		\partial_{x}[x,\mathcal{H}]f=[x,\mathcal{H}]\partial_{x}f=0,\quad\text{i.e.,}\quad[\mathcal{H}\partial_{x},x]f=\mathcal{H}f.\label{eq:CommuteHilbertDerivative}
	\end{align}
\end{enumerate}
Lemma~\ref{LemmaHilbertUsefulEquation} is fairly standard. For the
proof, see \cite{KimKimKwon2024arxiv} Appendix. Next, we state a
useful lemma coming from the nonnegativity of energy. A similar idea
was first used in \cite{Merle1993Duke} for the threshold dynamics
of NLS. 
\begin{lem}[Nonnegativity of energy]
	Let $\psi\in C^{1}\cap L^{\infty}$, $\partial_{x}\psi\in L^{\infty}$
	be real-valued. Then, for a $H^{1}$-solution $v$ to \eqref{CMdnls-gauged},
	we have 
	\begin{align}
		\bigg|\partial_{t}{\int_{\bbR}}\psi|v|^{2}dx\bigg|\lesssim\sqrt{E(v_{0})}\|\partial_{x}\psi\cdot v\|_{L^{2}}.\label{eq:nonnegative energy}
	\end{align}
\end{lem}

\begin{proof}
	We use \eqref{CMdnls-gauged} to compute the left hand side, 
	\begin{align*}
		\partial_{t}{\int_{\bbR}}\psi|v|^{2}dx=2\Re{\int_{\bbR}}\psi\partial_{t}v\ol vdx=-2\Re{\int_{\bbR}}(\partial_{x}\psi)(i\ol v\partial_{x}v)dx.
	\end{align*}
	We use the nonnegativity of energy $E(e^{ia\psi}v)\geq0$ for $a\psi\in\R$
	to estimate 
	\begin{align*}
		E(e^{ia\psi}v) & =\tfrac{1}{2}{\int_{\bbR}}\bigg|ia(\partial_{x}\psi)v+\partial_{x}v+\frac{1}{2}\mathcal{H}(|v|^{2})v\bigg|^{2}dx\\
		& =E(v)+a(i(\partial_{x}\psi)v,\partial_{x}v+\tfrac{1}{2}\mathcal{H}(|v|^{2})v)_{r}+\tfrac{a^{2}}{4}\|(\partial_{x}\psi)v\|_{L^{2}}^{2}\\
		& =E(v)+a(i(\partial_{x}\psi)v,\partial_{x}v)_{r}+\tfrac{a^{2}}{4}\|(\partial_{x}\psi)v\|_{L^{2}}^{2}.
	\end{align*}
	Then, the discriminant inequality gives 
	\begin{align*}
		\bigg|\Re{\int_{\bbR}}(\partial_{x}\psi)(i\ol v\partial_{x}v)dx\bigg|\leq\sqrt{2E_{0}}\|\partial_{x}\psi\cdot v\|_{L^{2}}.
	\end{align*}
\end{proof}

\subsection{Linearization of \texorpdfstring{\eqref{CMdnls-gauged}}{G-CM}}

As mentioned earlier we will proceed all the analysis for the gauged
equation \eqref{CMdnls-gauged}. Here, we review the linearization
of \eqref{CMdnls-gauged} around $Q$. All the material here was investigated
in \cite{KimKimKwon2024arxiv}, where more details can be found.

The form of energy in \eqref{eq:intro energy gauge} represents self-duality.
Introducing the operator 
\begin{align*}
	\mathbf{D}_{v}f\coloneqq\partial_{x}f+\frac{1}{2}\mathcal{H}(|v|^{2})f,
\end{align*}
the energy \eqref{eq:intro energy gauge} can be rewritten as 
\begin{align*}
	E(v)=\frac{1}{2}\int_{\mathbb{R}}\left|\mathbf{D}_{v}v\right|^{2}dx.
\end{align*}
In analogy with \cite{Bogomolnyi1976}, we call the nonlinear operator
$v\mapsto\mathbf{D}_{v}v$ the \textit{Bogomol'nyi operator}, and
the soliton $Q$ solves the Bogomol'nyi equation $\bfD_{Q}Q=0$. Similarly
to $\mathcal{R}$, $Q$ is the unique solution to the Bogomol'nyi
equation up to symmetries. We first linearize the Bogomol'nyi operator
$v\mapsto\mathbf{D}_{v}v$. We write 
\begin{align*}
	\mathbf{D}_{v+\eps}(v+\eps)=\mathbf{D}_{v}v+L_{v}\eps+N_{v}(\eps),
\end{align*}
where the linearized operator $L_{v}$ and the nonlinear part $N_{v}(\eps)$
are given by 
\begin{align*}
	L_{v}\eps & \coloneqq\partial_{x}\eps+\tfrac{1}{2}\mathcal{H}(|v|^{2})\eps+v\mathcal{H}(\text{Re}(\overline{v}\eps)),\\
	N_{v}(\eps) & \coloneqq\eps\mathcal{H}(\Re(\overline{v}\eps))+\tfrac{1}{2}(v+\eps)\mathcal{H}(|\eps|^{2}).
\end{align*}
The $L^{2}$-adjoint operator $L_{v}^{*}$ of $L_{v}$ with respect
to $(\cdot,\cdot)_{r}$ is given by 
\begin{align*}
	L_{v}^{*}\eps & \coloneqq-\partial_{x}\eps+\tfrac{1}{2}\mathcal{H}(|v|^{2})\eps-v\mathcal{H}(\text{Re}(\overline{v}\eps)),
\end{align*}
and using $L_{v}^{\ast}$ one can write $i\nabla E(v)=iL_{v}^{*}\bfD_{v}v$.
Now we linearize \eqref{CMdnls-gauged} as 
\begin{align*}
	iL_{w+\eps}^{*}\bfD_{w+\eps}(w+\eps)=iL_{w}^{*}\bfD_{w}w+i\calL_{w}\eps+R_{w}(\eps)
\end{align*}
where $i\calL_{w}\eps$ is the linear part, and $R_{w}(\eps)$ is
the nonlinear part. If one linearize at $w=Q$, then using $\bfD_{Q}Q=0$
one derive the self-dual factorization 
\begin{align*}
	i\calL_{Q}=iL_{Q}^{*}L_{Q}.
\end{align*}
We will modulate out the kernel directions of $i\mathcal{L}_{Q}$
and obtain the coercivity of the orthogonal part. For this purpose,
we recall kernel information and the coercivity estimate from \cite{KimKimKwon2024arxiv}:
\[
\textnormal{ker}\,L_{Q}=\textnormal{ker}\,\mathcal{L}_{Q}=\textnormal{span}_{\mathbb{R}}\{iQ,\Lambda Q,\partial_{x}Q\},
\]
where $\Lambda$ is the $L^{2}$-scaling generator, $\Lambda f:=\frac{f}{2}+x\partial_{x}f.$

Note that each kernel element is a generator of symmetry, phase rotation,
scaling, and space translation. For the coercivity estimate, due to
a degeneracy of $L_{Q}$, we need to use an adapted Sobolev space
\emph{$\dot{\mathcal{H}}^{1}$ }defined by a norm;\emph{ 
	\[
	\|f\|_{\dot{\mathcal{H}}^{1}}^{2}\coloneqq\|\partial_{x}f\|_{L^{2}}^{2}+\left\Vert \langle x\rangle^{-1}f\right\Vert _{L^{2}}^{2}.
	\]
}Note that $\dot{\mathcal{H}}^{1}\subset\dot{H}^{1}$ with $\norm f_{\dot{H}^{1}}=\norm{\partial_{x}f}_{L^{2}}$.
In this section, we use notation of truncated kernel elements, ${\calZ}_{1}=\Lambda Q\chi$,
$\mathcal{Z}_{2}=iQ\chi$ , $\mathcal{Z}_{3}=\partial_{x}Q\chi$. 
\begin{lem}[Coercivity for $L_{Q}$ on $\dot{\mathcal{H}}^{1}$, \cite{KimKimKwon2024arxiv}]
	\label{lem:Coercivity} We have a coercivity estimate 
	\begin{align*}
		\|v\|_{\dot{\mathcal{H}}^{1}}\sim\|L_{Q}v\|_{L^{2}},\quad\forall v\in\dot{\mathcal{H}}^{1}\cap\{\calZ_{1},\calZ_{2},\calZ_{3}\}^{\perp}.
	\end{align*}
	\emph{For later use, we will denote a truncated version of adapted
		Sobolev space, $\dot{\mathcal{H}}_{R}^{1}$ by a norm 
		\begin{equation}
			\|f\|_{\dot{\mathcal{H}}_{R}^{1}}^{2}\coloneqq\|\partial_{x}(\chi_{R}f)\|_{L^{2}}^{2}+\left\Vert \langle x\rangle^{-1}f\right\Vert _{L^{2}}^{2}.\label{eq:truncated H^1 norm}
		\end{equation}
		Note that the second term $\norm{\jp x^{-1}f}_{L^{2}}=\norm{\tfrac{1}{\sqrt{2}}Qf}_{L^{2}}$
		is not truncated by $R$. It covers global interaction of $Q$ and
		$f$. } 
\end{lem}

\section{Decomposition and energy bubbling}

\label{sec:bubbling}

In this section, we take a preliminary decomposition (bubbling) of
a blow-up solution. If a solution $v(t)$ blows up in finite time,
we can decompose the solution $v(t)$ into a modulated soliton and
radiation of the form $[Q+\wt{\eps}]_{\lambda,\gamma,x}$. This is due
to a variational characterization stating that $[Q]_{\la,\ga,x}$
is the unique nontrivial zero energy solution. A similar argument
was used in the context of the self-dual Chern--Simons--Schrödinger
equation (CSS) in \cite{KimKwonOh2025AJM}. The proof of this part
is fairly standard, so we state them in Lemma~\ref{lem:Proximity}
and~\ref{lem:Decomposition pre} and postpone the proof to the Appendix
\ref{sec:appendix variation}. This is the first step of bubbling
of the solution. However, for soliton resolution, we should be
able to extract a sequence of bubbles from the radiation part. In
this step, we will use a highly nontrivial \emph{energy bubbling estimate}
\eqref{eq:energy bubbling}. The estimate \eqref{eq:energy bubbling}
controls the interaction of $Q$ and the radiation part at each step
and the energy of the exterior part of the radiation. Hence, \eqref{eq:energy bubbling}
is crucial in the multi-bubble decomposition in the next section.
These results are summarized in the following proposition. 
\begin{prop}[One bubbling]
	\label{prop:Decomposition} Let $M\ge M(Q)$ be fixed. There exist
	small parameters $0<\alpha^{*}\ll\eta\ll R^{-1}\ll M^{-1}$ such that
	the following hold: For $v\in H^{1}$ satisfying 
	\[
	\|v\|_{L^{2}}\leq M,\qquad\sqrt{E(v)}\leq\alpha^{*}\|v\|_{\dot{H}^{1}},
	\]
	we have the decomposition as follows; 
	\begin{enumerate}
		\item (Decomposition) There exists unique continuous map $v\in H^{1}\mapsto(\lambda,\gamma,x,\wt{\eps})\in\bbR_{+}\times\bbR/2\pi\bbZ\times\bbR\times H^{1}$
		satisfying 
		\begin{align}
			v=[Q+\wt{\eps}]_{\lambda,\gamma,x},\quad(\wt{\eps},{\calZ}_{j})_{r}=0\text{ for }j=1,2,3,\label{eq:Decomposition and orthogonal}
		\end{align}
		and smallness 
		\begin{align*}
			\|\wt{\eps}\|_{\dot{\calH}^{1}}<\eta.
		\end{align*}
		Moreover, $v\mapsto(\lambda,\gamma,x)$ part is $C^{1}$. 
		\item (Estimate for $\lambda$) We have 
		\begin{align}
			\left|\frac{\|\partial_{x}v\|_{L^{2}}}{\|\partial_{x}Q\|_{L^{2}}}\lambda-1\right|\lesssim\|\partial_{x}\wt{\eps}\|_{L^{2}}.\label{eq:Decompose lambda estimate}
		\end{align}
		\item (Energy bubbling) We have 
		\begin{align}
			\|\wt{\eps}\|_{\dot{\calH}_{R}^{1}}^{2}+E(\varphi_{R}\wt{\eps})\lesssim_{M}\lambda^{2}E(v).\label{eq:Nolinear estimate truncated}
		\end{align}
	\end{enumerate}
\end{prop}

The proof of Proposition~\ref{prop:Decomposition} consists of three lemmas. First, tube stability (Lemma~\ref{lem:Proximity}) is
a consequence of the variational structure of $E(v)$. Next (Lemma~\ref{lem:Decomposition pre}), we decompose into $v=[Q+\wt\eps]_{\lambda,\gamma,x}$ with
appropriate orthogonality conditions on $\eps$, which is an application
of the implicit function theorem. The proofs of Lemma~\ref{lem:Proximity}
and Lemma~\ref{lem:Decomposition pre} are fairly standard and thus
postponed to the Appendix~\ref{sec:appendix variation}. 
\begin{lem}[Tube stability for small energy solutions]
	\label{lem:Proximity} For any $M>0$ and $\delta>0$, there exists
	$\alpha^{*}>0$ such that the following holds. For nonzero $v\in H^{1}$
	with $\|v\|_{L^{2}}\leq M$ and $\sqrt{E(v)}\leq\alpha^{*}\|v\|_{\dot{H}^{1}}$,
	there exist $\gamma\in\bbR/2\pi\bbZ$ and $x\in\bbR$ such that 
	\begin{align*}
		\|[v]_{\lambda,\gamma,x}^{-1}-Q\|_{\dot{\calH}^{1}}<\delta,
	\end{align*}
	where $\lambda\coloneqq\|Q\|_{\dot{H}^{1}}/\|v\|_{\dot{H}^{1}}$. 
\end{lem}

Since $v$ is close to a modulated soliton, the decomposition is possible
in a standard way. We define the soliton tube $\mathcal{T}_{\delta}$
by 
\begin{align*}
	\mathcal{T}_{\delta}\coloneqq\{v\in\dot{\calH}^{1}:\inf_{\lambda^{\prime}\in\bbR_{+},\gamma^{\prime}\in\bbR/2\pi\bbZ,x^{\prime}\in\bbR}\|[v]_{\lambda^{\prime},\gamma^{\prime},x^{\prime}}^{-1}-Q\|_{\dot{\calH}^{1}}<\delta\}.
\end{align*}

\begin{lem}[Decomposition]
	\label{lem:Decomposition pre} For any sufficiently small $\eta>0$,
	there exists $\delta>0$ such that the following hold: For any $v\in\mathcal{T}_{\delta}$,
	we have a unique decomposition, 
	\begin{enumerate}
		\item (Decomposition) There exists unique continuous map $v\mapsto(\lambda,\gamma,x,\eps)\allowbreak\in\bbR_{+}\times\bbR/2\pi\bbZ\times\bbR\times H^{1}$
		satisfying 
		\begin{align*}
			v=[Q+\wt{\eps}]_{\lambda,\gamma,x},\quad(\wt{\eps},\mathcal{Z}_{k})_{r}=0\text{ for }k=1,2,3,
		\end{align*}
		and smallness 
		\begin{align*}
			\|\wt{\eps}\|_{\dot{\calH}^{1}}<\eta.
		\end{align*}
		Moreover, $v\mapsto(\lambda,\gamma,x)$ part is $C^{1}$. 
		\item (Estimate for $\lambda$) We have the estimate \eqref{eq:Decompose lambda estimate}
		for $\lambda$. 
	\end{enumerate}
\end{lem}

The remaining part, \emph{energy bubbling}, is the heart of our analysis.
This provides an improved estimate on the radiation $\wt{\eps}$.
Specifically, it upgrades $\|\wt{\eps}\|_{\dot{\calH}^{1}}<\eta$
to $\|\wt{\eps}\|_{\dot{\calH}_{R}^{1}}\lesssim\la$. This leads to
the vanishing of the inner radiation $\chi_{R}\wt{\eps}$ and $Q\wt{\eps}$,
and subsequently deduces the asymptotic decoupling of the soliton
and $\wt{\eps}$. Observing that $\|\partial_{y}(\varphi_{R}\wt{\eps})\|_{L^{2}}$
is absent in $\|\wt{\eps}\|_{\dot{\calH}_{R}^{1}}$, one can expect
that there could be a nontrivial concentration in $\varphi_{R}\wt{\eps}$. 
\begin{lem}[Energy bubbling]
	\label{lem:Nonlinear Coercivity} Suppose $\|\wt{\eps}\|_{L^{2}}\leq M$
	and orthogonality conditions in \eqref{eq:Decomposition and orthogonal}.
	There exists $R=R(M)$ and $\eta=\eta(M)$ such that if $\|\wt{\eps}\|_{\dot{\calH}^{1}}\leq\eta$,
	then 
	\begin{align}
		E(Q+\wt{\eps})\gtrsim_{M}\|\wt{\eps}\|_{\dot{\calH}_{R}^{1}}^{2}+E(\varphi_{R}\wt{\eps}),\label{eq:energy bubbling}
	\end{align}
	where $\|\cdot\|_{\dot{\calH}_{R}^{1}}$ is defined in \eqref{eq:truncated H^1 norm}. 
\end{lem}

In fact, \eqref{eq:energy bubbling} is a combination of a coercivity
estimate of the inner radiation part $\|\wt{\eps}\|_{\dot{\calH}_{R}^{1}}^{2}$
and the energy control of the outer radiation part $E(\varphi_{R}\wt{\eps})$.
The former is a result of linear coercivity (Lemma~\ref{lem:Coercivity}).
The latter is obtained from the nonnegativity of energy. Although
we do not have a control over $\norm{\partial_{y}(\varphi_{R}\wt{\varepsilon})}_{L^{2}}$,
we can instead maintain the smallness of the energy of $\varphi_{R}\wt{\eps}$.
We believe that this approach is robust enough to be applied to other
models with nonnegative energy. There is a similar (even stronger)
nonlinear coercivity inequality for (CSS) \cite[Lemma 4.4]{KimKwonOh2025AJM},
\[
E(Q+\wt{\eps})\gtrsim_{M}\|\wt{\eps}\|_{\dot{\calH}^{1}}^{2}.
\]
It is worth noting that it also controls the outer radiation $\norm{\partial_{y}(\varphi_{R}\wt{\varepsilon})}_{L^{2}}$.
This is due to a special property of (CSS), so-called \emph{the defocusing
	nature at the exterior of a soliton.} From this specific property,
the authors in \cite{KimKwonOh2025AJM} were able to achieve soliton
resolution with a \emph{single bubble}. However, in \eqref{CMdnls-gauged},
we have no coercivity of $\varphi_{R}\wt{\eps}$. But the energy control
\eqref{eq:energy bubbling} indicates that after extracting a soliton
$Q$, there could be another bubble from $\varphi_{R}\wt{\eps}$ since
$E(\varphi_{R}\wt{\eps})$ is still small.

Having Lemmas~\ref{lem:Proximity},~\ref{lem:Decomposition pre},
and~\ref{lem:Nonlinear Coercivity}, we deduce Proposition~\ref{prop:Decomposition}.
Indeed, for a given $M>0$, there are $R=R(M)$ and $\eta=\eta(M,R)$.
Then, there exist $\delta=\delta(\eta)$ and $\alpha^{*}=\alpha^{*}(M,\delta)$
to satisfy the statement in Proposition~\ref{prop:Decomposition}.
We close this section with the proof of Lemma~\ref{lem:Nonlinear Coercivity}. 
\begin{proof}[Proof of Lemma~\ref{lem:Nonlinear Coercivity}]
	We decompose the domain into inner and outer regions. We will choose
	parameters $R,\eta$ such that $M^{-1}\gg R^{-1}\gg\eta$ . First,
	observe from the definition and assumption that 
	\begin{align*}
		\|\wt{\eps}\|_{\dot{\calH}^{1}}<\eta,\quad\|Q\wt{\eps}\|_{L^{2}}\leq2\|\wt{\eps}\|_{\dot{\calH}_{R^{\prime}}^{1}}\text{ for any }R^{\prime}>0.
	\end{align*}
	
	We claim a preliminary bubbling for suitable $R$ and $\eta.$ 
	\begin{align}
		E(Q+\wt{\eps})\gtrsim_{M}\|\wt{\eps}\|_{\dot{\calH}_{2R^{2}}^{1}}^{2}+E(\varphi_{4R^{2}}\wt{\eps})-O(\|{\textbf 1}_{|y|\in[R,2R]}|\wt{\eps}|_{-1}\|_{L^{2}}^{2}).\label{eq:Nonlinear energy before averaging}
	\end{align}
	We will prove \eqref{eq:Nonlinear energy before averaging} in Step
	1, 2, and 3. Then take an average over $R$ to finish the proof.\textbf{ }
	
	\textbf{Step 1.} In this step, we decompose the domain into $|x|\lesssim R$
	and $|x|\geq R$ and claim that 
	\begin{align}
		\begin{split}E(Q+\wt{\eps})\gtrsim_{\calZ_{j}} & \|\chi_{R}\wt{\eps}\|_{\dot{\calH}^{1}}^{2}-\|{\textbf 1}_{|y|\in[R,2R]}|\wt{\eps}|_{-1}\|_{L^{2}}^{2}\\
			& +O_{M}(1)\cdot o_{R^{-1},\eta\to0}(\|\wt{\eps}\|_{\dot{\calH}_{R}^{1}}^{2})+\eqref{eq:Nolinear energy middle and outer}.
		\end{split}
		\label{eq:Nolinear energy Step 1 goal}
	\end{align}
	In \eqref{eq:Nolinear energy Step 1 goal}, the first term is the
	inner term, the second term is the interaction term, and \eqref{eq:Nolinear energy middle and outer}
	is the intermediate and outer term. We will handle \eqref{eq:Nolinear energy middle and outer}
	in Step 2 and Step 3. We start by linearizing the energy; 
	\begin{align*}
		E(Q+\wt{\eps})=\tfrac{1}{2}\|\bfD_{Q+\wt{\eps}}(Q+\wt{\eps})\|_{L^{2}}^{2}=\tfrac{1}{2}\|L_{Q}\eps+N_{Q}(\wt{\eps})\|_{L^{2}}^{2},
	\end{align*}
	where $N_{Q}(\wt{\eps})$ is the nonlinear term, 
	\begin{align*}
		N_{Q}(\wt{\eps})=\wt{\eps}\mathcal{H}(\Re(Q\wt{\eps}))+\tfrac{1}{2}(Q+\wt{\eps})\mathcal{H}(|\wt{\eps}|^{2}).
	\end{align*}
	The first term of the $N_{Q}(\wt{\eps})$ can be absorbed
	in $O_{M}(1)\cdot o_{R^{-1},\eta\to0}(\|\wt{\eps}\|_{\dot{\calH}_{R}^{1}}^{2})$
	by 
	\begin{align*}
		\|\wt{\eps}\mathcal{H}(\Re(Q\wt{\eps}))\|_{L^{2}}\lesssim\|\wt{\eps}\|_{L^{\infty}}\|\wt{\eps}\|_{\dot{\calH}_{R}^{1}}\lesssim\|\wt{\eps}\|_{L^{2}}^{\frac{1}{2}}\|\wt{\eps}\|_{\dot{\calH}^{1}}^{\frac{1}{2}}\|\wt{\eps}\|_{\dot{\calH}_{R}^{1}}\leq M^{\frac{1}{2}}\eta^{\frac{1}{2}}\|\wt{\eps}\|_{\dot{\calH}_{R}^{1}}.
	\end{align*}
	Thus, it suffices to consider 
	\begin{align*}
		\|L_{Q}\wt{\eps}+\tfrac{1}{2}(Q+\wt{\eps})\mathcal{H}(|\wt{\eps}|^{2})\|_{L^{2}}^{2}.
	\end{align*}
	We decompose into inner and outer parts, 
	\begin{align}
		L_{Q}\wt{\eps}+\tfrac{1}{2}(Q+\wt{\eps})\mathcal{H}(|\wt{\eps}|^{2})= & L_{Q}(\chi_{R}\wt{\eps})+L_{Q}(\varphi_{R}\wt{\eps})\nonumber \\
		& +\tfrac{1}{2}\chi_{R}(Q+\wt{\eps})\mathcal{H}(|\wt{\eps}|^{2})+\tfrac{1}{2}\varphi_{R}(Q+\wt{\eps})\mathcal{H}(|\wt{\eps}|^{2})\nonumber \\
		= & L_{Q}(\chi_{R}\wt{\eps})+(\bfD_{Q}+\tfrac{1}{2}\mathcal{H}(|\wt{\eps}|^{2}))\varphi_{R}\wt{\eps}+\tfrac{1}{2}Q\calH(|\wt{\eps}|^{2})\label{eq:Nolinear energy divide 1}\\
		& +Q\mathcal{H}(\text{Re}(Q\varphi_{R}\wt{\eps}))+\tfrac{1}{2}\chi_{R}\wt{\eps}\mathcal{H}(|\wt{\eps}|^{2}).\label{eq:Nolinear energy divide 2}
	\end{align}
	We first show \eqref{eq:Nolinear energy divide 2} is perturbative.
	Using $\langle y\rangle^{-2}(1+y^{2})=1$ and \eqref{eq:CommuteHilbert},
	we compute the first term of \eqref{eq:Nolinear energy divide 2}
	by 
	\begin{align*}
		Q\mathcal{H}(\text{Re}(Q\varphi_{R}\wt{\eps}))= & Q\mathcal{H}(\text{Re}(\langle y\rangle^{-2}Q\varphi_{R}\wt{\eps}))+Q\mathcal{H}(\text{Re}(y^{2}\langle y\rangle^{-2}Q\varphi_{R}\wt{\eps}))\\
		= & Q\mathcal{H}(\text{Re}(\langle y\rangle^{-2}Q\varphi_{R}\wt{\eps}))+yQ\mathcal{H}(\text{Re}(y\langle y\rangle^{-2}Q\varphi_{R}\wt{\eps}))\\
		& -Q\frac{1}{\pi}{\int_{\bbR}}(\text{Re}(y\langle y\rangle^{-2}Q\varphi_{R}\wt{\eps}))dy.
	\end{align*}
	So, we deduce 
	\begin{align*}
		\|Q\mathcal{H}(\text{Re}(Q\varphi_{R}\wt{\eps}))\|_{L^{2}} & \lesssim\|Q^{2}\varphi_{R}\wt{\eps}\|_{L^{2}}+{\int_{\bbR}}|Q^{2}\varphi_{R}\wt{\eps}|dy\\
		& \lesssim\|Q^{\frac{1}{2}+}\|_{L^{2}}\|Q^{\frac{3}{2}-}\varphi_{R}\wt{\eps}\|_{L^{2}}\lesssim R^{-\frac{1}{2}+}\|\wt{\eps}\|_{\dot{\calH}_{R}^{1}}.
	\end{align*}
	We also have 
	\begin{align*}
		\|\chi_{R}\wt{\eps}\mathcal{H}(|\wt{\eps}|^{2})\|_{L^{2}}\lesssim R\|\wt{\eps}\|_{L^{2}}\|\wt{\eps}\|_{\dot{\calH}^{1}}\|\wt{\eps}\|_{\dot{\calH}_{R}^{1}}\leq RM\eta\|\wt{\eps}\|_{\dot{\calH}_{R}^{1}}.
	\end{align*}
	Hence, we have $\|\eqref{eq:Nolinear energy divide 2}\|_{L^{2}}=O_{M}(1)\cdot o_{R^{-1},\eta\to0}(\|\wt{\eps}\|_{\dot{\calH}_{R}^{1}})$.
	Now we estimate the main part \eqref{eq:Nolinear energy divide 1},
	\begin{align}
		\|\eqref{eq:Nolinear energy divide 1}\|_{L^{2}}^{2}= & \|L_{Q}(\chi_{R}\wt{\eps})+(\bfD_{Q}+\tfrac{1}{2}\mathcal{H}(|\wt{\eps}|^{2}))\varphi_{R}\wt{\eps}+\tfrac{1}{2}Q\calH(|\wt{\eps}|^{2})\|_{L^{2}}^{2}\nonumber \\
		= & \|L_{Q}(\chi_{R}\wt{\eps})\|_{L^{2}}^{2}\label{eq:Nolinear energy inner 0}\\
		& +2\Re{\int_{\bbR}}L_{Q}(\chi_{R}\ol{\wt{\eps}})(\bfD_{Q}+\tfrac{1}{2}\mathcal{H}(|\wt{\eps}|^{2}))\varphi_{R}\wt{\eps}dy\label{eq:Nolinear energy interaction 1}\\
		& +\Re{\int_{\bbR}}L_{Q}(\chi_{R}\ol{\wt{\eps}})Q\calH(|\wt{\eps}|^{2})dy\label{eq:Nolinear energy inter addition}\\
		& +\|(\bfD_{Q}+\tfrac{1}{2}\mathcal{H}(|\wt{\eps}|^{2}))\varphi_{R}\wt{\eps}+\tfrac{1}{2}Q\calH(|\wt{\eps}|^{2})\|_{L^{2}}^{2}.\label{eq:Nolinear energy middle and outer}
	\end{align}
	For the linear term in the inner region, we use the coercivity, Lemma
	\ref{lem:Coercivity}, to estimate \eqref{eq:Nolinear energy inner 0}
	\begin{align}
		\|L_{Q}(\chi_{R}\wt{\eps})\|_{L^{2}}^{2}\gtrsim_{\calZ_{j}}\|\chi_{R}\wt{\eps}\|_{\dot{\calH}^{1}}^{2}.\label{eq:Nolinear energy goal 1}
	\end{align}
	We now estimate the interaction term \eqref{eq:Nolinear energy interaction 1}.
	\begin{align}
		\eqref{eq:Nolinear energy interaction 1}\lesssim & \bigg|{\int_{\bbR}}\bfD_{Q}(\chi_{R}\ol{\wt{\eps}})\bfD_{Q}(\varphi_{R}\wt{\eps})dy\bigg|+\bigg|{\int_{\bbR}}\bfD_{Q}(\chi_{R}\ol{\wt{\eps}})\tfrac{1}{2}\mathcal{H}(|\wt{\eps}|^{2})\varphi_{R}\wt{\eps}dy\bigg|\nonumber \\
		& +\bigg|{\int_{\bbR}}Q\calH\Re(Q\chi_{R}\ol{\wt{\eps}})(\tfrac{y}{1+y^{2}}+\tfrac{1}{2}\mathcal{H}(|\wt{\eps}|^{2}))\varphi_{R}\wt{\eps}dy\bigg|\nonumber \\
		& +\bigg|{\int_{\bbR}}Q\calH\Re(Q\chi_{R}\ol{\wt{\eps}})\partial_{y}\varphi_{R}\wt{\eps}dy\bigg|\nonumber \\
		\lesssim & \|{\textbf 1}_{|y|\in[R,2R]}|\wt{\eps}|_{-1}\|_{L^{2}}^{2}+RM\eta\|\wt{\eps}\|_{\dot{\calH}_{R}^{1}}^{2}+R^{-1}(R^{-1}+M\eta)\|\wt{\eps}\|_{\dot{\calH}_{R}^{1}}^{2}\nonumber \\
		& +\bigg|{\int_{\bbR}}Q\calH\Re(Q\chi_{R}\ol{\wt{\eps}})\partial_{y}(\varphi_{R}\wt{\eps})dy\bigg|,\label{eq:Nolinear energy inter 2}
	\end{align}
	For \eqref{eq:Nolinear energy inter 2}, by integration by parts,
	we have 
	\begin{align*}
		\eqref{eq:Nolinear energy inter 2}=\bigg|{\int_{\bbR}}\partial_{y}(Q\calH\Re(Q\chi_{R}\ol{\wt{\eps}}))\varphi_{R}\wt{\eps}dy\bigg|\lesssim R^{-1}\|\wt{\eps}\|_{\dot{\calH}_{R}^{1}}^{2}.
	\end{align*}
	Hence, we have 
	\begin{align}
		|\eqref{eq:Nolinear energy interaction 1}| & \lesssim\|{\textbf 1}_{|y|\in[R,2R]}|\wt{\eps}|_{-1}\|_{L^{2}}^{2}+RM\eta\|\wt{\eps}\|_{\dot{\calH}_{R}^{1}}^{2}+R^{-1}\|\wt{\eps}\|_{\dot{\calH}_{R}^{1}}^{2}\nonumber \\
		& \lesssim\|{\textbf 1}_{|y|\in[R,2R]}|\wt{\eps}|_{-1}\|_{L^{2}}^{2}+O_{M}(1)\cdot o_{R^{-1},\eta\to0}(\|\wt{\eps}\|_{\dot{\calH}_{R}^{1}}^{2}).\label{eq:Nolinear energy goal 2}
	\end{align}
	Now, we control \eqref{eq:Nolinear energy inter addition}. We compute
	\begin{align}
		\eqref{eq:Nolinear energy inter addition}= & \Re{\int_{\bbR}}\partial_{y}(\chi_{R}\ol{\wt{\eps}})Q\calH(|\wt{\eps}|^{2})dy\label{eq:Nolinear energy inter addition 1}\\
		& +\Re{\int_{\bbR}}\tfrac{y}{1+y^{2}}\chi_{R}\ol{\wt{\eps}}Q\calH(|\wt{\eps}|^{2})dy\label{eq:Nolinear energy inter addition 2}\\
		& +{\int_{\bbR}}\calH\Re(Q\chi_{R}\ol{\wt{\eps}})Q^{2}\calH(|\wt{\eps}|^{2})dy.\label{eq:Nolinear energy inter addition 3}
	\end{align}
	From $Q_y=-\frac{y}{1+y^2}Q$, we have 
	\begin{align*}
		\eqref{eq:Nolinear energy inter addition 1}
		= \eqref{eq:Nolinear energy inter addition 2}-\Re{\int_{\bbR}}\chi_{R}\ol{\wt{\eps}}Q|D|(|\wt{\eps}|^{2})dy,
	\end{align*}
	and this implies 
	\begin{align}
		\eqref{eq:Nolinear energy inter addition}=2\eqref{eq:Nolinear energy inter addition 2}+\eqref{eq:Nolinear energy inter addition 3}-\Re{\int_{\bbR}}\chi_{R}\ol{\wt{\eps}}Q|D|(|\wt{\eps}|^{2})dy.\label{eq:Nolinear energy inter addition 1-1}
	\end{align}
	For the last term of \eqref{eq:Nolinear energy inter addition 1-1},
	using $\langle y\rangle^{-2}(1+y^{2})=1$ and \eqref{eq:CommuteHilbertDerivative},
	we have 
	\begin{align*}
		|D|(|\wt{\eps}|^{2})=\calH(\tfrac{y}{1+y^{2}}|\wt{\eps}|^{2})+yQ|D|(\tfrac{y}{1+y^{2}}|\wt{\eps}|^{2})+|D|(\tfrac{1}{1+y^{2}}|\wt{\eps}|^{2}),
	\end{align*}
	and hence, we estimate 
	\begin{align}
		{\int_{\bbR}}\chi_{R}\ol{\wt{\eps}}Q|D|(|\wt{\eps}|^{2})dy= & {\int_{\bbR}}\chi_{R}\ol{\wt{\eps}}Q\calH(\tfrac{y}{1+y^{2}}|\wt{\eps}|^{2})dy-{\int_{\bbR}}\partial_{y}(\chi_{R}\ol{\wt{\eps}}yQ)\calH(\tfrac{y}{1+y^{2}}|\wt{\eps}|^{2})dy\nonumber \\
		&
		-{\int_{\bbR}}\partial_{y}(\chi_{R}\ol{\wt{\eps}}Q)\calH(\tfrac{1}{1+y^{2}}|\wt{\eps}|^{2})dy \nonumber \\
		\lesssim & M^{\frac{1}{2}}\eta^{\frac{1}{2}}\|\wt{\eps}\|_{\dot{\calH}_{R}^{1}}^{2}.\label{eq:Nolinear energy inter addition 1-2}
	\end{align}
	Now, we estimate $2\eqref{eq:Nolinear energy inter addition 2}+\eqref{eq:Nolinear energy inter addition 3}$.
	Using $2\frac{y}{1+y^{2}}=yQ^{2}=\calH(Q^{2})$ and \eqref{eq:HilbertProductRule}
	with $f=Q^{2}$ and $g=|\wt{\eps}|^{2}$, we deduce 
	\begin{align*}
		2\eqref{eq:Nolinear energy inter addition 2}+\eqref{eq:Nolinear energy inter addition 3}= & \Re\bigg({\int_{\bbR}}Q\chi_{R}\ol{\wt{\eps}}\calH(Q^{2})\calH(|\wt{\eps}|^{2})-Q\chi_{R}\ol{\wt{\eps}}\calH(Q^{2}\calH(|\wt{\eps}|^{2}))dy\bigg)\\
		= & \Re\bigg({\int_{\bbR}}Q\chi_{R}\ol{\wt{\eps}}Q^{2}|\wt{\eps}|^{2}+Q\chi_{R}\ol{\wt{\eps}}\calH(\calH(Q^{2})|\wt{\eps}|^{2})dy\bigg).
	\end{align*}
	Thus, we have 
	\begin{align}
		|2\eqref{eq:Nolinear energy inter addition 2}+\eqref{eq:Nolinear energy inter addition 3}|\lesssim M^{\frac{1}{2}}\eta^{\frac{1}{2}}\|\wt{\eps}\|_{\dot{\calH}_{R}^{1}}^{2}.\label{eq:Nolinear energy inter addition 1-3}
	\end{align}
	Collecting \eqref{eq:Nolinear energy inter addition 1-1}, \eqref{eq:Nolinear energy inter addition 1-2},
	and \eqref{eq:Nolinear energy inter addition 1-3}, we have 
	\begin{align}
		\eqref{eq:Nolinear energy inter addition}=O_{M}(1)\cdot o_{\eta\to0}(\|\wt{\eps}\|_{\dot{\calH}_{R}^{1}}^{2}).\label{eq:Nolinear energy goal 3}
	\end{align}
	Now we finish to prove \eqref{eq:Nolinear energy Step 1 goal} by
	summarizing from \eqref{eq:Nolinear energy goal 1}, \eqref{eq:Nolinear energy goal 2},
	and \eqref{eq:Nolinear energy goal 3}.
	
	\textbf{Step 2.} We estimate \eqref{eq:Nolinear energy middle and outer}
	by decomposing into intermediate and outer regions, and in this step
	we estimate the intermediate part and summarize in the main claim,
	\eqref{eq:Nolinear energy Step 2 goal-2}. We first compute 
	\begin{align}
		\eqref{eq:Nolinear energy middle and outer}= & \|(\bfD_{Q}+\tfrac{1}{2}\mathcal{H}(|\wt{\eps}|^{2}))\varphi_{R}\wt{\eps}\|_{L^{2}}^{2}\label{eq:Nolinear energy inter 0-1}\\
		& +\Re{\int_{\bbR}}(\bfD_{Q}+\tfrac{1}{2}\mathcal{H}(|\wt{\eps}|^{2}))\varphi_{R}\ol{\wt{\eps}}\cdot Q\calH(|\wt{\eps}|^{2})dy\label{eq:Nolinear energy inter 0-2}\\
		& +\|\tfrac{1}{2}Q\calH(|\wt{\eps}|^{2})\|_{L^{2}}^{2}.\label{eq:Nolinear energy inter remain}
	\end{align}
	For \eqref{eq:Nolinear energy inter 0-2}, we have 
	\begin{align}
		\eqref{eq:Nolinear energy inter 0-2} & =\Re{\int_{\bbR}}(\bfD_{Q}+\tfrac{1}{2}\mathcal{H}(|\wt{\eps}|^{2}))\varphi_{R}\ol{\wt{\eps}}\cdot\varphi_{\frac{R}{2}}Q\calH(|\wt{\eps}|^{2})dy\nonumber \\
		& \leq\tfrac{1}{2}\|(\bfD_{Q}+\tfrac{1}{2}\mathcal{H}(|\wt{\eps}|^{2}))\varphi_{R}\wt{\eps}\|_{L^{2}}^{2}+\tfrac{1}{2}\|\varphi_{\frac{R}{2}}Q\calH(|\wt{\eps}|^{2})\|_{L^{2}}^{2}.\label{eq:Nolinear energy inter 0-3}
	\end{align}
	Using $\langle y\rangle^{-2}(1+y^{2})=1$ and \eqref{eq:CommuteHilbert},
	we compute 
	\begin{align*}
		\calH(|\wt{\eps}|^{2})=\calH(\tfrac{1}{1+y^{2}}|\wt{\eps}|^{2})+y\calH(\tfrac{y}{1+y^{2}}|\wt{\eps}|^{2})-\frac{1}{\pi}{\int_{\bbR}}\tfrac{y}{1+y^{2}}|\wt{\eps}|^{2}dy^{\prime}.
	\end{align*}
	So, we estimate the second term of \eqref{eq:Nolinear energy inter 0-3}
	such as 
	\begin{align*}
		\|\varphi_{\frac{R}{2}}Q\calH(|\wt{\eps}|^{2})\|_{L^{2}}\lesssim(M^{\frac{1}{2}}\eta^{\frac{1}{2}}+R^{-\frac{1}{2}}M)\|\tfrac{1}{\langle y\rangle}\wt{\eps}\|_{L^{2}}=O_{M}(1)\cdot o_{R^{-1},\eta\to0}(\|\wt{\eps}\|_{\dot{\calH}_{R}^{1}}).
	\end{align*}
	Therefore, we have 
	\begin{align*}
		|\eqref{eq:Nolinear energy inter 0-2}| & \leq\tfrac{1}{2}\|(\bfD_{Q}+\tfrac{1}{2}\mathcal{H}(|\wt{\eps}|^{2}))\varphi_{R}\wt{\eps}\|_{L^{2}}^{2}+O_{M}(1)\cdot o_{R^{-1},\eta\to0}(\|\wt{\eps}\|_{\dot{\calH}_{R}^{1}}^{2})\\
		& =\tfrac{1}{2}\eqref{eq:Nolinear energy inter 0-1}+O_{M}(1)\cdot o_{R^{-1},\eta\to0}(\|\wt{\eps}\|_{\dot{\calH}_{R}^{1}}^{2}),
	\end{align*}
	and 
	\begin{align*}
		\eqref{eq:Nolinear energy middle and outer}=\eqref{eq:Nolinear energy inter 0-1}+\eqref{eq:Nolinear energy inter 0-2}+\eqref{eq:Nolinear energy inter remain}\geq\tfrac{1}{2}\eqref{eq:Nolinear energy inter 0-1}+\eqref{eq:Nolinear energy inter remain}+o_{R^{-1},\eta\to0}(\|\wt{\eps}\|_{\dot{\calH}_{R}^{1}}^{2}).
	\end{align*}
	\eqref{eq:Nolinear energy inter remain} is the outer term and will
	be estimated in Step 3. There is an outer part in \eqref{eq:Nolinear energy inter 0-1},
	too. We extract it in the following: 
	\begin{align}
		\eqref{eq:Nolinear energy inter 0-1}=\int_{|y|\geq R}\calE+\Re{\int_{|y|\geq R}}\partial_{y}(\varphi_{R}\ol{\wt{\eps}})\tfrac{2y}{1+y^{2}}\varphi_{R}\wt{\eps}, \label{eq:Nolinear energy inter and outer 2}
	\end{align}
	where
	\begin{align*}
		\calE\coloneqq |\partial_{y}(\varphi_{R}\wt{\eps})|^{2}+|\tfrac{y}{1+y^{2}}+\tfrac{1}{2}\calH(|\wt{\eps}|^{2})|^{2}|\varphi_{R}\wt{\eps}|^{2}+\Re[\partial_{y}(\varphi_{R}\ol{\wt{\eps}})\calH(|\wt{\eps}|^{2})\varphi_{R}\wt{\eps}].
	\end{align*}
	For the last term of \eqref{eq:Nolinear energy inter and outer 2}, we
	have 
	\begin{align*}
		\eqref{eq:Nolinear energy inter and outer 2}=-\frac{1}{2}{\int_{|y|\geq R}}\partial_{y}(\tfrac{2y}{1+y^{2}})|\varphi_{R}\wt{\eps}|^{2}={\int_{|y|\geq R}}\tfrac{y^{2}-1}{(1+y^{2})^{2}}|\varphi_{R}\wt{\eps}|^{2}\geq\frac{1}{2}\|\varphi_{R}\tfrac{1}{\langle y\rangle}\wt{\eps}\|_{L^{2}}^{2}.
	\end{align*}
	We use abbreviated notation for the integrals in the intermediate
	region and outer region. For a function $f$ with $\text{supp }f\subset[R,\infty)$,
	e.g. $\varphi_{R}\wt{\eps}$, we denote 
	\begin{align*}
		{\int_{\text{inter}}}f\coloneqq{\int_{|y|\geq R}}(1-\varphi_{4R^{2}}^{2})fdy,\quad{\int_{\text{outer}}}f\coloneqq{\int_{|y|\geq R}}\varphi_{4R^{2}}^{2}fdy.
	\end{align*}
	We also decompose $\int_{|y|\geq R}\calE$ into the intermediate part and the outer part,
	\begin{align*}
		\int_{|y|\geq R}\calE={\int_{\text{inter}}}\calE+{\int_{\text{outer}}}\calE.
	\end{align*}
	We control the
	intermediate term ${\int_{\text{inter}}}\calE$. We observe that
	\begin{align*}
		|\calH(|\wt{\eps}|^{2})|\leq C\|\wt{\eps}\|_{L^{2}}\|\wt{\eps}\|_{\dot{H}^{1}}\leq CM\eta.
	\end{align*}
	On $|y|\in[R,4R^{2}]$, we have 
	\begin{align*}
		|{\textbf 1}_{|y|\in[R,4R^{2}]}\tfrac{1}{2}\calH(|\wt{\eps}|^{2})|\leq(1-c(R,\eta))|{\textbf 1}_{|y|\in[R,4R^{2}]}(\tfrac{y}{1+y^{2}}+\tfrac{1}{2}\calH(|\wt{\eps}|^{2}))|
	\end{align*}
	where 
	\begin{align*}
		\frac{1}{2}<c(R,\eta)=\frac{\frac{4R^{2}}{1+(4R^{2})^{2}}-CM\eta}{\frac{4R^{2}}{1+(4R^{2})^{2}}-\frac{C}{2}M\eta}<1.
	\end{align*}
	Thus, we have 
	\begin{align*}
		{\int_{\text{inter}}}\calE\geq c(R,\eta){\int_{\text{inter}}}|\partial_{y}(\varphi_{R}\wt{\eps})|^{2}>\frac{1}{2}{\int_{\text{inter}}}|\partial_{y}(\varphi_{R}\wt{\eps})|^{2}.
	\end{align*}
	As a result, we have arrived at 
	\begin{align}
		\eqref{eq:Nolinear energy middle and outer}\geq & \eqref{eq:Nolinear energy inter remain}+\frac{1}{2}{\int_{\text{outer}}}\calE+\frac{1}{4}{\int_{\text{inter}}}|\partial_{y}(\varphi_{R}\wt{\eps})|^{2}\label{eq:Nolinear energy Step 2 goal-1}\\
		& +\frac{1}{4}\|\varphi_{R}\tfrac{1}{\langle y\rangle}\wt{\eps}\|_{L^{2}}^{2}+O_{M}(1)\cdot o_{R^{-1},\eta\to0}(\|\wt{\eps}\|_{\dot{\calH}_{R}^{1}}^{2}).\label{eq:Nolinear energy Step 2 goal-2}
	\end{align}
	\textbf{ }\eqref{eq:Nolinear energy Step 2 goal-1} will be handled
	in Step 3.\textbf{ }The first term of \eqref{eq:Nolinear energy Step 2 goal-2}
	is a part of $\norm{\wt{\eps}}_{\dot{\calH}_{R}^{1}}^{2}$.
	
	\textbf{Step 3.} We control the outer part \eqref{eq:Nolinear energy Step 2 goal-1}.
	A crucial feature of this step is to extract the energy $E(\varphi_{4R^{2}}\wt{\eps})$
	of the outer radiation as a lower bound. We observe that ${\int_{\text{outer}}}\calE$ becomes
	\begin{align}
		{\int_{\text{outer}}}\calE= & {\int_{\text{outer}}}|\partial_{y}(\varphi_{R}\wt{\eps})+\tfrac{1}{2}\calH(|\wt{\eps}|^{2})|^{2} +\tfrac{y^{2}}{(1+y^{2})^{2}}|\varphi_{R}\wt{\eps}|^{2}+\tfrac{y}{1+y^{2}}\calH(|\wt{\eps}|^{2})|\varphi_{R}\wt{\eps}|^{2}.\label{eq:Nolinear energy outer cancel pre}
	\end{align}
	Note that the first two terms of \eqref{eq:Nolinear energy outer cancel pre}
	are positive. By discarding the second term and part of the first term in \eqref{eq:Nolinear energy outer cancel pre}, we reduce the outer part \eqref{eq:Nolinear energy Step 2 goal-1}
	to 
	\begin{align}
		\eqref{eq:Nolinear energy Step 2 goal-1}=&\eqref{eq:Nolinear energy inter remain} +\frac{1}{2}{\int_{\text{outer}}}\calE+\frac{1}{4}{\int_{\text{inter}}}|\partial_{y}(\varphi_{R}\wt{\eps})|^{2}\nonumber \\
		\geq &
		\frac{1}{4}{\int_{\text{outer}}}|\partial_{y}(\varphi_{R}\wt{\eps})+\tfrac{1}{2}\calH(|\wt{\eps}|^{2})|^{2}+\frac{1}{4}{\int_{\text{inter}}}|\partial_{y}(\varphi_{R}\wt{\eps})|^{2}\label{eq:Nolinear energy outer energy}\\
		& +\eqref{eq:Nolinear energy inter remain}+\frac{1}{2}{\int_{\text{outer}}}\tfrac{y}{1+y^{2}}\calH(|\wt{\eps}|^{2})|\varphi_{R}\wt{\eps}|^{2}.\label{eq:Nolinear energy outer cancel}
	\end{align}
	We have to find out the outer energy from \eqref{eq:Nolinear energy outer energy}.
	To do this, we deal with the cut-off error. We note that \eqref{eq:Nolinear energy outer energy} equals to
	\begin{align*}
		\frac{1}{4}{\int_{|y|\geq R}}|\partial_{y}(\varphi_{R}\wt{\eps})|^{2}dy
		+\frac{1}{4}\Re{\int_{\text{outer}}}\partial_{y}(\varphi_{R}\wt{\eps})\calH(|\wt{\eps}|^{2})+\tfrac{1}{4}|\calH(|\wt{\eps}|^{2})|^2.
	\end{align*}
	We claim that 
	\begin{align}
		{\int_{|y|\geq R}}|\partial_{y}(\varphi_{R}\wt{\eps})|^{2}dy\geq & {\int_{|y|\geq R}}|\partial_{y}(\varphi_{4R^{2}}\wt{\eps})|^{2}dy+\|\chi_{2R^{2}}\varphi_{R}\wt{\eps}\|_{\dot{H}^{1}}^{2}-O_{\chi}(\|\wt{\eps}\|_{\dot{\calH}_{R}^{1}}^{2}).\label{eq:sec3 step3 cutoff claim}
	\end{align}
	We start from
	\begin{align}
		{\int_{|y|\geq R}}|\partial_{y}(\varphi_{R}\wt{\eps})|^{2}dy={\int_{\text{inter}}}|\partial_{y}(\varphi_{R}\wt{\eps})|^{2}+{\int_{\text{outer}}}|\partial_{y}(\varphi_{R}\wt{\eps})|^{2}.\label{eq:sec3 step3 cutoff 0}
	\end{align}
	For the outer term, we compute 
	\begin{align}
		{\int_{\text{outer}}}|\partial_{y}(\varphi_{R}\wt{\eps})|^{2}= & {\int_{|y|\geq R}}|\partial_{y}(\varphi_{4R^{2}}\wt{\eps})-\partial_{y}(\varphi_{4R^{2}})\wt{\eps}|^{2}dy\nonumber \\
		= & {\int_{|y|\geq R}}|\partial_{y}(\varphi_{4R^{2}}\wt{\eps})|dy+{\int_{|y|\geq R}}|\partial_{y}(\varphi_{4R^{2}})\wt{\eps}|^{2}dy\label{eq:sec3 step3 cutoff 1}\\
		& -2\Re{\int_{|y|\geq R}}\partial_{y}(\varphi_{4R^{2}}\wt{\eps})\partial_{y}(\varphi_{4R^{2}})\ol{\wt{\eps}}dy.\label{eq:sec3 step3 cutoff 2}
	\end{align}
	For the second term of \eqref{eq:sec3 step3 cutoff 1}, we deduce
	\begin{align}
		{\int_{|y|\geq R}}|\partial_{y}(\varphi_{4R^{2}})\wt{\eps}|^{2}dy=\frac{1}{(4R^{2})^{2}}{\int_{|y|\geq R}}|(\chi_{y})_{4R^{2}}\wt{\eps}|^{2}dy\lesssim_{\chi}\|Q\wt{\eps}\|_{L^{2}}^{2}\leq\|\wt{\eps}\|_{\dot{\calH}_{R}^{1}}^{2}.\label{eq:sec3 step3 cutoff 1-1}
	\end{align}
	We control \eqref{eq:sec3 step3 cutoff 2} as 
	\begin{align}
		\eqref{eq:sec3 step3 cutoff 2} & =-2{\int_{|y|\geq R}}|\partial_{y}(\varphi_{4R^{2}})\wt{\eps}|^{2}dy-2\Re{\int_{|y|\geq R}}\varphi_{4R^{2}}\partial_{y}(\varphi_{R}\wt{\eps})\partial_{y}(\varphi_{4R^{2}})\ol{\wt{\eps}}dy\nonumber \\
		& \leq O_{\chi}(1)\cdot\|\wt{\eps}\|_{\dot{\calH}_{R}^{1}}^{2}+2\bigg|{\int_{|y|\geq R}}\varphi_{4R^{2}}\partial_{y}(\varphi_{R}\wt{\eps})\partial_{y}(\varphi_{4R^{2}})\ol{\wt{\eps}}dy\bigg|.\label{eq:sec3 step3 cutoff 3}
	\end{align}
	The second term of \eqref{eq:sec3 step3 cutoff 3} is equal to
	\begin{align}
		&2\bigg|{\int_{|y|\geq R}}(\chi_{y}\varphi_1)(\tfrac{y}{4R^2})\cdot\partial_{y}(\varphi_{R}\wt{\eps})\cdot\tfrac{1}{4R^{2}}{\textbf 1}_{y\sim4R^{2}}\ol{\wt{\eps}}dy\bigg|\nonumber \\
		& \leq\frac{1}{C_{\chi}}{\int_{|y|\geq R}}|(\chi_{y}\varphi_1)(\tfrac{y}{4R^2})\cdot\partial_{y}(\varphi_{R}\wt{\eps})|^{2}dy+C_{\chi}^{\prime}\|\wt{\eps}\|_{\dot{\calH}_{R}^{1}}^{2}.\label{eq:sec3 step3 cutoff 3-1}
	\end{align}
	Here $C_{\chi}$ comes from $|\chi_{y}|^{2}\leq C_{\chi}\chi$.
	Therefore, gathering \eqref{eq:sec3 step3 cutoff 0}, \eqref{eq:sec3 step3 cutoff 1-1},
	\eqref{eq:sec3 step3 cutoff 3}, and \eqref{eq:sec3 step3 cutoff 3-1},
	we derive that 
	\begin{align}
		{\int_{|y|\geq R}}|\partial_{y}(\varphi_{R}\wt{\eps})|^{2}dy\geq & {\int_{|y|\geq R}}|\partial_{y}(\varphi_{4R^{2}}\wt{\eps})|^{2}dy-O_{\chi}(\|\wt{\eps}\|_{\dot{\calH}_{R}^{1}}^{2})\nonumber \\
		& +{\int_{|y|\geq R}}[(1-\varphi_{4R^{2}}^{2})-\tfrac{1}{C_{\chi}}(\chi_{y}\varphi_1)^2(\tfrac{y}{4R^2})]|\partial_{y}(\varphi_{R}\wt{\eps})|^{2}dy.\label{eq:sec3 step3 cutoff 4}
	\end{align}
	By the pointwise estimate $|\chi_{y}|^{2}\leq C_{\chi}\chi$, we
	have 
	\begin{align*}
		\tfrac{1}{C_{\chi}}|\chi_{y}\varphi|^{2}\leq\chi\leq2\chi-\chi^{2}=1-\varphi^{2}.
	\end{align*}
	This implies that, for large $2R^2\gg \|\chi_y\|_{L^\infty}$,
	\begin{align*}
		\eqref{eq:sec3 step3 cutoff 4}\geq\|{\textbf 1}_{|y|\leq 4R^2}\partial_y(\varphi_{R}\wt{\eps})\|_{L^2}^{2}\geq \|\chi_{2R^{2}}\varphi_{R}\wt{\eps}\|_{\dot{H}^{1}}^{2}
	\end{align*}
	and we conclude the claim \eqref{eq:sec3 step3 cutoff claim}.
	
	Now, using \eqref{eq:sec3 step3 cutoff claim}, we compute \eqref{eq:Nolinear energy outer energy}
	as 
	\begin{align*}
		4\eqref{eq:Nolinear energy outer energy}\geq & \|\partial_{y}(\varphi_{4R^{2}}\wt{\eps})+\tfrac{1}{2}\calH(|\wt{\eps}|^{2})\varphi_{4R^{2}}\wt{\eps}\|_{L^{2}}^{2}+\|\chi_{2R^{2}}\varphi_{R}\wt{\eps}\|_{\dot{H}^{1}}^{2}\\
		& +\Re{\int_{\bbR}}[\varphi_{4R^{2}},\partial_{y}](\varphi_{R}\ol{\wt{\eps}})\calH(|\wt{\eps}|^{2})\varphi_{4R^{2}}\wt{\eps}dy-O_{\chi}(\|\wt{\eps}\|_{\dot{\calH}_{R}^{1}}^{2})\\
		= & \|\partial_{y}(\varphi_{4R^{2}}\wt{\eps})+\tfrac{1}{2}\calH(|\varphi_{4R^{2}}\wt{\eps}|^{2})\varphi_{4R^{2}}\wt{\eps}\|_{L^{2}}^{2}+\|\chi_{2R^{2}}\varphi_{R}\wt{\eps}\|_{\dot{H}^{1}}^{2}\\
		& +O(R^{2}M\eta\|\tfrac{1}{\langle y\rangle}\wt{\eps}\|_{L^{2}}^{2})-O_{\chi}(\|\wt{\eps}\|_{\dot{\calH}_{R}^{1}}^{2})\\
		= & E(\varphi_{4R^{2}}\wt{\eps})+\|\chi_{2R^{2}}\varphi_{R}\wt{\eps}\|_{\dot{H}^{1}}^{2}
		\\
		&-O_{\chi}(\|\wt{\eps}\|_{\dot{\calH}_{R}^{1}}^{2})+O_{M}(1)\cdot o_{R^{-1},\eta\to0}(\|\wt{\eps}\|_{\dot{\calH}_{R}^{1}}^{2}).
	\end{align*}
	In addition, we check that $\eqref{eq:Nolinear energy outer energy}\geq0$. So, discarding $(1-\delta^{\prime})\eqref{eq:Nolinear energy outer energy}$
	for a small $\delta^{\prime}>0$, we obtain 
	\begin{align}
		\begin{split}\eqref{eq:Nolinear energy outer energy}\geq & \tfrac{\delta^{\prime}}{4}E(\varphi_{4R^{2}}\wt{\eps})+\tfrac{\delta^{\prime}}{4}\|\chi_{2R^{2}}\varphi_{R}\wt{\eps}\|_{\dot{H}^{1}}^{2}\\
			& -\delta^{\prime}O_{\chi}(\|\wt{\eps}\|_{\dot{\calH}_{R}^{1}}^{2})+O_{M}(1)\cdot o_{R^{-1},\eta\to0}(\|\wt{\eps}\|_{\dot{\calH}_{R}^{1}}^{2})
		\end{split}
		\label{eq:Nolinear energy outer extra energy goal}
	\end{align}
	We control \eqref{eq:Nolinear energy outer cancel}. We have 
	\begin{align}
		\eqref{eq:Nolinear energy outer cancel}= & \eqref{eq:Nolinear energy inter remain}+\frac{1}{2}{\int_{\bbR}}\tfrac{y}{1+y^{2}}\calH(|\wt{\eps}|^{2})\varphi_{4R^{2}}^{2}|\wt{\eps}|^{2}dy\nonumber \\
		= & \frac{1}{2}{\int_{\bbR}}\tfrac{y}{1+y^{2}}\calH((1-\varphi_{4R^{2}}^{2})|\wt{\eps}|^{2})|\varphi_{4R^{2}}\wt{\eps}|^{2}dy\label{eq:Nolinear energy outer good}\\
		& +\eqref{eq:Nolinear energy inter remain}+\frac{1}{2}{\int_{\bbR}}\tfrac{y}{1+y^{2}}\calH(|\varphi_{4R^{2}}\wt{\eps}|^{2})|\varphi_{4R^{2}}\wt{\eps}|^{2}dy.\label{eq:Nolinear energy outer bad}
	\end{align}
	For \eqref{eq:Nolinear energy outer good}, we have 
	\begin{align}
		|\eqref{eq:Nolinear energy outer good}|\lesssim R^{C}M\eta\|\tfrac{1}{\langle y\rangle}\wt{\eps}\|_{L^{2}}^{2}=O_{M}(1)\cdot o_{R^{-1},\eta\to0}(\|\wt{\eps}\|_{\dot{\calH}_{R}^{1}}^{2}).\label{eq:Nolinear energy outer cancel goal 1}
	\end{align}
	Now we estimate \eqref{eq:Nolinear energy outer bad}. For the second
	term of \eqref{eq:Nolinear energy outer bad}, by using $2\frac{y}{1+y^{2}}=\calH(Q^{2})$
	and \eqref{eq:HilbertProductRule} with $f=g=|\varphi_{4R^{2}}\wt{\eps}|^{2}$,
	we compute 
	\begin{align}
		4{\int_{\bbR}}\tfrac{y}{1+y^{2}}\calH(|\varphi_{4R^{2}}\wt{\eps}|^{2})|\varphi_{4R^{2}}\wt{\eps}|^{2}dy & =-2{\int_{\bbR}}Q^{2}\calH[\calH(|\varphi_{4R^{2}}\wt{\eps}|^{2})|\varphi_{4R^{2}}\wt{\eps}|^{2}]dy\nonumber \\
		& ={\int_{\bbR}}Q^{2}|\varphi_{4R^{2}}\wt{\eps}|^{4}-Q^{2}[\calH(|\varphi_{4R^{2}}\wt{\eps}|^{2})]^{2}dy.\label{eq:Nolinear energy outer bad 2}
	\end{align}
	To cancel out the second term of \eqref{eq:Nolinear energy outer bad 2},
	we use $\eqref{eq:Nolinear energy inter remain}=\|\tfrac{1}{2}Q\calH(|\wt{\eps}|^{2})\|_{L^{2}}^{2}$.
	We have 
	\begin{align}
		4\eqref{eq:Nolinear energy inter remain}= & {\int_{\bbR}}Q^{2}|\calH((1-\varphi_{4R^{2}}^{2})|\wt{\eps}|^{2})|^{2}+2Q^{2}\calH((1-\varphi_{4R^{2}}^{2})|\wt{\eps}|^{2})\calH(|\varphi_{4R^{2}}\wt{\eps}|^{2})dy\label{eq:Nolinear energy last inner}\\
		& +{\int_{\bbR}}Q^{2}|\calH(|\varphi_{4R^{2}}\wt{\eps}|^{2})|^{2}dy,\label{eq:Nolinear energy last outer}
	\end{align}
	and 
	\begin{align*}
		|\eqref{eq:Nolinear energy last inner}|\lesssim R^{C}M\eta\|\langle y\rangle^{-1}\wt{\eps}\|_{\dot{\calH}_{R}^{1}}^{2}=O_{M}(1)\cdot o_{R^{-1},\eta\to0}(\|\wt{\eps}\|_{\dot{\calH}_{R}^{1}}^{2}).
	\end{align*}
	In addition, the second term of \eqref{eq:Nolinear energy outer bad 2}
	is canceled by \eqref{eq:Nolinear energy last outer}. Therefore,
	we derive 
	\begin{align}
		\eqref{eq:Nolinear energy outer bad}\geq\frac{1}{8}{\int_{\bbR}}Q^{2}|\varphi_{4R^{2}}\wt{\eps}|^{4}+Q^{2}|\calH(|\varphi_{4R^{2}}\wt{\eps}|^{2})|^{2}dy+O_{M}(1)\cdot o_{R^{-1},\eta\to0}(\|\wt{\eps}\|_{\dot{\calH}_{R}^{1}}^{2}).\label{eq:Nolinear energy outer cancel goal 2}
	\end{align}
	Collecting \eqref{eq:Nolinear energy outer extra energy goal}, \eqref{eq:Nolinear energy outer cancel goal 1},
	and \eqref{eq:Nolinear energy outer cancel goal 2}, and discarding
	first two positive terms in \eqref{eq:Nolinear energy outer cancel goal 2},
	we deduce 
	\begin{align}
		\begin{split}  \eqref{eq:Nolinear energy Step 2 goal-1}\geq &\tfrac{\delta^{\prime}}{4}E(\varphi_{4R^{2}}\wt{\eps})+\tfrac{\delta^{\prime}}{4}\|\chi_{2R^{2}}\varphi_{R}\wt{\eps}\|_{\dot{H}^{1}}^{2}\\
			& -\delta^{\prime}O_{\chi}(\|\wt{\eps}\|_{\dot{\calH}_{R}^{1}}^{2})+O_{M}(1)\cdot o_{R^{-1},\eta\to0}(\|\wt{\eps}\|_{\dot{\calH}_{R}^{1}}^{2}).
		\end{split}
		\label{eq:Nolinear energy Step 3 goal}
	\end{align}
	Finally, combining \eqref{eq:Nolinear energy Step 1 goal}, \eqref{eq:Nolinear energy Step 2 goal-2},
	and \eqref{eq:Nolinear energy Step 3 goal}, we arrive at 
	\begin{align}
		\begin{split}E(Q+\wt{\eps})\geq & (C-\delta^{\prime}C_{\chi}^{\prime})\|\wt{\eps}\|_{\dot{\calH}_{R}^{1}}^{2}+\tfrac{\delta^{\prime}}{4}\|\chi_{2R^{2}}\varphi_{R}\wt{\eps}\|_{\dot{H}^{1}}^{2}+\tfrac{\delta^{\prime}}{4}E(\varphi_{4R^{2}}\wt{\eps})\\
			& -C\|{\textbf 1}_{|y|\in[R,2R]}|\wt{\eps}|_{-1}\|_{L^{2}}^{2}+O_{M}(1)\cdot o_{R^{-1},\eta\to0}(\|\wt{\eps}\|_{\dot{\calH}_{R}^{1}}^{2}),
		\end{split}
		\label{eq:finalizing (3.5)}
	\end{align}
	where $C$ is the implicit constant in \eqref{eq:finalizing (3.5)}
	depending only on $\mathcal{Z}_{j}$ and $\chi$. Taking small $\delta^{\prime}=\delta^{\prime}(\calZ_{j},\chi)$,
	large $R=R(\calZ_{j},\chi,M,\delta^{\prime})\gg1$ and small $\eta=\eta(\calZ_{j},\chi,M,\delta^{\prime},R)\ll1$,
	we can conclude \eqref{eq:Nonlinear energy before averaging}.
	
	\textbf{Step 4.} Now, we finalize the proof by taking a logarithmic
	average over $R$. One goal is to eliminate the term $\|{\textbf 1}_{[R,2R]}|\wt{\eps}|_{-1}\|_{L^{2}}^{2}$.
	Indeed, we replace $R$ by $R^{\prime}$, and take a logarithmic integral $\frac{1}{\log R}\int_{R}^{R^{2}}\frac{dR^{\prime}}{R^{\prime}}$ of \eqref{eq:Nonlinear energy before averaging} on $[R,R^{2}]$. Then, by Fubini theorem, we have
	\begin{align*}
		\frac{1}{\log R}\int_{R}^{R^2}\left\|{\textbf 1}_{|y|\in[R^{\prime},2R^{\prime}]} |\wt\eps|_{-1}\right\|_{L_y^2}^2 \frac{dR^\prime}{R^\prime}
		=&\frac{1}{\log R}
		\int_{R}^{R^2}
		\left[\int {\textbf 1}_{|y|\in[R^{\prime},2R^{\prime}]} |\wt\eps|_{-1}^2(y) dy \right]\frac{dR^\prime}{R^\prime}
		\\
		=&
		\int \left[\frac{1}{\log R}\int_{R}^{R^2} {\textbf 1}_{|y|\in[R^{\prime},2R^{\prime}]} \frac{dR^\prime}{R^\prime} \right] |\wt\eps|_{-1}^2(y) dy.
	\end{align*}
	By direct computation, we derive
	\begin{align*}
		\frac{1}{\log R}\int_{R}^{R^2} {\textbf 1}_{|y|\in[R^{\prime},2R^{\prime}]} \frac{dR^\prime}{R^\prime}
		=
		\begin{cases}
			0, & |y| \notin [R,2R^2], \\
			\frac{1}{\log R} \log (\frac{|y|}{R}), 
			& |y| \in [R,2R], \\
			\frac{\log 2}{\log R}, 
			& |y| \in [2R,R^2], \\
			\frac{1}{\log R} \log (\frac{2R^2}{|y|}),
			& |y| \in [R^2,2R^2].
		\end{cases}
	\end{align*}
	Therefore, we have
	\begin{align*}
		\frac{1}{\log R}\int_{R}^{R^2} {\textbf 1}_{|y|\in[R^{\prime},2R^{\prime}]} \frac{dR^\prime}{R^\prime}
		\leq {\textbf 1}_{|y|\in[R,2R^{2}]}  \frac{\log 2}{\log R},
	\end{align*}
	which implies
	\begin{align*}
		\frac{1}{\log R}\int_{R}^{R^2}\left\|{\textbf 1}_{|y|\in[R^{\prime},2R^{\prime}]} |\wt\eps|_{-1}\right\|_{L_y^2}^2 \frac{dR^\prime}{R^\prime}
		\lesssim 
		\frac{1}{\log R}\left\|{\textbf 1}_{|y|\in[R,2R^{2}]} |\wt\eps|_{-1}\right\|_{L^2}^2.
	\end{align*}
	This and \eqref{eq:Nonlinear energy before averaging} yield (after taking $R\gg1$ larger if necessary)
	\begin{align}
		E(Q+\wt{\eps}) & \gtrsim_{M}\|\wt{\eps}\|_{\dot{\calH}_{2R^{2}}^{1}}^{2}+E(\varphi_{4R^{4}}\wt{\eps})-O(\tfrac{1}{\log R}\left\|{\textbf 1}_{|y|\in[R,2R^{2}]}|\wt{\eps}|_{-1}\right\|_{L^{2}}^{2})\nonumber \\
		& \gtrsim_{M}\|\wt{\eps}\|_{\dot{\calH}_{2R^{2}}^{1}}^{2}+E(\varphi_{4R^{4}}\wt{\eps})\geq E(\varphi_{4R^{4}}\wt{\eps}).\label{eq:Nolinear energy goal cutoff 1}
	\end{align}
	Moreover, we redo all the argument with $R_{1}=\sqrt{2}R^{2}$, by
	taking smaller $\eta_{1}$ depending on $R_{1}$. Then we have 
	\begin{align}
		E(Q+\wt{\eps}) & \gtrsim_{M}\|\wt{\eps}\|_{\dot{\calH}_{4R^{4}}^{1}}^{2}+E(\varphi_{16R^{8}}\wt{\eps})\geq\|\wt{\eps}\|_{\dot{\calH}_{4R^{4}}^{1}}^{2}.\label{eq:Nolinear energy goal cutoff 2}
	\end{align}
	Hence, combining \eqref{eq:Nolinear energy goal cutoff 1} and \eqref{eq:Nolinear energy goal cutoff 2},
	and renaming $4R^{4}$ by $R$, and $\eta_{1}$ by $\eta$, we conclude
	\eqref{eq:energy bubbling}. This finishes the proof. 
\end{proof}

\section{Multi-soliton configuration}
\label{sec:multisoliton}

\label{sec:Multi sol decomp} In this section, we establish a multi-soliton
decomposition by iterating one bubbling procedure in Proposition~\ref{prop:Decomposition}.
In this process, we bypass a time-sequential soliton resolution and
directly prove the soliton resolution continuously in time. From this
section, for given initial data $v_{0}$ to \eqref{CMdnls-gauged},
we denote the mass and energy for $v_{0}$ by $(M_{0},E_{0})\coloneqq(M(v_{0}),E(v_{0}))$.
We assume that the solution $v(t)$ to \eqref{CMdnls-gauged} blows
up at the time $T$. Then $\|v(t)\|_{\dot{H}^{1}}\to\infty$ as $t\to T$,
and there exists a time $0<T_{1}<T$ such that $\sqrt{E(v(t))}\leq\alpha^{*}\|v(t)\|_{\dot{H}}$
on $[T_{1},T)$. We apply the decomposition, Proposition \eqref{prop:Decomposition},
and then there exist $\lambda_{1}(t),\gamma_{1}(t),x_{1}(t),\wt{\eps}_{1}$
such that 
\begin{align}
	v=[Q+\wt{\eps}_{1}]_{\lambda_{1},\gamma_{1},x_{1}},\quad(\wt{\eps}_{1},\mathcal{Z}_{k})_{r}=0\text{ for }k=1,2,3,\label{eq:4intro decompose}
\end{align}
\begin{align}
	\lambda_{1}^{-1}\sim_{M_{0}}\|v(t)\|_{\dot{H}^{1}}\to\infty,\label{eq:v(t) blowup rate lambda 1}
\end{align}
\begin{align}
	\|\wt{\eps}_{1}\|_{\dot{\calH}_{R}^{1}}^{2}+E(\varphi_{R}\wt{\eps}_{1})\lesssim_{M_{0}}\lambda_{1}^{2}E_{0}.\label{eq:4intro nonlinear coercivity}
\end{align}
\eqref{eq:4intro nonlinear coercivity} indicates that we can control
the inner part of radiation $\|\wt{\eps}_{1}\|_{\dot{\calH}_{R}^{1}}\lesssim\lambda_{1}$
and the outer radiation energy $E(\varphi_{R}\wt{\eps}_{1})\lesssim\la_{1}^{2}$.
However, we cannot control the outer part of radiation $\varphi_{R}\wt{\eps}_{1}$.
More precisely, it is possible that $\|\varphi_{R}\wt{\eps}_{1}\|_{\dot{H}^{1}}\gg\lambda_{1}$,
and one does not expect that the radiation term $[\wt{\eps}_{1}]_{\lambda_{1},\gamma_{1},x_{1}}$
to converge to some asymptotic profile.\footnote{It is instructive to compare to the self-dual \emph{equivariant} Chern-Simons-Schrödinger
	equation (CSS) in \cite{KimKwonOh2025AJM}. There, the authors
	obtain $\|\wt{\eps}_{1}\|_{\dot{\calH}^{1}}\lesssim\lambda_{1}$ and
	proceed to show soliton resolution with a single bubble.} In this case, for the outer radiation, we sustain near-zero energy
$E(\varphi_{R}\wt{\eps}_{1})\lesssim\la_{1}^{2}\ll\|\varphi_{R}\wt{\eps}_{1}\|_{\dot{H}^{1}}^{2}$.
This bound enables us to reapply Proposition~\ref{prop:Decomposition}
to extract another bubble from $\varphi_{R}\wt{\eps}_{1}$. We will
iterate this procedure until the radiation satisfies $\|\varphi_{R}\wt{\eps}_{N}\|_{\dot{H}^{1}}\lesssim\lambda_{N}$
and thus converges to some asymptotic profile. At each bubbling out,
the mass of radiation drops by $M(Q)$. So, the iteration should halt
in finitely many steps. Indeed, for the second bubbling, the radiation satisfies
either $\|\varphi_{R}\wt{\eps}_{1}\|_{\dot{H}^{1}}\gg\lambda_{1}$
or $\|\varphi_{R}\wt{\eps}_{1}\|_{\dot{H}^{1}}\lesssim\lambda_{1}$.
In other words, either 
\begin{align*}
	\liminf_{t\to T}\frac{\lambda_{1}(t)}{\|\varphi_{R}\wt{\eps}_{1}(t)\|_{\dot{H}^{1}}}=0,\quad\text{or}\quad\liminf_{t\to T}\frac{\lambda_{1}(t)}{\|\varphi_{R}\wt{\eps}_{1}(t)\|_{\dot{H}^{1}}}>0.
\end{align*}
The latter case corresponds to $\|\wt{\eps}_{1}\|_{\dot{H}^{1}}\lesssim\lambda_{1}$,
which results in convergence to an asymptotic profile. In this case, $v(t)$
is a single-bubble blow-up solution. For the former case, we may reapply Proposition~\ref{prop:Decomposition} to $v=\varphi_{R}\wt{\eps}_{1}$ in order to extract a second bubble along a sequence of times: there exists
a sequence $\{t_{n}\}_{n\in\bbN}$ with $t_{n}\to T$ such that $\lim_{n\to\infty}\frac{\lambda_{1}(t_{n})}{\|\varphi_{R}\wt{\eps}_{1}(t_{n})\|_{\dot{H}^{1}}}=0$.
It then follows that $\varphi_{R}\wt{\eps}_{1}$ has sequentially small energy, namely,
\begin{equation*}
	\sqrt{E(\varphi_{R}\wt{\eps}_{1}(t_{n}))}<\alpha^{*}\|\varphi_{R}\wt{\eps}_{1}(t_{n})\|_{\dot{H}^{1}}.
\end{equation*}
Hence, we have the decomposition: 
\begin{align*}
	\begin{gathered}
		\varphi_{R}\wt{\eps}_{1}(t_{n})=[Q+\wt{\eps}_{2,n}]_{\wt{\lambda}_{2,n},\wt{\gamma}_{2,n},\wt x_{2,n}},
		\\
		v(t_{n})=([Q]_{\lambda_{1},\gamma_{1},x_{1}}+[Q]_{\lambda_{2},\gamma_{2},x_{2}}+\eps_{2})(t_{n}),
	\end{gathered}
\end{align*}
with 
\begin{align*}
	\begin{gathered}
		\wt{\lambda}_{2,n}\sim\frac{\|Q\|_{\dot{H}^{1}}}{\|\varphi_{R}\wt{\eps}_{1}(t_{n})\|_{\dot{H}^{1}}},
		\quad 
		(\lambda_{2},\gamma_{2},x_{2})(t_{n})\coloneqq(\lambda_{1}\wt{\lambda}_{2,n},\gamma_{1}+\wt{\gamma}_{2,n},x_{1}+\lambda_{1}\wt x_{2,n})(t_{n}),
		\\
		\eps_{2}\coloneqq[\chi_{R}\wt{\eps}_{1}]_{\lambda_{1},\gamma_{1},x_{1}}+[\wt{\eps}_{1}]_{\lambda_{2},\gamma_{2},x_{2}}.
	\end{gathered}
\end{align*}
Moreover, we have the \eqref{eq:4intro nonlinear coercivity} bound
for the next step: 
\begin{align*}
	\|\wt{\eps}_{2,n}\|_{\dot{\calH}_{R}^{1}}^{2}+E(\varphi_{R}\wt{\eps}_{2,n})\lesssim\wt{\lambda}_{2,n}^{2}E(\varphi_{R}\wt{\eps}_{1}(t_{n}))\lesssim_{M_{0}}(\wt{\lambda}_{2,n}\lambda_{1})^{2}(t_{n})E_{0}=\lambda_{2}^{2}(t_{n})E_{0}.
\end{align*}
And then, we can check whether $\liminf_{n\to\infty}\frac{\lambda_{2}(t)}{\|\varphi_{R}\wt{\eps}_{2}(t_{n})\|_{\dot{H}^{1}}}=0$
or not to proceed to the next step. This allows us to obtain sequential
in time soliton resolution. However, we will bypass the sequential
soliton resolution and directly prove continuous in time soliton resolution.
For this goal, we improve 
\begin{equation}
	\liminf_{t\to T}\frac{\lambda_{1}(t)}{\|\varphi_{R}\wt{\eps}_{1}(t)\|_{\dot{H}^{1}}}=0\quad\Rightarrow\quad\lim_{t\to T}\frac{\lambda_{1}(t)}{\|\varphi_{R}\wt{\eps}_{1}(t)\|_{\dot{H}^{1}}}=0.\label{eq:no return property intro}
\end{equation}
Equipped with \eqref{eq:no return property intro}, we can derive
continuous-in-time soliton resolution as explained above. In the rest
of this section, we establish the multi-soliton decomposition by an
induction argument. Among other things, as we work in a nonradial
setting, a crucial part of the induction step is to control the translation
parameters: 
\[
\lim_{t\to T}x_{j}(t)\to x_{j}(T),\qquad|x_{j}(T)|<\infty.
\]

\subsection{Induction}

In order to implement the previously described strategy, we will prove
via an induction argument. Fix an initial data $v_{0}\in H^{1}$ and
the blow-up solution $v(t)$ as above. Fix $R,\eta$, and $\alpha^{*}$
depending on $M_{0}$ as in Proposition~\ref{prop:Decomposition}.
At the zeroth step, we use conventions $T_{0}\coloneqq0$, $\varphi_{R}\wt{\eps}_{0}(t)\coloneqq v(t)$,
and $(\lambda_{0},\gamma_{0},x_{0})=(1,0,0)$.

For each $k\geq1$, we define the induction hypothesis statement for
$k$-soliton configuration, $P(k;v_{0})=P(k)$, as follows; 
\begin{quote}
	For $1\le j\le k$, there exist a time $T_{j}>0$, $C^{1}$ modulation
	parameters $(\wt{\lambda}_{j},\wt{\gamma}_{j},\wt x_{j})$ and $(\lambda_{j},\gamma_{j},x_{j})\in\bbR_{+}\times\bbR/2\pi\bbZ\times\bbR$,
	and radiation terms $\wt{\eps}_{j},\eps_{j}\in\dot{\calH}^{1}$ defined on
	$t\in[T_{j},T)$ such that the following properties hold:
\end{quote}

\begin{enumerate}[label=(H\arabic{*})]
	\item \label{state:no-return property}(No-return property) We have 
	\begin{equation}
		\lim_{t\to T}\frac{\lambda_{j-1}(t)}{\|\varphi_{R}\wt{\eps}_{j-1}(t)\|_{\dot{H}^{1}}}=0.\label{eq:no return}
	\end{equation}
	\item \label{state:further decomposition}(Further decomposition) There
	exists a $T_{j-1}<T_{j}<T$ such that there exist $C^{1}$ modulation
	parameters $(\wt{\lambda}_{j},\wt{\gamma}_{j},\wt x_{j})\in\bbR_{+}\times\bbR/2\pi\bbZ\times\bbR$
	and radiation term $\wt{\eps}_{j}\in\dot{\calH}^{1}$ on $t\in[T_{j},T)$
	which satisfy the following: 
	\begin{align*}
		\varphi_{R}\wt{\eps}_{j-1}=[Q+\wt{\eps}_{j}]_{\wt{\lambda}_{j},\wt{\gamma}_{j},\wt x_{j}},\quad\quad(\wt{\eps}_{j},\mathcal{Z}_{i})_{r}=0\text{ for }i=1,2,3,
	\end{align*}
	with the smallness $\|\wt{\eps}_{j}\|_{\dot{\calH}^{1}}<\eta$ and
	the energy bubbling 
	\begin{align}
		\|\wt{\eps}_{j}\|_{\dot{\calH}_{R}^{1}}^{2}+E(\varphi_{R}\wt{\eps}_{j})\lesssim_{M_{0}}\wt{\lambda}_{j}^{2}E(\varphi_{R}\wt{\eps}_{j-1})\lesssim_{M_{0},E_{0}}\lambda_{j}^{2}.\label{eq:induction k nonlinear energy}
	\end{align}
	Moreover, for $t\in[T_{j},T)$, we have 
	\begin{align}
		v(t)=\sum_{i=1}^{j}[Q]_{\lambda_{i},\gamma_{i},x_{i}}+\eps_{j}\label{eq:induction k v decom}
	\end{align}
	where 
	\begin{align}
		\eps_{j}=\sum_{i=1}^{j-1}[\chi_{R}\wt{\eps}_{i}]_{\lambda_{i},\gamma_{i},x_{i}}+[\wt{\eps}_{j}]_{\lambda_{j},\gamma_{j},x_{j}},\label{eq:induction k radiation decom}
	\end{align}
	and 
	\begin{align}
		(\lambda_{j},\gamma_{j},x_{j})\coloneqq(\lambda_{j-1}\wt{\lambda}_{j},\gamma_{j-1}+\wt{\gamma}_{j},x_{j-1}+\lambda_{j-1}\wt x_{j}).\label{eq:induction k modul para}
	\end{align}
	\item \label{state:soliton decoupling}(Soliton decoupling) We have $\lim_{t\to T^{-}}\lambda_{j}(t)=0$.
	If $j\geq2$, we have 
	\begin{align}
		\lim_{t\to T}\left(\frac{\lambda_{j}(t)}{\lambda_{j-1}(t)}+\left|\frac{x_{j}(t)-x_{j-1}(t)}{\lambda_{j-1}(t)}\right|\right)=\lim_{t\to T}(\wt{\lambda}_{j}(t)+|\wt x_{j}(t)|)=\infty.\label{eq:induction k no same rate}
	\end{align}
	Moreover, we have 
	\begin{align}
		\lambda_{1}\lesssim\lambda_{2}\lesssim\cdots\lesssim\lambda_{j},\label{eq:induction lambda}
	\end{align}
	and if $\lambda_{i}\sim\lambda_{j}$ for some $i<j$, we have 
	\begin{align}
		\lim_{t\to T}\left|\frac{x_{j}(t)-x_{i}(t)}{\lambda_{i}(t)}\right|=\infty.\label{eq:induction translation}
	\end{align}
	\item \label{state:translation}(Convergence of the translation) $\lim_{t\to T}x_{j}(t)\eqqcolon x_{j}(T)$
	exists and $|x_{j}(T)|<\infty$. 
\end{enumerate}
Once we have $P(k)$ with $k$-soliton decomposition, we encounter
a dichotomy for the next step. Define the statements $Q(k;v_{0})=Q(k)$
and $Q(k;v_{0})^{c}=Q(k)^{c}$ by the followings: 
\begin{align*}
	Q(k):\liminf_{t\to T}\frac{\lambda_{k}(t)}{\|\varphi_{R}\wt{\eps}_{k}(t)\|_{\dot{H}^{1}}}=0,\qquad Q(k)^{c}:\liminf_{t\to T^{-}}\frac{\lambda_{k}(t)}{\|\varphi_{R}\wt{\eps}_{k}(t)\|_{\dot{H}^{1}}}>0.
\end{align*}
Now, we are ready to state the initial case and induction step. 
\begin{lem}[Initial case]
	\label{lem:Induction initial} Let $v$ be a blow-up solution to
	\eqref{CMdnls-gauged} with initial data $v_{0}\in H^{1}$. Then $P(1)$
	is true. 
\end{lem}

\begin{lem}[Induction]
	\label{lem:Induction} Assume that $P(k-1)$ is true. If $Q(k-1)$
	is true, then $P(k)$ is true. 
\end{lem}

We also claim that the induction stops within finite steps with respect
to the initial value of mass. 
\begin{lem}[Halting of induction]
	\label{lem:Halting induction} There exists an integer $N\in\bbN$
	with $1\leq N\leq\frac{M(v_{0})}{M(Q)}$ such that $P(N)$ and $Q(N)^{c}$
	hold true. Moreover, we have $\|\eps_{N}\|_{H^{1}}\lesssim1$ uniformly
	in time. 
\end{lem}

At each step, we have a dichotomy, $Q(k)$ or $Q(k)^{c}$. If $Q(k)$
holds true (and $P(k+1)$), we can extract a soliton. Otherwise, we
have $\|\varphi_{R}\wt{\eps}_{k}(t)\|_{\dot{H}^{1}}\lesssim\la_{k}(t)$
and then $\|\eps_{k}(t)\|_{H^{1}}\lesssim1$. By a similar argument
to \cite{MerleRaphael2005CMP}, $\varepsilon_{k}$ converges to some
asymptotic profile.

In the above configuration, it is possible that $\lim_{t\to T}\big|\frac{x_{j}(t)-x_{i}(t)}{\lambda_{i}(t)}\big|<\infty$.
This a nonradial version of bubble tree, i.e., $|x_{i}(t)-x_{j}(t)|\lesssim\min(\la_{i},\la_{j})\to0$.
Such a bubble tree naturally appears in statements of other soliton
resolutions for blow-up solutions. However, to our knowledge, there
is no finite-time bubble tree construction in relevant models. Moreover,
there is a specific clue indicating the absence of a bubble tree in
\eqref{CMdnls-gauged}. If there were a bubble tree in \eqref{CMdnls-gauged},
then after taking the inverse gauge transform $\mathcal{G}^{-1}$,
there would be discontinuities in phase at some points for each soliton,
which is somewhat unusual. See Remark~\ref{rem:phase correction}
for more details. Fortunately, we are able to verify that there is
no bubble tree in \eqref{CMdnls} and \eqref{CMdnls-gauged}. 
\begin{prop}[No bubble tree solutions]
	\label{prop:no bubble tree} Suppose $P(k)$ is true for $k\geq2$.
	For any $1\leq i\neq j\leq k$, we have 
	\begin{align}
		\lim_{t\to T}\left|\frac{x_{i}(t)-x_{j}(t)}{\lambda_{i}(t)}\right|=\infty.\label{eq:no bubble eq in prop}
	\end{align}
\end{prop}

From the continuity of $x(t)$ and $\la(t)$, we have either $\lim_{t\to T}\frac{x_{i}-x_{j}}{\lambda_{i}}$
is $\infty$ or $-\infty$.

\subsection{Proof of the induction}

We provide the proof of induction steps, Lemmas~\ref{lem:Induction initial},\ref{lem:Induction},
\ref{lem:Halting induction}, and Proposition~\ref{prop:no bubble tree}.

We begin with defining abridged notations for modulation parameters
and transformations: 
\begin{align}
	\textrm{g}(t)\coloneqq(\lambda,\gamma,x)(t),\quad\textrm{g}_{j}(t)\coloneqq(\lambda_{j},\gamma_{j},x_{j})(t),\quad\wt{\textrm{g}}_{j}(t)\coloneqq(\wt{\lambda}_{j},\wt{\gamma}_{j},\wt x_{j})(t),\label{eq: g def}
\end{align}
\begin{align}
	[f]_{\textrm{g}_{j}}\coloneqq[f]_{\lambda_{j},\gamma_{j},x_{j}},\,[f]_{\wt{\textrm{g}}_{j}}\coloneqq[f]_{\wt{\lambda}_{j},\wt{\gamma}_{j},\wt x_{j}},\,[f]_{\textrm{g}_{j}}^{-1}\coloneqq[f]_{\lambda_{j},\gamma_{j},x_{j}}^{-1},\, [f]_{\wt{\textrm{g}}_{j}}^{-1}\coloneqq[f]_{\wt{\lambda}_{j},\wt{\gamma}_{j},\wt x_{j}}^{-1}.\label{eq: g def renormal}
\end{align}
Moreover, we denote by $\textrm{g}_{i,j}$, 
\begin{align}
	\textrm{g}_{i,j}\coloneqq(\lambda_{i}\lambda_{j}^{-1},\gamma_{i}-\gamma_{j},(x_{i}-x_{j})\lambda_{j}^{-1}).\label{eq: g_i,j def}
\end{align}

\begin{lem}
	Let \emph{$\textrm{g}_{j}$} and \emph{$\wt{\textrm{g}}_{j}$} satisfy
	relations in \eqref{eq:induction k modul para}. Then, for any $1\le\ell<k$,
	we have 
	\begin{align}
		[[[f]_{\wt{\emph{\textrm{g}}}_{k}}]_{\wt{\emph{\textrm{g}}}_{k-1}}\cdots]_{\wt{\emph{\textrm{g}}}_{\ell+1}}=[f]_{\emph{\textrm{g}}_{k,\ell}},\quad[[[f]_{\wt{\emph{\textrm{g}}}_{\ell+1}}^{-1}]_{\wt{\emph{\textrm{g}}}_{\ell+2}}^{-1}\cdots]_{\wt{\emph{\textrm{g}}}_{k}}^{-1}=[f]_{\emph{\textrm{g}}_{k,\ell}}^{-1}.\label{eq:modulation induction}
	\end{align}
\end{lem}

\begin{proof}
	Since $[[f]_{\textrm{g}}]_{\textrm{g}}^{-1}=f$, it suffices to show,
	for any $\ell<k$, 
	\begin{align}
		[[[f]_{\wt{\textrm{g}}_{k}}]_{\wt{\textrm{g}}_{k-1}}\cdots]_{\wt{\textrm{g}}_{\ell+1}}=[f]_{\textrm{g}_{k,\ell}}.\label{eq:modulation induction pf1}
	\end{align}
	We show \eqref{eq:modulation induction pf1} by induction in descending
	order. For the initial case, using the relation \eqref{eq:induction k modul para},
	we have 
	\begin{align}
		[[f]_{\wt{\textrm{g}}_{k}}]_{\wt{\textrm{g}}_{k-1}}(x) & =\frac{e^{i(\wt{\gamma}_{k-1}+\wt{\gamma}_{k})}}{\wt{\lambda}_{k-1}\wt{\lambda}_{k}}f\bigg(\frac{x-\wt x_{k-1}-\wt{\lambda}_{k-1}\wt x_{k}}{\wt{\lambda}_{k-1}\wt{\lambda}_{k}}\bigg)\nonumber \\
		& =\frac{e^{i(\gamma_{k}-\gamma_{k-2})}}{\lambda_{k}\lambda_{k-2}^{-1}}f\bigg(\frac{x-(x_{k}-x_{k-2})\lambda_{k-2}^{-1}}{\lambda_{k}\lambda_{k-2}^{-1}}\bigg)=[f]_{\textrm{g}_{k,k-2}}.\label{eq:modulation induction pf2}
	\end{align}
	For a fixed $k$, we assume that \eqref{eq:modulation induction pf1}
	is true for all $j+1\leq\ell\leq k$. Then, we want to show \eqref{eq:modulation induction pf1}
	for $\ell=j$, and it suffices to prove 
	\begin{align}
		[[f]_{\textrm{g}_{k,\ell}}]_{\wt{\textrm{g}}_{\ell}}=[f]_{\textrm{g}_{k,\ell-1}}.\label{eq:modulation induction pf3}
	\end{align}
	Similarly to \eqref{eq:modulation induction pf2}, we can show \eqref{eq:modulation induction pf3}
	for scaling and phase rotation parts. For the translation part, we
	compute the translation part of RHS of \eqref{eq:modulation induction pf3}
	by 
	\begin{align*}
		\frac{\frac{x-\wt x_{\ell}}{\wt{\lambda}_{\ell}}-(x_{k}-x_{\ell})\lambda_{\ell}^{-1}}{\lambda_{k}\lambda_{\ell}^{-1}}=\frac{x-\wt x_{\ell}-(x_{k}-x_{\ell})\lambda_{\ell-1}^{-1}}{\lambda_{k}\lambda_{\ell-1}^{-1}}.
	\end{align*}
	Using $x_{\ell}=x_{\ell-1}+\wt{\lambda}_{\ell-1}\wt x_{\ell}$, we
	have 
	\begin{align*}
		\wt x_{\ell}+(x_{k}-x_{\ell})\lambda_{\ell-1}^{-1}=(x_{k}-x_{\ell-1})\lambda_{\ell-1}^{-1},
	\end{align*}
	and this proves \eqref{eq:modulation induction pf3}. Therefore, by
	the induction, we have \eqref{eq:modulation induction pf1}, and finish
	the proof. 
\end{proof}
When we have $k$-soliton configuration, we naturally have an asymptotic
decoupling. 
\begin{lem}[Decoupling of localized mass]
	We suppose that $P(k)$ holds true. Let $\psi\in L^{\infty}$. Then,
	we have 
	\begin{align}
		\int_{\bbR}\psi|v|^{2}dx=\sum_{j=1}^{k}\int_{\bbR}\psi|[Q]_{\emph{\textrm{g}}_{k}}|^{2}+\psi|\eps_{k}|^{2}dx+o_{t\to T}(1)\cdot\|\psi\|_{L^{\infty}}.\label{eq:soliton decoupling}
	\end{align}
\end{lem}

\begin{proof}
	From the assumption, we have the decomposition \eqref{eq:induction k v decom}
	with \eqref{eq:induction k no same rate} and \eqref{eq:induction translation}.
	We compute $\int_{\bbR}\psi|v|^{2}$ by 
	\begin{align}
		{\int_{\bbR}}\psi|{\sum_{j=1}^{k}}[Q]_{\textrm{g}_{j}}+\eps_{k}|^{2}= & {\sum_{j=1}^{k}}{\int_{\bbR}}\psi|[Q]_{\textrm{g}_{j}}|^{2}+\psi|\eps_{k}|^{2}\nonumber \\
		& +2{\sum_{j<i}^{k}}(\psi[Q]_{\textrm{g}_{j}},[Q]_{\textrm{g}_{i}})_{r}+2{\sum_{j=1}^{k}}(\psi[Q]_{\textrm{g}_{j}},\eps_{k})_{r}.\label{eq:no return gr interaction}
	\end{align}
	Our goal is to show $\eqref{eq:no return gr interaction}=o_{t\to T}(1)$.
	For the first term, if $\lambda_{j}\ll\lambda_{i}$, we deduce 
	\begin{align*}
		(\psi[Q]_{\textrm{g}_{j}},[Q]_{\textrm{g}_{i}})_{r}=(Q,[\psi[Q]_{\textrm{g}_{i}}]_{\textrm{g}_{j}}^{-1})_{r} & \lesssim\|Q\|_{L^{1+}}\|[\psi[Q]_{\textrm{g}_{i}}]_{\textrm{g}_{j}}^{-1}\|_{L^{\infty-}}\\
		& \lesssim(\lambda_{j}\lambda_{i}^{-1})^{\frac{1}{2}-}\|\psi\|_{L^{\infty}}=o_{t\to T}(1)\cdot\|\psi\|_{L^{\infty}}.
	\end{align*}
	If $\lambda_{j}\sim\lambda_{i}$, thanks to \eqref{eq:induction translation},
	\begin{align*}
		(Q,[\psi[Q]_{\textrm{g}_{i}}]_{\textrm{g}_{j}}^{-1})_{r} & \lesssim\|\psi\|_{L^{\infty}}{\int_{\bbR}}Q(y)Q(y+\tfrac{x_{j}-x_{i}}{\lambda_{i}})dy\\
		& \lesssim\tfrac{\lambda_{i}}{x_{j}-x_{i}}\cdot\|\psi\|_{L^{\infty}}=o_{t\to T}(1)\cdot\|\psi\|_{L^{\infty}}.
	\end{align*}
	We compute the second term, by using $\eps_{j}=\sum_{i=j+1}^{k}[Q]_{\textrm{g}_{i}}+\eps_{k}$,
	\begin{align*}
		(\psi[Q]_{\textrm{g}_{j}},\eps_{k})_{r} & =(\psi[Q]_{\textrm{g}_{j}},\eps_{j})_{r}-{\sum_{i=j+1}^{k}}(\psi[Q]_{\textrm{g}_{j}},[Q]_{\textrm{g}_{i}})_{r}\\
		& =(\psi[Q]_{\textrm{g}_{j}},\eps_{j})_{r}-o_{t\to T}(1)\cdot\|\psi\|_{L^{\infty}}.
	\end{align*}
	We decompose 
	\begin{align}
		(\psi[Q]_{\textrm{g}_{j}},\eps_{j})_{r}= & {\sum_{i=1}^{j}}(\psi[Q]_{\textrm{g}_{j}},[\chi_{R}\wt{\eps}_{i}]_{\textrm{g}_{i}})_{r}\label{eq:no return gr interaction second 1}\\
		& +(\psi[Q]_{\textrm{g}_{j}},[\wt{\eps}_{j}]_{\textrm{g}_{j}})_{r}.\label{eq:no return gr interaction second 2}
	\end{align}
	We have 
	\begin{align*}
		\eqref{eq:no return gr interaction second 1}\lesssim{\sum_{i=1}^{j}}\|\chi_{R}\wt{\eps}_{i}\|_{L^{2}}\cdot\|\psi\|_{L^{\infty}}\lesssim{\sum_{i=1}^{j}}\lambda_{i}\cdot\|\psi\|_{L^{\infty}}=o_{t\to T}(1)\cdot\|\psi\|_{L^{\infty}}.
	\end{align*}
	For, \eqref{eq:no return gr interaction second 2}, we have 
	\begin{align*}
		(\psi[Q]_{\textrm{g}_{j}},[\wt{\eps}_{j}]_{\textrm{g}_{j}})_{r}\leq|(Q,|\wt{\eps}_{j}|)_{r}|\cdot\|\psi\|_{L^{\infty}}\lesssim\lambda_{j}^{\frac{1}{2}-}\cdot\|\psi\|_{L^{\infty}}=o_{t\to T}(1)\cdot\|\psi\|_{L^{\infty}}.
	\end{align*}
	Therefore, we arrive at $\eqref{eq:no return gr interaction}=o_{t\to T}(1)\cdot\|\psi\|_{L^{\infty}}$,
	and conclude 
	\begin{align*}
		{\int_{\bbR}}\psi|v|^{2}={\sum_{j=1}^{k}}{\int_{\bbR}}\psi|[Q]_{\textrm{g}_{j}}|^{2}+{\int_{\bbR}}\psi|\eps_{k}|^{2}+o_{t\to T}(1)\cdot\|\psi\|_{L^{\infty}}.
	\end{align*}
\end{proof}
Now, we prove the initial induction step. 
\begin{proof}[Proof of Lemma~\ref{lem:Induction initial}]
	\ref{state:further decomposition} was explained in \eqref{eq:4intro decompose},
	\eqref{eq:v(t) blowup rate lambda 1}, and \eqref{eq:4intro nonlinear coercivity}.
	Since there is only one soliton for now, it only remains to prove
	the convergence of translation parameter $x_{1}(t)$,~\ref{state:translation}.
	In fact, this can be shown by a similar argument in \cite{Raphael2005MathAnnalen}.
	Instead, we present a different argument that will be employed in
	later induction step. First, we claim 
	\begin{align}
		\sup_{t}|x_{1}(t)|<\infty.\label{eq:No rapid x1 claim}
	\end{align}
	Suppose not. Then, there exists a sequence $t_{n}\to T$ such that
	$x_{1}(t_{n})\to\pm\infty$. Without loss of generality, we may assume
	$\lim_{n\to\infty}x_{k}(t_{n})=+\infty$.  We suppress the dependence on the time sequence $(t_n)$ when no confusion arises. Denote $\chi_{|x-c|\leq\td R}\coloneqq\chi(\frac{x-c}{\td R})$. Then, we have $|\partial_{x}(x\chi_{|x-c|\leq\td R})|\lesssim1+\tfrac{|c|}{\td R}$.
	According to \eqref{eq:nonnegative energy}, we have 
	\begin{align}
		\partial_{t}{\int_{\bbR}}\chi_{|x-c|\leq\td R}x|v|^{2}dx=O(\sqrt{M_{0}E_{0}}(1+\tfrac{|c|}{\td R})).\label{eq:x2 dt bound}
	\end{align}
	Integrating \eqref{eq:x2 dt bound} on $[0,t_{n}]$, and taking $c=x_{1}(t_{n}),\td R=|x_{1}(t_{n})|/3$,
	we have 
	\begin{align}
		\sup_{n}\bigg|{\int_{\bbR}}\chi_{|x-x_{1}|\leq\frac{|x_{1}|}{3}}x|v|^{2}(t_{n})dx\bigg|<C\sqrt{M_{0}E_{0}}.\label{eq:x1 centre mass bdd}
	\end{align}
	Using \eqref{eq:4intro decompose} with the notation $\eps_{1}=[\wt{\eps}_{1}]_{\textrm{g}_{1}}$
	in \eqref{eq:induction k radiation decom}, we compute 
	\begin{align}
		{\int_{\bbR}}\chi_{|x-x_{1}|\leq\frac{|x_{1}|}{3}}x|v|^{2}= & {\int_{\bbR}}\chi_{|x-x_{1}|\leq\frac{|x_{1}|}{3}}x|[Q]_{\textrm{g}_{1}}|^{2}\label{eq:initial step sol}\\
		& +2{\int_{\bbR}}\chi_{|x-x_{1}|\leq\frac{|x_{1}|}{3}}x\Re([Q]_{\textrm{g}_{1}}\ol{\eps}_{1})\label{eq:initial step interact}\\
		& +{\int_{\bbR}}\chi_{|x-x_{1}|\leq\frac{|x_{1}|}{3}}x|\eps_{1}|^{2}.\nonumber 
	\end{align}
	By a direct computation, we have 
	\begin{align}
		\eqref{eq:initial step sol}=2{\int_{\bbR}}\chi_{\frac{x_{1}}{3\lambda_{1}}}\frac{\lambda_{1}y+x_{1}}{1+y^{2}}dy=2\pi x_{1}-O(\tfrac{\lambda_{1}}{x_{1}}).\label{eq:initial step sol compute}
	\end{align}
	Thanks to \eqref{eq:4intro nonlinear coercivity}, we have $\|Q\wt{\eps}_{1}\|_{L^{2}}\lesssim_{M_{0},E_{0}}\lambda_{1}$,
	and this implies 
	\begin{align}
		\eqref{eq:initial step interact}\lesssim|x_{1}|\cdot{\int_{\bbR}}Q\wt{\eps}_{1}\lesssim|x_{1}|\cdot\|Q^{\frac{1}{2}+}\|_{L^{2}}\|Q^{\frac{1}{2}-}\wt{\eps}_{1}\|_{L^{2}}\lesssim|x_{1}|\cdot\lambda_{1}^{\frac{1}{2}-}.\label{eq:initial step interact compute}
	\end{align}
	Thus, by \eqref{eq:initial step sol compute} and \eqref{eq:initial step interact compute}
	with $\lambda_{1}\to0$ and $x_{1}\to+\infty$, we have 
	\begin{align*}
		{\int_{\bbR}}\chi_{|x-x_{1}|\leq\frac{|x_{1}|}{3}}x|v|^{2}=(2\pi-o_{n\to\infty}(1))x_{1}-o_{n\to\infty}(1)+{\int_{\bbR}}\chi_{|x-x_{1}|\leq\frac{|x_{1}|}{3}}x|\eps_{1}|^{2}.
	\end{align*}
	Since ${\int_{\bbR}}\chi_{|x-x_{1}|\leq\frac{|x_{1}|}{3}}x|\eps_{1}|^{2}\sim x_{1}\cdot{\int_{\bbR}}\chi_{|x-x_{1}|\leq\frac{|x_{1}|}{3}}|\eps_{1}|^{2}\geq0$,
	we deduce 
	\begin{align*}
		{\int_{\bbR}}\chi_{|x-x_{1}|\leq\frac{|x_{1}|}{3}}x|v|^{2}\to+\infty,
	\end{align*}
	which contradicts to \eqref{eq:x1 centre mass bdd}. Therefore, we
	have \eqref{eq:No rapid x1 claim}.
	
	Next, we finish showing~\ref{state:translation}, that is, $\lim_{t\to T}x_{1}(t)\eqqcolon x_{1}(T)$
	and $|x_{1}(T)|<\infty$. We assume that $\lim_{t\to T}x_{1}(t)$
	does not exist. Then, there exist $\{a_{n}\}_{n\in\bbN},\allowbreak\{b_{n}\}_{n\in\bbN}$
	such that $a_{n},b_{n}\to T$ and $x_{1}(a_{n})\to x_{1,a}$, $x_{1}(b_{n})\to x_{1,b}$
	with $x_{1,a}\neq x_{1,b}$, and $|x_{1,a}|$, $|x_{1,b}|<C$ due
	to \eqref{eq:No rapid x1 claim}. By symmetry and taking translation,
	we may assume $0=x_{1,a}<x_{1,b}$. We claim that there exist a uniform
	constant $c=c(M_{0},E_{0})$ such that 
	\begin{equation}
		0<c\le x_{1,b}.\label{eq:uniform gap bound of x_k}
	\end{equation}
	To show a contradiction, we observe the exterior mass, $I_{r}(t)\coloneqq\int\varphi_{r}|v|^{2}(t)$
	with $0<r\ll x_{1,b}$. Integrating from $a_{n}$ to $b_{n}$ by \eqref{eq:nonnegative energy},
	we have 
	\begin{align*}
		|I_{r}(b_{n})-I_{r}(a_{n})|\lesssim|b_{n}-a_{n}|\cdot r^{-1}=o_{n\to\infty}(1)\cdot r^{-1}.
	\end{align*}
	By a similar argument to \eqref{eq:initial step interact compute},
	we have $([Q]_{\textrm{g}_{1}},\eps_{1})_{r}=(Q,\wt{\eps}_{1})_{r}=o_{t\to T}(1)$.
	Thus, we have 
	\begin{align*}
		I_{r}(b_{n})-I_{r}(a_{n})= & {\int_{\bbR}}\varphi_{r}|[Q]_{\textrm{g}_{1}}|^{2}(b_{n})-{\int_{\bbR}}\varphi_{r}|[Q]_{\textrm{g}_{1}}|^{2}(a_{n})\\
		& +{\int_{\bbR}}\varphi_{r}|\eps_{1}|^{2}(b_{n})-{\int_{\bbR}}\varphi_{r}|\eps_{1}|^{2}(a_{n})+o_{n\to\infty}(1),
	\end{align*}
	and $\|v_{0}\|_{L^{2}}^{2}-2\pi=\|\eps_{1}(t)\|_{L^{2}}^{2}+o_{t\to T}(1)$.
	Moreover, since $x_{1,a}=0$, we have 
	\begin{equation*}
		{\int_{\bbR}}\varphi_{r}|[Q]_{\textrm{g}_{1}}|^{2}(a_{n})=o_{n\to\infty}(1),\quad \text{and thus}, \quad  {\int_{\bbR}}\varphi_{r}|[Q]_{\textrm{g}_{1}}|^{2}(b_{n})=2\pi+o_{n\to\infty}(1).
	\end{equation*}
	Thus, we derive 
	\begin{align}
		\left|2\pi-{\int_{\bbR}}\chi_{r}|\eps_{1}|^{2}(b_{n})+{\int_{\bbR}}\chi_{r}|\eps_{1}|^{2}(a_{n})\right|\leq o_{n\to\infty}(1)\cdot r^{-1}+o_{n\to\infty}(1).\label{eq:x1 differ gap}
	\end{align}
	For $\int_{\bbR}\chi_{r}|\eps_{1}|^{2}$, we have 
	\begin{align}
		{\int_{\bbR}}\chi_{r}|[\wt{\eps}_{1}]_{\lambda_{1},\gamma_{1},x_{1}}|^{2}={\int_{\bbR}}\chi_{|y+\frac{x_{1}}{\lambda_{1}}|\lesssim\frac{r}{\lambda_{1}}}|\wt{\eps}_{1}|^{2} & \leq{\int_{\bbR}}{\textbf 1}_{[\frac{-x_{1}-Cr}{\lambda_{1}},\frac{-x_{1}+Cr}{\lambda_{k}}]}|\wt{\eps}_{1}|^{2}\nonumber \\
		& \lesssim\frac{x_{1}^{2}+r^{2}}{\lambda_{1}^{2}}{\int_{\bbR}}|Q\wt{\eps}_{1}|^{2}\lesssim_{M_{0},E_{0}}x_{1}^{2}+r^{2}.\label{eq:eps 1 inner estimate 2}
	\end{align}
	Therefore, by \eqref{eq:eps 1 inner estimate 2}, we have 
	\begin{align}
		\begin{split}{\int_{\bbR}}\chi_{r}|\eps_{1}|^{2}(b_{n}) & \leq C_{M_{0},E_{0}}(x_{1}(b_{n})^{2}+r^{2})+o_{n\to\infty}(1),\\
			{\int_{\bbR}}\chi_{r}|\eps_{1}|^{2}(a_{n}) & \leq C_{M_{0},E_{0}}r^{2}+o_{n\to\infty}(1).
		\end{split}
		\label{eq:eps 1 inner estimate 3}
	\end{align}
	From \eqref{eq:x1 differ gap} and \eqref{eq:eps 1 inner estimate 3},
	we deduce 
	\begin{align*}
		2\pi\leq2C_{M_{0},E_{0}}(x_{1}(b_{n})^{2}+r^{2})+o_{n\to\infty}(1)\cdot(1+r^{-1}).
	\end{align*}
	Thus, choose $r$ arbitrarily small and taking $n\to\infty$, we have
	$\frac{\pi}{C_{M_{0},E_{0}}}=:c^{2}\le x_{1,b}^{2}$. This prove \eqref{eq:uniform gap bound of x_k}.
	Once we have \eqref{eq:uniform gap bound of x_k}, then we can find
	another sequence $(t_{n})$ such that $x_{1}(t_{n})\to\td x_{1}$
	for any $\td x_{1}\in(x_{1,a},x_{1,b})=(0,x_{1,b})$, since $x_{1}(t)$
	is $C^{1}$. This obviously makes a contradiction by reapplying the
	uniform gap \eqref{eq:uniform gap bound of x_k} of $\lim_{n\to\infty}x_{1}(a_{n})$
	and $\lim_{n\to\infty}x_{1}(t_{n})$. Hence, we conclude that $\lim_{t\to T}x_{1}(t)=x_{1}(T)$
	exists, and we finish the proof. 
\end{proof}
\begin{proof}[Proof of Lemma~\ref{lem:Induction}]
	
	\textbf{Step 1.} ($k$-th sequential decomposition) We assume that
	$P(k-1)$ and $Q(k-1)$ hold true. We first claim $k$-th time-sequential
	decomposition. From $Q(k-1)$, we have 
	\begin{align}
		\liminf_{t\to T}\frac{\lambda_{k-1}(t)}{\|\varphi_{R}\wt{\eps}_{k-1}(t)\|_{\dot{H}^{1}}}=0.\label{eq:induction lambda k-1 assume}
	\end{align}
	and so there is a time sequence $t_{n}$ so that $\lim_{n\to\infty}\frac{\lambda_{k-1}(t_{n})}{\|\varphi_{R}\wt{\eps}_{k-1}(t_{n})\|_{\dot{H}^{1}}}=0$.
	Then, \eqref{eq:induction k nonlinear energy} in $P(k-1)$ implies,
	for large $n$, 
	\begin{align*}
		\sqrt{E(\varphi_{R}\wt{\eps}_{k-1}(t_{n}))}\lesssim_{M_{0},E_{0}}\lambda_{k-1}(t_{n})\ll\|\varphi_{R}\wt{\eps}_{k-1}(t_{n})\|_{\dot{H}^{1}}.
	\end{align*}
	So, we apply the decomposition (Proposition~\ref{prop:Decomposition})
	for $\{\varphi_{R}\wt{\eps}_{k-1}(t_{n})\}$, and then there exist
	$(\lambda_{k,n},\wt{\gamma}_{k,n},\wt x_{k,n},\wt{\eps}_{k,n})$ such
	that 
	\begin{align}
		\varphi_{R}\wt{\eps}_{k-1}(t_{n})=[Q+\wt{\eps}_{k,n}]_{\wt{\lambda}_{k,n},\wt{\gamma}_{k,n},\wt x_{k,n}}\label{eq:induction pf kth decom}
	\end{align}
	with the smallness $\|\wt{\eps}_{k,n}\|_{\dot{\calH}^{1}}<\eta$,
	\begin{align}
		\wt{\lambda}_{k,n}\sim\frac{\|Q\|_{\dot{H}^{1}}}{\|\varphi_{R}\wt{\eps}_{k-1}(t_{n})\|_{\dot{H}^{1}}},\label{eq:induction pf kth lambda}
	\end{align}
	and $(\wt{\eps}_{k,n},\mathcal{Z}_{1})_{r}=(\wt{\eps}_{k,n},\mathcal{Z}_{2})_{r}=(\wt{\eps}_{k,n},\mathcal{Z}_{3})_{r}=0$.
	Moreover, by \eqref{eq:Nolinear estimate truncated}, we have 
	\begin{align}
		\|\wt{\eps}_{k,n}\|_{\dot{\calH}_{R}^{1}}^{2}+E(\varphi_{R}\wt{\eps}_{k,n})\lesssim_{M_{0}}\wt{\lambda}_{k,n}^{2}E(\varphi_{R}\wt{\eps}_{k-1}(t_{n})).\label{eq:induction pf kth nonlinear energy}
	\end{align}
	We denote $(\wt{\lambda}_{k,n},\wt{\gamma}_{k,n},\wt x_{k,n},\wt{\eps}_{k,n})$
	by $(\wt{\lambda}_{k}(t_{n}),\wt{\gamma}_{k}(t_{n}),\wt x_{k}(t_{n}),\wt{\eps}_{k}(t_{n}))$,
	and we just denote by $(\wt{\lambda}_{k},\wt{\gamma}_{k},\wt x_{k},\wt{\eps}_{k})=(\wt{\textrm{g}}_{k},\wt{\eps}_{k})$
	if there is no confusion. From $P(k-1)$, we have continuous in time
	$k-1$-th configuration 
	\begin{align*}
		v(t)={\sum_{j=1}^{k-1}}[Q]_{\textrm{g}_{j}}+\eps_{k-1}\text{ on }t\in[T_{k-1},T).
	\end{align*}
	Define $(\lambda_{k},\gamma_{k},x_{k},\eps_{k})(t_{n})=(\textrm{g}_{k},\eps_{k})(t_{n})$
	via \eqref{eq:induction k radiation decom} and \eqref{eq:induction k modul para}.
	Using \eqref{eq:induction k radiation decom} and \eqref{eq:induction pf kth decom},
	we deduce 
	\begin{align*}
		\eps_{k-1}(t_{n})=[Q]_{\lambda_{k},\gamma_{k},x_{k}}(t_{n})+\eps_{k}(t_{n})=[Q]_{\textrm{g}_{k}}(t_{n})+\eps_{k}(t_{n}).
	\end{align*}
	Therefore, we have the sequential $k$-th decomposition 
	\begin{align}
		v(t_{n})=\bigg({\sum_{j=1}^{k}}[Q]_{\textrm{g}_{j}}+\eps_{k}\bigg)(t_{n}),\quad\|\wt{\eps}_{k}(t_{n})\|_{\dot{\calH}_{R}^{1}}\lesssim\lambda_{k}(t_{n}).\label{eq:Soliton nonradial second}
	\end{align}
	Note that the second inequality of \eqref{eq:Soliton nonradial second}
	comes from \eqref{eq:induction pf kth nonlinear energy} and \eqref{eq:induction k nonlinear energy}
	in $P(k-1)$. In addition, we obtain $\lambda_{k}(t_{n})\to0$ as
	$n\to\infty$ by \eqref{eq:induction lambda k-1 assume}, \eqref{eq:induction pf kth lambda},
	and the definition of $\lambda_{k}(t_{n})$.
	
	\textbf{Step 2.} (Boundedness of translation) Here we show a partial
	information of translation parameter $x_{k}(t)$,~\ref{state:translation}
	(of $P(k)$). Indeed, we claim that 
	\begin{align}
		\sup_{n}|x_{k}(t_{n})|<\infty.\label{eq:No rapid x2 claim}
	\end{align}
	The proof of \eqref{eq:No rapid x2 claim} follows a similar argument
	to \eqref{eq:No rapid x1 claim}. So here we only sketch it. From
	$P(k-1)$, $x_{j}(t_{n})$ converges to $x_{j}(T)$, and $|x_{j}(T)|<\infty$
	for all $j=1,2\cdots,k-1$. If \eqref{eq:No rapid x2 claim} fails,
	passing to a subsequence (again denoting it by $t_{n}$), we have
	$x_{k}(t_{n})\to\pm\infty$. Without loss of generality, we may assume
	$\lim_{n\to\infty}x_{k}(t_{n})=+\infty$. By a similar argument to
	\eqref{eq:x1 centre mass bdd}, taking $c=x_{k}(t_{n}),\td R=|x_{k}(t_{n})|/3$,
	we have 
	\begin{align}
		\sup_{n}\bigg|{\int_{\bbR}}\chi_{|x-x_{k}|\leq\frac{|x_{k}|}{3}}x|v|^{2}(t_{n})dx\bigg|<C\sqrt{M_{0}E_{0}}.\label{eq:x2 centre mass bdd}
	\end{align}
	On the other hand, by \eqref{eq:soliton decoupling}, we have 
	\begin{align*}
		{\int_{\bbR}}\chi_{|x-c|\leq\td R}x|v|^{2}= & {\sum_{j=1}^{k-1}}{\int_{\bbR}}\chi_{|x-c|\leq\td R}x|[Q]_{\textrm{g}_{j}}|^{2}+{\int_{\bbR}}\chi_{|x-c|\leq\td R}x|\eps_{k-1}|^{2}\\
		& +o_{t\to T}(1)\cdot(1+\tfrac{|c|}{R}).
	\end{align*}
	Therefore, taking $c=x_{1}(t_{n}),\td R=|x_{1}(t_{n})|/3$, and $t=t_{n}$,
	we derive 
	\begin{align}
		{\int_{\bbR}}\chi_{|x-x_{k}|\leq\frac{|x_{k}|}{3}}x|v|^{2}= & {\int_{\bbR}}\chi_{|x-x_{k}|\leq\frac{|x_{k}|}{3}}x{\sum_{j=1}^{k}}|[Q]_{\textrm{g}_{j}}|^{2}\label{eq:induction pf soliton L2}\\
		& +2{\int_{\bbR}}\chi_{|x-x_{k}|\leq\frac{|x_{k}|}{3}}x\Re([Q]_{\textrm{g}_{k}}\ol{\eps}_{k})\label{eq:induction pf soliton radiation interaction}\\
		& +{\int_{\bbR}}\chi_{|x-x_{k}|\leq\frac{|x_{k}|}{3}}x|\eps_{k}|^{2}+o_{t\to T}(1).\label{eq:induction pf soliton radiation}
	\end{align}
	We compute \eqref{eq:induction pf soliton L2}. For $j<k$, we have
	\begin{align}
		{\int_{\bbR}}\chi_{|x-x_{k}|\leq\frac{x_{k}}{3}}x|[Q]_{\textrm{g}_{j}}|^{2}(t_{n})dx & \sim{\int_{\bbR}}{\textbf 1}_{|\lambda_{j}y+x_{j}-x_{k}|\leq\frac{2x_{k}}{3}}\frac{\lambda_{j}y+x_{j}}{1+y^{2}}dy\nonumber \\
		& \lesssim \frac{\lambda_{j}^{2}}{|x_{k}|}\cdot{\int_{\bbR}}{\textbf 1}_{|\lambda_{j}y+x_{j}-x_{k}|\leq\frac{2x_{k}}{3}}dy\nonumber \\
		& \lesssim\lambda_{j}\to0.\label{eq:induction step2 soliton j<k}
	\end{align}
	Here, we used $y\sim x_{k}\lambda_{j}^{-1}$ which comes from $|x_{j}(T)|<\infty$
	and $x_{k}(t_{n})\to\infty$. For $j=k$, we have 
	\begin{align}
		{\int_{\bbR}}\chi_{|x-x_{k}|\leq\frac{x_{k}}{3}}x|[Q]_{\textrm{g}_{k}}|^{2}dx=2{\int_{\bbR}}\chi_{\frac{x_{k}}{3\lambda_{k}}}\frac{\lambda_{j}y+x_{j}}{1+y^{2}}dy=2\pi x_{k}-O(\tfrac{\lambda_{k}}{x_{k}}).\label{eq:induction step2 soliton j equ k}
	\end{align}
	For \eqref{eq:induction pf soliton radiation interaction}, we remind
	$\eps_{k}=\sum_{j=1}^{k-1}[\chi_{R}\wt{\eps}_{j}]_{\textrm{g}_{j}}+[\wt{\eps}_{k}]_{\textrm{g}_{k}}$
	and $\|Q\wt{\eps}_{j}\|_{L^{2}}\lesssim\lambda_{j}$. We have 
	\begin{align}
		\eqref{eq:induction pf soliton radiation interaction} & \lesssim|x_{k}|\cdot\bigg({\sum_{j=1}^{k}}|([Q]_{\textrm{g}_{k}},[\chi_{R}\wt{\eps}_{j}]_{\textrm{g}_{j}})_{r}|+\left|([Q]_{\textrm{g}_{k}},[\wt{\eps}_{k}]_{\textrm{g}_{k}})_{r}\right|\bigg)\nonumber \\
		& \lesssim|x_{k}|\cdot\bigg({\sum_{j=1}^{k-1}}\|Q\|_{L^{2}}\|\chi_{R}\wt{\eps}_{j}\|_{L^{2}}+{\int_{\bbR}}|Q\wt{\eps}_{k}|\bigg)\nonumber \\
		& \lesssim_{R}|x_{k}|\bigg({\sum_{j=1}^{k-1}}\lambda_{j}+\lambda_{k}^{\frac{1}{2}-}\bigg)=o_{n\to\infty}(1)\cdot|x_{k}|.\label{eq:induction step2 soliton radiate inter}
	\end{align}
	Here, we note that 
	\begin{align*}
		\|\chi_{R}\wt{\eps}_{j}\|_{L^{2}}\lesssim_{R}\|Q\wt{\eps}_{j}\|_{L^{2}}\lesssim\lambda_{j},\quad{\int_{\bbR}}|Q\wt{\eps}_{k}|\lesssim\|Q^{\frac{1}{2}+}\|_{L^{2}}\|Q^{\frac{1}{2}-}\wt{\eps}_{k}\|_{L^{2}}\lesssim\lambda_{k}^{\frac{1}{2}-}.
	\end{align*}
	By \eqref{eq:induction step2 soliton j<k}, \eqref{eq:induction step2 soliton j equ k},
	and \eqref{eq:induction step2 soliton radiate inter}, we get to 
	\begin{align*}
		{\int_{\bbR}}\chi_{|x-x_{k}|\leq\frac{|x_{k}|}{3}}x|v|^{2}=(2\pi-o_{n\to\infty}(1))x_{k}-o_{n\to\infty}(1)+\eqref{eq:induction pf soliton radiation}.
	\end{align*}
	For the last term \eqref{eq:induction pf soliton radiation}, since
	$x\geq0$ on $x\in[\frac{x_{k}}{3},\frac{5x_{k}}{3}]$, we have $\eqref{eq:induction pf soliton radiation}\geq0+o_{n\to\infty}(1)$.
	Therefore, we deduce 
	\begin{align*}
		{\int_{\bbR}}\chi_{|x-x_{k}|\leq\frac{x_{k}}{3}}x|v|^{2}(t_{n})dx\geq x_{k}(t_{n})\to\infty,
	\end{align*}
	and this contradicts to \eqref{eq:x2 centre mass bdd}. This finishes
	to prove \eqref{eq:No rapid x2 claim}.
	
	\textbf{Step 3.} (No return property) In this step, we show~\ref{state:no-return property}.
	This part is crucial to upgrade the sequential soliton resolution
	to continuous in time resolution. In this step, we use the convergence
	of $x_{j}$ for $j\le k-1$ (\ref{state:translation} in $P(j)$).
	Suppose 
	\begin{align*}
		\limsup_{t\to T}\frac{\lambda_{k-1}(t)}{\|\varphi_{R}\wt{\eps}_{k-1}(t)\|_{\dot{H}^{1}}}>0.
	\end{align*}
	Then, we can find a subsequence of $\{t_{n}\}_{n\in\bbN}$ which is
	denoted by $\{a_{n}\}_{n\in\bbN}$, and another sequence $\{b_{n}\}_{n\in\bbN}$
	such that $a_{n},b_{n}\to T$ and 
	\begin{align}
		\frac{\lambda_{k-1}(a_{n})}{\|\varphi_{R}\wt{\eps}_{k-1}(a_{n})\|_{\dot{H}^{1}}}\to0,\quad\frac{\lambda_{k-1}(b_{n})}{\|\varphi_{R}\wt{\eps}_{k-1}(b_{n})\|_{\dot{H}^{1}}}\to c>0.\label{eq:no return assume}
	\end{align}
	As explained above, along the sequence of $(a_{n})$, we can extract
	another soliton $[Q]_{\mathrm{g_{k}}}$, while there is no other soliton
	along $(b_{n})$. However, the nonnegativity of energy implies that
	the exterior mass of the last soliton is Lipschitz in time. This will
	make a contradiction.
	
	Indeed, passing a subsequence, we have $x_{k}(a_{n})\to x_{k,\infty}$
	with $|x_{k,\infty}|<\infty$ by the Step 1. Taking translation, we
	may assume $x_{k,\infty}=0$. In order to bring out a contradiction,
	we investigate the exterior mass, $I_{r}(t)=\int\varphi_{r}|v|^{2}$
	for a sufficiently small $r=r(M_{0},E_{0},c)$ to be chosen later.
	Thanks to \eqref{eq:nonnegative energy}, we have 
	\begin{align*}
		|\partial_{t}I_{r}(t)|\lesssim\sqrt{E_{0}}\|(\partial_{x}\varphi_{r})\cdot v\|_{L^{2}}\lesssim\sqrt{M_{0}E_{0}}r^{-1}.
	\end{align*}
	Integrating this, we deduce 
	\begin{align}
		|I_{r}(b)-I_{r}(a)|\lesssim_{M_{0},E_{0}}|b-a|r^{-1},\quad\forall a,b<T.\label{eq:no return gr lipschitz}
	\end{align}
	Now, we estimate $I_{r}(t)$ on each sequence $(a_{n})$ and $(b_{n})$.
	From \eqref{eq:soliton decoupling}, we have 
	\begin{align}
		I_{r}(t)={\sum_{j=1}^{k-1}}{\int_{\bbR}}\varphi_{r}|[Q]_{\textrm{g}_{j}}|^{2}+{\int_{\bbR}}\varphi_{r}|\eps_{k-1}|^{2}+o_{t\to T}(1)\cdot r^{-1},\label{eq:no return gr bn}
	\end{align}
	for any $t\ge0$. Firstly, we use \eqref{eq:no return gr bn} for
	$t=a_{n}$ to obtain \eqref{eq:no return gr an}. Indeed, we take
	a further decompose by $\eps_{k-1}=[Q]_{\textrm{g}_{k}}+\eps_{k}$using
	Proposition~\ref{prop:Decomposition}, and then have 
	\begin{align*}
		{\int_{\bbR}}\varphi_{r}|\eps_{k-1}|^{2}(a_{n})={\int_{\bbR}}\varphi_{r}|[Q]_{\textrm{g}_{k}}|^{2}+2(\varphi_{r}[Q]_{\textrm{g}_{k}},\eps_{k})_{r}+{\int_{\bbR}}\varphi_{r}|\eps_{k}|^{2}.
	\end{align*}
	We also estimate $(\varphi_{r}[Q]_{\textrm{g}_{k}},\eps_{k})_{r}\lesssim\lambda_{k}^{\frac{1}{2}-}=o_{n\to\infty}(1)$
	by a similar argument to \eqref{eq:induction pf soliton radiation interaction}.
	Then we derive 
	\begin{align*}
		I_{r}(a_{n})={\sum_{j=1}^{k}}{\int_{\bbR}}\varphi_{r}|[Q]_{\textrm{g}_{j}}|^{2}(a_{n})+{\int_{\bbR}}\varphi_{r}|\eps_{k}|^{2}(a_{n})+o_{n\to\infty}(1)\cdot r^{-1}.
	\end{align*}
	Moreover, since $x_{k}(a_{n})\to0$, after taking large $n$ so that
	$|x_{k}(a_{n})|<\frac{r}{2}$, we have 
	\begin{align*}
		{\int_{\bbR}}\varphi_{r}|[Q]_{\textrm{g}_{k}}|^{2}(a_{n})\leq{\int_{\bbR}}\varphi_{\frac{r}{2\lambda_{k}}}Q^{2}\lesssim\lambda_{k}(a_{n})r^{-1}.
	\end{align*}
	Therefore, we arrive at 
	\begin{align}
		I_{r}(a_{n})={\sum_{j=1}^{k-1}}{\int_{\bbR}}\varphi_{r}|[Q]_{\textrm{g}_{j}}|^{2}(a_{n})+{\int_{\bbR}}\varphi_{r}|\eps_{k}|^{2}(a_{n})+o_{n\to\infty}(1)\cdot r^{-1}.\label{eq:no return gr an}
	\end{align}
	On the sequence $t=b_{n}$, from \eqref{eq:no return assume}, we
	know $\lambda_{k-1}(b_{n})\sim_{c}\|\varphi_{R}\wt{\eps}_{k-1}(b_{n})\|_{\dot{H}^{1}}$.
	Thanks to energy bubbling \eqref{eq:induction k nonlinear energy},
	we also know $\|\wt{\eps}_{k-1}\|_{\dot{\calH}_{R}^{1}}\lesssim_{M_{0},E_{0}}\lambda_{k-1}$,
	and hence we have 
	\begin{align*}
		\|\wt{\eps}_{k-1}(b_{n})\|_{\dot{\calH}^{1}}\leq\|\wt{\eps}_{k-1}(b_{n})\|_{\dot{\calH}_{R}^{1}}+\|\varphi_{R}\wt{\eps}_{k-1}(b_{n})\|_{\dot{H}^{1}}\lesssim_{M_{0},E_{0},c}\lambda_{k-1}(b_{n}).
	\end{align*}
	From this and energy bubbling $\|\wt{\eps}_{j}\|_{\dot{\calH}_{R}^{1}}\lesssim_{M_{0},E_{0}}\lambda_{j}$,
	we estimate 
	\begin{align}
		\|\eps_{k-1}(b_{n})\|_{H^{1}} & \lesssim_{c}\|v_{0}\|_{L^{2}}+{\sum_{j=1}^{k-2}}\frac{1}{\lambda_{j}(b_{n})}\|\chi_{R}\wt{\eps}_{j}(b_{n})\|_{\dot{H}^{1}}+\frac{1}{\lambda_{k-1}(b_{n})}\|\wt{\eps}_{k-1}(b_{n})\|_{\dot{H}^{1}}\nonumber \\
		& \lesssim_{M_{0},E_{0},c}1.\label{eq:no return e k-1 bound}
	\end{align}
	Now, we are ready to bring out the contradiction. From \eqref{eq:no return gr lipschitz}
	and $a_{n},b_{n}\to T$ as $n\to\infty$, we have 
	\begin{align*}
		|I_{r}(b_{n})-I_{r}(a_{n})|\lesssim|b_{n}-a_{n}|r^{-1}=o_{n\to\infty}(1)\cdot r^{-1}.
	\end{align*}
	By this, \eqref{eq:no return gr bn}, and \eqref{eq:no return gr an},
	\begin{align}
		{\int_{\bbR}} & \varphi_{r}|\eps_{k-1}|^{2}(b_{n})-{\int_{\bbR}}\varphi_{r}|\eps_{k}|^{2}(a_{n})\nonumber \\
		\lesssim & {\sum_{j=1}^{k-1}}\left|{\int_{\bbR}}\varphi_{r}|[Q]_{\textrm{g}_{j}}|^{2}(b_{n})-|[Q]_{\textrm{g}_{j}}|^{2}(a_{n})\right|+o_{n\to\infty}(1)\cdot r^{-1}.\label{eq:no return soliton differ}
	\end{align}
	Using DCT, we have 
	\begin{align*}
		{\int_{\bbR}}\varphi_{r}|[Q]_{\textrm{g}_{j}}|^{2}={\int_{\bbR}}\varphi_{r}(\lambda_{j}y+x_{j})\cdot Q^{2}\to{\int_{\bbR}}\varphi_{r}(x_{j}(T))\cdot Q^{2}\text{ as }t\to T.
	\end{align*}
	Since $a_{n},b_{n}\to T$ as $n\to\infty$, we deduce $\eqref{eq:no return soliton differ}=o_{n\to\infty}(1)\cdot r^{-1}$.
	This is the reason what we need the convergence of $x_{j}(t)$ as
	$t\to T$. So, we have 
	\begin{align}
		{\int_{\bbR}}\varphi_{r}|\eps_{k-1}|^{2}(b_{n})-{\int_{\bbR}}\varphi_{r}|\eps_{k}|^{2}(a_{n})\lesssim o_{n\to\infty}(1)\cdot r^{-1}.\label{eq:no return gr differ}
	\end{align}
	On the other hand, again from \eqref{eq:soliton decoupling} with
	$\psi\equiv1$, we have 
	\begin{align*}
		\|v_{0}\|_{L^{2}}^{2}=\|v(t)\|_{L^{2}}^{2}=(k-1)\cdot2\pi+\|\eps_{k-1}(t)\|_{L^{2}}^{2}+o_{t\to T}(1).
	\end{align*}
	On $a_{n}$, by arguing as like \eqref{eq:no return gr an}, we have
	\begin{align*}
		\|v_{0}\|_{L^{2}}^{2}=k\cdot2\pi+\|\eps_{k}(a_{n})\|_{L^{2}}^{2}+o_{n\to\infty}(1).
	\end{align*}
	Thus, we have 
	\begin{align}
		{\int_{\bbR}}|\eps_{k-1}(b_{n})|^{2}-{\int_{\bbR}}|\eps_{k}(a_{n})|^{2}=2\pi+o_{n\to\infty}(1).\label{eq:no return mass differ}
	\end{align}
	By \eqref{eq:no return gr differ} and \eqref{eq:no return mass differ},
	we have 
	\begin{align*}
		\bigg|-{\int_{\bbR}}\chi_{r}|\eps_{k-1}|^{2}(b_{n})+{\int_{\bbR}}\chi_{r}|\eps_{k}|^{2}(a_{n})+2\pi\bigg|\lesssim o_{n\to\infty}(1)\cdot r^{-1}.
	\end{align*}
	Thanks to \eqref{eq:no return e k-1 bound}, we have 
	\begin{align*}
		{\int_{\bbR}}\chi_{r}|\eps_{k-1}|^{2}(b_{n})\leq\|\chi_{r}\|_{L^{1}}\|\eps_{k-1}(b_{n})\|_{L^{\infty}}^{2}\lesssim r\|\eps_{k-1}(b_{n})\|_{H^{1}}^{2}\lesssim_{M_{0},E_{0},c}r.
	\end{align*}
	Therefore, we can take $r$ sufficiently small so that 
	\begin{align*}
		\sup_{n}{\int_{\bbR}}\chi_{r}|\eps_{k-1}|^{2}(b_{n})\leq\epsilon,
	\end{align*}
	for given small $\epsilon>0$. We remark that the choice of $r$ does
	not depend on $n$. Thus, we obtain 
	\begin{align*}
		2\pi\leq{\int_{\bbR}}\chi_{r}|\eps_{k}|^{2}(a_{n})+2\pi\lesssim_{M_{0},E_{0},c}o_{n\to\infty}(1)\cdot r^{-1}+\epsilon.
	\end{align*}
	Now, taking $n\to\infty$, we have $2\pi\lesssim_{M_{0},E_{0},c}\epsilon$
	for arbitrary small $\epsilon$, which leads a contradiction. Hence,
	we conclude $\limsup_{t\to T}\frac{\lambda_{k-1}(t)}{\|\varphi_{R}\wt{\eps}_{k-1}(t)\|_{\dot{H}^{1}}}=0$,
	and this proves~\ref{state:no-return property}.
	
	\textbf{Step 4.} We finish to prove~\ref{state:further decomposition}
	and~\ref{state:soliton decoupling}. By Step 3, we have \eqref{eq:no return}.
	This means that by \eqref{eq:induction k nonlinear energy} there
	exists a $T_{k}$ with $T_{k-1}<T_{k}<T$ satisfying 
	\begin{align*}
		\sqrt{E(\varphi_{R}\wt{\eps}_{k-1}(t))}<\alpha^{*}\|\varphi_{R}\wt{\eps}_{k-1}(t)\|_{\dot{H}^{1}}\quad\text{on}\quad t\in[T_{k},T).
	\end{align*}
	Therefore, we can apply the decomposition Proposition~\ref{prop:Decomposition}
	for $\varphi_{R}\wt{\eps}_{k-1}(t)$ on $[T_{k},T)$, and we have
	\ref{state:further decomposition}. Now, we prove~\ref{state:soliton decoupling}.
	From Proposition~\ref{prop:Decomposition}, we have 
	\begin{align*}
		\wt{\lambda}_{k}(t)^{-1}\sim\|\varphi_{R}\wt{\eps}_{k-1}(t)\|_{\dot{H}^{1}},
	\end{align*}
	and deduce $\lim_{t\to T}\lambda_{k}(t)=0$ by~\ref{state:no-return property}.
	From $\|\wt{\eps}_{k-1}\|_{\dot{\calH}^{1}}<\eta$, we have $\wt{\lambda}_{k}^{-1}\sim\|\varphi_{R}\wt{\eps}_{k-1}\|_{\dot{H}^{1}}<\eta$,
	and deduce that $\lambda_{k-1}=\wt{\lambda}_{k}^{-1}\la_{k}\lesssim\lambda_{k}$,
	\eqref{eq:induction lambda}, and $\wt{\lambda}_{k}\not\to0$.
	
	Now, we prove \eqref{eq:induction k no same rate}. Suppose otherwise. Then,
	we have $\liminf_{t\to T}\wt{\lambda}_{k}<\infty$ and $\liminf_{t\to T}|\wt x_{k}|<\infty$, so there exists a sequence $\{t_{n}\}_{n\in\bbN}$ such that $\lim_{n\to\infty}\wt{\lambda}_{k}(t_{n})\eqqcolon\wt{\lambda}_{k,\infty}<\infty$ and $\lim_{n\to\infty}\wt x_{k}(t_{n})\eqqcolon\wt x_{k,\infty}$ with $|\wt x_{k,\infty}|<\infty$ as $t_{n}\to T$. 
	Further taking subsequence if necessary, $\lim_{n\to\infty}x_{k}(t_{n})=x_{k,\infty}$ exists, and by the step 2, we also have $|x_{k,\infty}|<\infty$.
	Let $f_{n}=[\varphi_{R}\wt{\eps}_{k-1}]_{\wt{\textrm{g}}_{k}}^{-1}(t_{n})$,
	and applying the variational argument to $f_{n}$ (Lemma~\ref{lem:variational argument}),
	we have 
	\begin{align}
		f_{n}\rightharpoonup[Q]_{\lambda_{k,0},\gamma_{k,0},x_{k,0}}\text{ weakly in }H^{1}\label{eq:pf induc para decouple fn}
	\end{align}
	for some fixed $(\lambda_{k,0},\gamma_{k,0},x_{k,0})$. We have 
	\begin{align}
		0 & =\lim_{n\to\infty}\Re{\int_{\bbR}}\varphi_{R}\wt{\eps}_{k-1}(t_{n})e^{-i\gamma_{k,0}}\cdot\chi_{\frac{R}{2}}e^{-i\wt{\gamma}_{k}(t_{n})}dx\nonumber \\
		& =\lim_{n\to\infty}\Re{\int_{\bbR}}f_{n}e^{-i\gamma_{k,0}}\cdot\chi_{|\wt{\lambda}_{k,\infty}x+\wt x_{k,\infty}|\leq\frac{R}{2}}dy\nonumber \\
		& =\Re{\int_{\bbR}}[Q]_{\lambda_{k,0},0,x_{k,0}}\cdot\chi_{|\wt{\lambda}_{k,\infty}x+\wt x_{k,\infty}|\leq\frac{R}{2}}dy\neq0,\label{eq:pf induc para decouple}
	\end{align}
	and this is a contradiction. So, we derive \eqref{eq:induction k no same rate}.
	
	Now, we show \eqref{eq:induction translation} by using induction
	in descending order for $j$. The initial case $j=k-1$ already was
	shown in \eqref{eq:induction k no same rate}. We assume that \eqref{eq:induction translation}
	holds true for $j=\ell+1,\ell+2,\cdots,k-1$. To show \eqref{eq:induction translation}
	for $j=l$, we use the contradiction argument. We suppose \eqref{eq:induction translation}
	is not true for $j=l$. Then, we can find a sequence $t_{n}\to T$
	such that $(x_{k}-x_{\ell})\lambda_{\ell}^{-1}(t_{n})\to x_{k,\ell,\infty}$
	with $|x_{k,\ell,\infty}|<\infty$ and $\lambda_{k}\lambda_{\ell}^{-1}(t_{n})\to\lambda_{k,\ell,\infty}$.
	We have $0<\lambda_{k,\ell,\infty}<\infty$ by the assumption $\lambda_{\ell}\sim\lambda_{k}$.
	We consider 
	\begin{align*}
		f(t_{n};k,\ell)\coloneqq[[[\varphi_{R}\wt{\eps}_{\ell}]_{\wt{\textrm{g}}_{\ell+1}}^{-1}]_{\wt{\textrm{g}}_{l+2}}^{-1}\cdots]_{\wt{\textrm{g}}_{k}}^{-1}(t_{n}).
	\end{align*}
	Using \eqref{eq:modulation induction}, we rewrite 
	\begin{align*}
		f(t_{n};k,\ell)=[\varphi_{R}\wt{\eps}_{\ell}]_{\textrm{g}_{k,\ell}}^{-1}(t_{n})=\bigg(\frac{\lambda_{k}}{\lambda_{\ell}}\bigg)^{\frac{1}{2}}e^{-i(\gamma_{k}-\gamma_{\ell})}(\varphi_{R}\wt{\eps}_{\ell})\bigg(\frac{\lambda_{k}}{\lambda_{\ell}}\cdot+\frac{x_{k}-x_{\ell}}{\lambda_{\ell}}\bigg)(t_{n}).
	\end{align*}
	Note that $\textrm{g}_{k,\ell}$ is given by \eqref{eq: g_i,j def}.
	On the other hand, from the decomposition $[\varphi_{R}\wt{\eps}_{i}]_{\wt{\textrm{g}}_{i+1}}^{-1}=Q+\wt{\eps}_{i+1}$,
	we have 
	\begin{align*}
		f(t_{n};k,\ell)={\sum_{i=\ell+1}^{k-1}}[Q+\chi_{R}\wt{\eps}_{i}]_{\textrm{g}_{k,i}}^{-1}(t_{n})+f_{n},
	\end{align*}
	where $f_{n}=[\varphi_{R}\wt{\eps}_{k-1}]_{\wt{\textrm{g}}_{k}}^{-1}(t_{n})$
	as given above. By the induction hypothesis with $t_{n}\to T$, we
	have $|(x_{k}-x_{i})\lambda_{i}^{-1}|\to\infty$ for $\ell+1\leq i\leq k-1$,
	and 
	\begin{align}
		{\sum_{i=\ell+1}^{k-1}}[Q]_{\textrm{g}_{k,i}}^{-1}(t_{n})\rightharpoonup0\text{ weakly in }H^{1}.\label{eq:pf induc soliton trans weak zero}
	\end{align}
	Moreover, from \eqref{eq:induction k nonlinear energy} and the assumption
	$\lambda_{\ell}\sim\lambda_{k}$ with \eqref{eq:induction lambda},
	we have 
	\begin{align}
		\|[\chi_{R}\wt{\eps}_{i}]_{\textrm{g}_{k,i}}^{-1}\|_{H^{1}}\lesssim\lambda_{i}\to0.\label{eq:pf induc radiation zero}
	\end{align}
	Gathering \eqref{eq:pf induc para decouple fn} \eqref{eq:pf induc soliton trans weak zero},
	and \eqref{eq:pf induc radiation zero}, we arrive at 
	\begin{align*}
		f(t_{n};k,\ell)\rightharpoonup[Q]_{\lambda_{k,0},\gamma_{k,0},x_{k,0}}\text{ weakly in }H^{1}.
	\end{align*}
	Therefore, as in \eqref{eq:pf induc para decouple}, we have 
	\begin{align*}
		0 & =\lim_{n\to\infty}\Re{\int_{\bbR}}\varphi_{R}\wt{\eps}_{\ell}(t_{n})e^{-i\gamma_{k,0}}\cdot\chi_{\frac{R}{2}}e^{-i(\gamma_{k}-\gamma_{\ell})(t_{n})}dx\\
		& =\lim_{n\to\infty}\Re{\int_{\bbR}}f(t_{n};k,\ell)e^{-i\gamma_{k,0}}\cdot\chi_{|\lambda_{k,\ell,\infty}x+x_{k,\ell,\infty}|\leq\frac{R}{2}}dy\\
		& =\Re{\int_{\bbR}}[Q]_{\lambda_{k,0},0,x_{k,0}}\cdot\chi_{|\lambda_{k,\ell,\infty}x+x_{k,\ell,\infty}|\leq\frac{R}{2}}dy\neq0,
	\end{align*}
	which lead a contradiction. Hence, \eqref{eq:induction translation}
	holds true for $j=l$. This finishes the proof of \eqref{eq:induction translation}.
	
	\textbf{Step 5.} It remains to show~\ref{state:translation}. The proof
	of~\ref{state:translation} is also similar to that of~\ref{state:translation}
	of $P(1)$ in Lemma~\ref{lem:Induction initial}.
	
	Recall that we have shown that $\limsup_{t\to T}|x_{k}(t)|<\infty$
	in Step 2. So, it suffices to show $x_{k}(t)$ converges as $t\to T$.
	Suppose not. Then, there exist $\{a_{n}\}_{n\in\bbN},\{b_{n}\}_{n\in\bbN}$
	so that $a_{n},b_{n}\to T$ and $x_{k}(a_{n})\to x_{k,a}$, $x_{k}(b_{n})\to x_{k,b}$
	with $x_{k,a}\neq x_{k,b}$. As in the proof of Lemma~\ref{lem:Induction initial},
	we can show the uniform gap \eqref{eq:uniform gap bound of x_k} of
	$0=x_{k,a}<c\le x_{k,b}$. Indeed, arguing as above, we have 
	\begin{align*}
		|I_{r}(b_{n})-I_{r}(a_{n})| & =o_{n\to\infty}(1)\cdot r^{-1}\\
		& =\bigg|2\pi+{\int_{\bbR}}\varphi_{r}|\eps_{k}|^{2}(b_{n})-{\int_{\bbR}}\varphi_{r}|\eps_{k}|^{2}(a_{n})\bigg|+o_{n\to\infty}(1).
	\end{align*}
	From $\|v_{0}\|_{L^{2}}^{2}-k\cdot2\pi=\|\eps_{k}(t)\|_{L^{2}}^{2}+o_{t\to T}(1)$, it follows that $\left|\int_{\bbR}|\eps_{k}|^{2}(b_{n})-|\eps_{k}|^{2}(a_{n})\right|\allowbreak=o_{n\to\infty}(1)$.
	Thus, we have 
	\begin{align}
		\left|2\pi-{\int_{\bbR}}\chi_{r}|\eps_{k}|^{2}(b_{n})+{\int_{\bbR}}\chi_{r}|\eps_{k}|^{2}(a_{n})\right|\leq o_{n\to\infty}(1)\cdot(1+r^{-1}).\label{eq:x differ gap}
	\end{align}
	Now, we estimate $\int_{\bbR}\chi_{r}|\eps_{k}|^{2}$. We have 
	\begin{align}
		{\int_{\bbR}}\chi_{r}|\eps_{k}|^{2}\lesssim{\sum_{j=1}^{k-1}}{\int_{\bbR}}\chi_{r}|[\chi_{R}\wt{\eps}_{j}]_{\textrm{g}_{j}}|^{2}+{\int_{\bbR}}\chi_{r}|[\wt{\eps}_{k}]_{\textrm{g}_{k}}|^{2}\lesssim{\sum_{j=1}^{k-1}}\lambda_{j}^{2}+{\int_{\bbR}}\chi_{r}|[\wt{\eps}_{k}]_{\textrm{g}_{k}}|^{2}.\label{eq:eps k inner estimate 1}
	\end{align}
	We also have 
	\begin{align}
		{\int_{\bbR}}\chi_{r}|[\wt{\eps}_{k}]_{\textrm{g}_{k}}|^{2}={\int_{\bbR}}\chi_{|y+\frac{x_{k}}{\lambda_{k}}|\lesssim\frac{r}{\lambda_{k}}}|\wt{\eps}_{k}|^{2}\lesssim\frac{x_{k}^{2}+r^{2}}{\lambda_{k}^{2}}{\int_{\bbR}}|Q\wt{\eps}_{k}|^{2}\lesssim_{M_{0},E_{0}}x_{k}^{2}+r^{2}.\label{eq:eps k inner estimate 2}
	\end{align}
	Therefore, by \eqref{eq:x differ gap}, \eqref{eq:eps k inner estimate 1},
	and \eqref{eq:eps k inner estimate 2}, we have 
	\begin{align*}
		2\pi\leq2C_{M_{0},E_{0}}(x_{k}(b_{n})^{2}+r^{2})+o_{n\to\infty}(1)\cdot(1+r^{-1}),
	\end{align*}
	and taking $r$ small and $n\to\infty,$there is a uniform bound on
	$x_{k,b}$ such as $0<c<x_{k,b}$. As in the proof of Lemma~\ref{lem:Induction initial},
	we can conclude~\ref{state:translation}. This finishes the proof. 
\end{proof}
Next, we show that the induction has to stop in finite steps, since
at each step, the mass drops by soliton mass $2\pi$. 
\begin{proof}[Proof of Lemma~\ref{lem:Halting induction}]
	Let $N_{0}\in\bbN$ be such that $N_{0}\leq\frac{M_{0}}{M(Q)}<N_{0}+1$.
	Assume that for all $N\leq N_{0}$, $Q(N)$ is true. Then, by Lemma
	\ref{lem:Induction initial} and~\ref{lem:Induction}, we have $v(t)={\sum_{j=1}^{N}}[Q]_{\textrm{g}_{j}}+\eps_{N}$
	for all $N\leq N_{0}$, and 
	\begin{align}
		\liminf_{t\to T}\frac{\lambda_{N_{0}}(t)}{\|\varphi_{R}\wt{\eps}_{N_{0}}(t)\|_{\dot{H}^{1}}}=0.\label{eq:Halting last 1}
	\end{align}
	Thanks to \eqref{eq:soliton decoupling} with $\psi\equiv1$, we have
	\begin{align*}
		M_{0}=\|v(t)\|_{L^{2}}^{2}=N_{0}M(Q)+\|\wt{\eps}_{N_{0}}\|_{L^{2}}^{2}+o_{t\to T}(1).
	\end{align*}
	There exist small $\epsilon>0$ and $T_{N_{0}}\leq T^{\prime}<T$
	such that for $T^{\prime}\leq t<T$, we have 
	\begin{align*}
		0\leq\|\varphi_{R}\wt{\eps}_{N_{0}}(t)\|_{L^{2}}^{2}<(\tfrac{M_{0}}{M(Q)}-N_{0})M(Q)+\epsilon<M(Q)=2\pi.
	\end{align*}
	By \eqref{eq:Halting last 1}, there exists a sequence of times $\{t_{n}\}_{n\in\bbN}\subset[T^{\prime},T)$ with $t_{n}\to T$ such that $\frac{\wt{\lambda}_{N_{0}}(t_{n})}{\|\varphi_{R}\wt{\eps}_{N_{0}}(t_{n})\|_{\dot{H}^{1}}}\to0$.
	Let $f_{n}$ be 
	\begin{align*}
		f_{n}\coloneqq[\varphi_{R}\wt{\eps}_{N_{0}}(t_{n})]_{\td{\lambda}_{N_{0},n},0,0},\quad\text{where}\quad\td{\lambda}_{N_{0},n}=\|\varphi_{R}\wt{\eps}_{N_{0}}(t_{n})\|_{\dot{H}^{1}}.
	\end{align*}
	Then, we have 
	\begin{align*}
		\sup_{n}\|f_{n}\|_{L^{2}}^{2}=\sup_{n}\|\varphi_{R}\wt{\eps}_{N_{0}}(t)\|_{L^{2}}^{2}<2\pi,\quad\|f_{n}\|_{\dot{H}^{1}}=1\quad\text{and}\quad E(f_{n})\to0.
	\end{align*}
	Now, applying Lemma~\ref{lem:variational argument}, we have $\liminf_{n\to\infty}\|\varphi_{R}\wt{\eps}_{N_{0}}(t)\|_{L^{2}}^{2}\geq2\pi$,
	and this is a contradiction. Thus, there exists a $1\leq N\leq N_{0}$
	such that $Q(N)^{c}$ is true.
	
	Now, we prove $\|\eps_{N}\|_{H^{1}}\lesssim1$. It suffices to show
	$\|\eps_{N}\|_{\dot{H}^{1}}\lesssim1$. Since $Q(N)^{c}$ is true,
	we have $\|\varphi_{R}\wt{\eps}_{N}(t)\|_{\dot{H}^{1}}\lesssim\la_{N}(t)$,
	and combining with \eqref{eq:induction k nonlinear energy}, we conclude
	$\|\eps_{N}(t)\|_{\dot{H}^{1}}\lesssim1$. 
\end{proof}
We now end this section with proving that there is no bubble tree,
\eqref{eq:no bubble eq in prop}. 
\begin{proof}[Proof of Proposition~\ref{prop:no bubble tree}]
	Since $\lambda_{i}\lesssim_{M_{0}}\lambda_{j}$ for $i<j$, it suffices
	to show $|(x_{i}-x_{j})\lambda_{j}^{-1}(t)|\to\infty$ as $t\to T$
	for $i<j$. We first claim that, for $1\leq k\leq N-1$, 
	\begin{align}
		\big|(x_{k}-x_{k+1})\lambda_{k+1}^{-1}(t)\big|\to\infty.\label{eq:no bubble tree index gap 1}
	\end{align}
	Suppose not. Then, there exists a time sequence $t_{n}\to T$ and
	a constant $|y_{\infty}|<\infty$ such that 
	\begin{align*}
		(x_{k}-x_{k+1})\lambda_{k+1}^{-1}(t_{n})\to y_{\infty}\text{ as }n\to\infty.
	\end{align*}
	We remark that $(x_{k}-x_{k+1})\lambda_{k+1}^{-1}=\wt x_{k+1}\wt{\lambda}_{k+1}^{-1}$
	and $1\lesssim_{M_{0}}\wt{\lambda}_{k+1}$. To reach a contradiction,
	we consider $\|Q\varphi_{R}\wt{\eps}_{k}\|_{L^{2}}$. Thanks to \eqref{eq:induction k nonlinear energy},
	we have 
	\begin{align}
		\|Q\varphi_{R}\wt{\eps}_{k}\|_{L^{2}}^{2}\lesssim\|Q\chi_{R}\wt{\eps}_{k}\|_{L^{2}}^{2}+\|Q\wt{\eps}_{k}\|_{L^{2}}^{2}\lesssim\lambda_{k}^{2}.\label{eq:wtepsk lambda bound}
	\end{align}
	From the decomposition $\varphi_{R}\wt{\eps}_{k}=[Q+\wt{\eps}_{k+1}]_{\wt{\textrm{g}}_{k+1}}$,
	we obtain 
	\begin{align*}
		\|Q\varphi_{R}\wt{\eps}_{k}\|_{L^{2}}^{2}=\int_{\bbR}Q^{2}|[Q+\wt{\eps}_{k+1}]_{\wt{\textrm{g}}_{k+1}}|^{2}dx.
	\end{align*}
	By renormalizing with $y=\frac{x-\wt x_{k+1}}{\wt{\lambda}_{k+1}}$,
	we have for large $n$, 
	\begin{align*}
		\|Q\varphi_{R}\wt{\eps}_{k}\|_{L^{2}}^{2}= & \int_{\bbR}\frac{2}{1+(\wt{\lambda}_{k+1}y+\wt x_{k+1})^{2}}|Q+\wt{\eps}_{k+1}|^{2}dy\\
		\gtrsim & _{M_{0}}\frac{1}{\wt{\lambda}_{k+1}^{2}(1+|y_{\infty}|)}\int_{|y|\leq1}|Q+\wt{\eps}_{k+1}|^{2}dy.
	\end{align*}
	Here, again according to \eqref{eq:induction k nonlinear energy},
	we have 
	\begin{align*}
		\int_{|y|\leq1}|Q\wt{\eps}_{k+1}|+|\wt{\eps}_{k+1}|^{2}dy\lesssim_{M_{0},E_{0}}\lambda_{k+1}+\lambda_{k+1}^{2}\to0.
	\end{align*}
	This implies that $\int_{|y|\leq1}|Q+\wt{\eps}_{k+1}|^{2}dy\sim\int_{|y|\leq1}Q^{2}dy>0$
	as $n\to\infty$. Thus, we derive 
	\begin{align*}
		\wt{\lambda}_{k+1}^{-1}(t_{n})\lesssim_{M_{0},E_{0},y_{\infty}}\lambda_{k}(t_{n}),
	\end{align*}
	for large $n$. However, from the definition of $\lambda_{k+1}$ and
	$\lambda_{k+1}\to0$ as $t_{n}\to T$, we have 
	\begin{align*}
		1\lesssim_{M_{0},E_{0},C}\lambda_{k}\wt{\lambda}_{k+1}=\lambda_{k+1}=o_{n\to\infty}(1),
	\end{align*}
	and this is a contradiction. Thus, we have \eqref{eq:no bubble tree index gap 1}.
	
	Now, we show the general case using an induction. We assume that $|(x_{i}-x_{j})\lambda_{j}^{-1}(t)|\to\infty$
	for all $1\leq i<j\leq k-1$. We will show that for $j=k$, 
	\begin{align}
		\big|(x_{i}-x_{k})\lambda_{k}^{-1}(t)\big|\to\infty\text{ for all }1\leq i<k.\label{eq:no bubble pf goal}
	\end{align}
	Again, we take further induction on $i=1,\cdots,k-1$ in descending
	order, as we already proved it for $i=k-1$ from \eqref{eq:no bubble tree index gap 1}.
	We assume that \eqref{eq:no bubble pf goal} is true for $\ell+1\leq i\leq k-1$.
	Then, we want to show that \eqref{eq:no bubble pf goal} holds for
	$i=\ell$. Suppose not, then we again find a sequence $t_{n}\to T$
	and a constant $|y_{\infty}|<\infty$ such that 
	\begin{align*}
		(x_{\ell}-x_{k})\lambda_{k}^{-1}(t_{n})\to y_{\infty}.
	\end{align*}
	In a similar manner to \eqref{eq:no bubble tree index gap 1}, we
	consider $\|Q\varphi_{R}\wt{\eps}_{\ell}\|_{L^{2}}$ to derive a contradiction.
	The estimate such as \eqref{eq:wtepsk lambda bound} also holds true.
	From the decomposition $\varphi_{R}\wt{\eps}_{j}=[Q+\wt{\eps}_{j+1}]_{\wt{\textrm{g}}_{j+1}}$
	for $\ell\leq j\leq k-1$, and \eqref{eq:modulation induction}, we
	have 
	\begin{align}
		\|Q\varphi_{R}\wt{\eps}_{\ell}\|_{L^{2}}^{2}={\int_{\bbR}}Q^{2}\bigg|{\sum_{j=\ell+1}^{k-1}}[Q+\chi_{R}\wt{\eps}_{j}]_{\textrm{g}_{j,\ell}}+[Q+\wt{\eps}_{k}]_{\textrm{g}_{k,\ell}}\bigg|^{2}dx.\label{eq:no bubble pf 1}
	\end{align}
	Here, we recall $\textrm{g}_{j,\ell}$ given in \eqref{eq: g_i,j def},
	and we note that 
	\begin{align*}
		[f]_{\textrm{g}_{j,\ell}}=\frac{e^{i(\gamma_{j}-\gamma_{\ell})}}{(\lambda_{j}\lambda_{\ell}^{-1})^{\frac{1}{2}}}f\bigg(\frac{\cdot-(x_{j}-x_{\ell})\lambda_{\ell}^{-1}}{\lambda_{j}\lambda_{\ell}^{-1}}\bigg),
	\end{align*}
	Substituting with $y=(x-(x_{k}-x_{\ell})\lambda_{\ell}^{-1})(\lambda_{k}\lambda_{\ell}^{-1})^{-1}$,
	we rewrite \eqref{eq:no bubble pf 1} as
	\begin{align}
		{\int_{\bbR}}\frac{2}{1+(\lambda_{k}\lambda_{\ell}^{-1}y+(x_{k}-x_{\ell})\lambda_{\ell}^{-1})^{2}}\bigg|Q+\wt{\eps}_{k}+{\sum_{j=\ell+1}^{k-1}}[Q+\chi_{R}\wt{\eps}_{j}]_{\textrm{g}_{k,j}}^{-1}\bigg|^{2}dy.\label{eq:k+l contradiction 1}
	\end{align}
	Here, we used that, for $\ell<j<k$, 
	\begin{align*}
		[[Q+\chi_{R}\wt{\eps}_{j}]_{\textrm{g}_{j,\ell}}]_{\textrm{g}_{k,\ell}}^{-1}=[Q+\chi_{R}\wt{\eps}_{j}]_{\textrm{g}_{k,j}}^{-1}
	\end{align*}
	which also can be shown by \eqref{eq:modulation induction}. We write
	\begin{align*}
		\lambda_{k}\lambda_{\ell}^{-1}y+(x_{k}-x_{\ell})\lambda_{\ell}^{-1}&=\lambda_{k}\lambda_{\ell}^{-1}(y+(x_{k}-x_{\ell})\lambda_{k}^{-1})
		\\
		&=\lambda_{k}\lambda_{\ell}^{-1}(y-y_{\infty}+o_{n\to\infty}(1)).
	\end{align*}
	So, on $|y|\leq1$, we again have, from $1\lesssim_{M_{0}}\lambda_{k}\lambda_{\ell}^{-1}$,
	\begin{align}
		\eqref{eq:k+l contradiction 1}\gtrsim_{M_{0},y_{\infty}} & (\lambda_{k}\lambda_{\ell}^{-1})^{-2}\bigg[\int_{|y|\leq1}Q^{2}dy-O\bigg(\int_{|y|\leq1}|\wt{\eps}_{k}|^{2}dy\bigg)\nonumber \\
		& -O\bigg(\int_{|y|\leq1}\sum_{j=\ell+1}^{k-1}|[Q+\chi_{R}\wt{\eps}_{j}]_{\textrm{g}_{k,j}}^{-1}|^{2}dy\bigg)\bigg].\label{eq:k+l contradiction 2}
	\end{align}
	We have $\|{\textbf 1}_{|y|\leq1}\wt{\eps}_{k}\|_{L^{2}}\lesssim_{M_{0},E_{0}}\lambda_{k}\to0$.
	In addition, we have $\|\chi_{R}\wt{\eps}_{j}\|_{L^{2}}\lesssim\lambda_{j}\to0$.
	Therefore, we reduce the right side of \eqref{eq:k+l contradiction 2}
	to 
	\begin{align*}
		\lambda_{k}^{-2}\lambda_{\ell}^{2}\bigg[\int_{|y|\leq1}Q^{2}dy-O\bigg(\sum_{j=\ell+1}^{k-1}\int_{|x-(x_{k}-x_{j})\lambda_{j}^{-1}|\leq\lambda_{k}\lambda_{j}^{-1}}Q^{2}dx\bigg)-o_{n\to\infty}(1)\bigg].
	\end{align*}
	By the induction hypothesis, $|(x_{k}-x_{j})\lambda_{j}^{-1}|\to\infty$
	for all $\ell+1\leq j\leq k-1$. Thus we have 
	\begin{align*}
		\int_{|x-(x_{k}-x_{j})\lambda_{j}^{-1}|\leq\lambda_{k}\lambda_{j}^{-1}}Q^{2}dx & =\int_{\lambda_{k}\lambda_{j}^{-1}(-1+(x_{k}-x_{j})\lambda_{k}^{-1})}^{\lambda_{k}\lambda_{j}^{-1}(1+(x_{k}-x_{j})\lambda_{k}^{-1})}Q^{2}dx\sim\bigg|\frac{\lambda_{k}}{x_{k}-x_{j}}\bigg|\to0.
	\end{align*}
	Thus, the right side of \eqref{eq:k+l contradiction 2} becomes 
	\begin{align*}
		\lambda_{k}^{-2}\lambda_{\ell}^{2}\bigg[\int_{|y|\leq1}Q^{2}dy-o_{n\to\infty}(1)\bigg],
	\end{align*}
	and this means that 
	\begin{align*}
		\lambda_{\ell}\gtrsim\|Q\varphi_{R}\wt{\eps}_{\ell}\|_{L^{2}}\gtrsim_{M_{0},E_{0},y_{\infty}}\lambda_{k}^{-1}\lambda_{\ell},
	\end{align*}
	and we obtain $1\lesssim\lambda_{k}=o_{n\to\infty}(1)$. This makes
	a contradiction. Therefore, we conclude that $|(x_{\ell}-x_{k})\lambda_{k}^{-1}|\to\infty$,
	and by the induction, we deduce \eqref{eq:no bubble pf goal}. This
	finishes the proof. 
\end{proof}

\section{Proof of Theorem~\ref{thm:Soliton resolution} and~\ref{thm:Soliton resolution gauged}}

\label{sec:prooffinish} In this section we complete the proof of
Theorem~\ref{thm:Soliton resolution} and~\ref{thm:Soliton resolution gauged}
based on the multi-soliton configuration derived in Section~\ref{sec:multisoliton}.
We first prove the blow-up case of Theorem~\ref{thm:Soliton resolution gauged}
and then use the pseudo-conformal transform to show the global solution
case. For this part, we require that $u(t)\in H^{1,1}$. Theorem~\ref{thm:Soliton resolution}
is proved using the gauge transform $\mathcal{G}$ and $\mathcal{G}^{-1}$
from Theorem~\ref{thm:Soliton resolution gauged}. 
\begin{proof}[Proof of Theorem~\ref{thm:Soliton resolution gauged}]
	
	For a finite-time blow-up solution $v(t)$ to \eqref{CMdnls-gauged},
	in Section~\ref{sec:multisoliton}, we have a multi-soliton configuration.
	That is, there exists an $N\geq1$ such that $P(N)$ and $Q(N)^{c}$
	are true. More precisely, there exists a $0<T_{N}<T$ such that for
	$t\in[T_{N},T)$, there exist $(\lambda_{j},\gamma_{j},x_{j},\eps_{j},\wt{\eps}_{j})=(\textrm{g}_{j},\eps_{j},\wt{\eps}_{j})$
	that satisfy the followings: 
	\begin{align}
		v=\sum_{j=1}^{k}[Q]_{\textrm{g}_{j}}+\eps_{k},\quad\eps_{k}=\sum_{j=1}^{k-1}[\chi_{R}\wt{\eps}_{j}]_{\textrm{g}_{j}}+[\wt{\eps}_{k}]_{\textrm{g}_{k}},\quad\text{for all }k\leq N,\label{eq:proof Outer decomposition}
	\end{align}
	and 
	\begin{align}
		\|\wt{\eps}_{k}\|_{\dot{\calH}_{R}^{1}}\lesssim\lambda_{k}\text{ for }k=1,2\cdots,N-1,\quad\|\wt{\eps}_{N}\|_{\dot{\calH}^{1}}\lesssim\lambda_{N}.\label{eq:proof outer radiation scale bound}
	\end{align}
	Moreover, we have 
	\begin{align}
		\sup_{1\leq k\leq N}\sup_{t}|x_{k}(t)|\leq C<\infty,\quad\lim_{t\to T}x_{j}(t)\eqqcolon x_{j}(T),\text{ exists}.\label{eq:proof outer translation bound}
	\end{align}
	
	To finish the proof for finite-time blow-up solutions, we remain to
	show that $\lambda_{N}\lesssim T-t$, $\eps_{N}(t)\to z^{\ast}$ in
	$L^{2}$, and $z^{\ast}$ satisfies $M(z^{\ast})=M(v_{0})-N\cdot M(Q)$
	and $\partial_{x}z^{\ast}\in L^{2}.$
	
	\textbf{Step 1.} \emph{Convergence of $\eps_{N}(t)$ and regularity
		of $z^{*}$.}
	
	Our first goal is to show that $\eps_{N}(t)$ converges to an asymptotic
	profile $z^{\ast}$ in $L^{2}$. This proof is motivated by \cite{MerleRaphael2005CMP}.
	We will truncate the outer region of each contracting soliton and
	show the convergence of $\eps_{N}(t)$ in the truncated outer region.
	Then, using $\norm{\eps_{N}(t)}_{H^{1}}\lesssim1$, we can conclude
	the convergence in $L^{2}(\R)$. Define a smooth cutoff away from
	the centers of solitons $x_{j}(T)$ for $1\leq j\leq N$, by 
	\begin{align*}
		\Phi_{r}(x)\coloneqq\prod_{j=1}^{N}\varphi_{r}(x-x_{j}(T)).
	\end{align*}
	Then, we have the outer convergence of the radiation. 
	\begin{lem}[Outer convergence]
		\label{lem:Outer convergence} There exists $z^{\ast}\in L^{2}$
		such that for any $r>0$, we have $\Phi_{r}\eps_{N}(t)\to\Phi_{r}z^{\ast}$
		in $L^{2}$ as $t\to T$. 
	\end{lem}
	
	\begin{proof}
		We first claim that for any $\delta_{1}>0$, there exist $T_{0}<T$
		and $\delta_{2}\in(0,T-T_{0})$ such that 
		\begin{align}
			\sup_{\tau\in(0,\delta_{2})}\sup_{t\in[T_{0},T-\tau)}\|\Phi_{r}(v(t+\tau)-v(t))\|_{L^{2}}<\delta_{1}.\label{eq:Outer nonlinear goal}
		\end{align}
		Denote $\td v^{\tau}(t)\coloneqq\Phi_{r}(v(t+\tau)-v(t))$. Then,
		we have 
		\begin{align*}
			(i\partial_{t}+\partial_{xx})\td v^{\tau}(t)=[\partial_{xx},\Phi_{r}]((v(t+\tau)-v(t))+\Phi_{r}(\calN(v(t+\tau))-\calN(v(t))).
		\end{align*}
		Here, $\calN(v)=\frac{1}{4}|v|^{4}v-v|D||v|^{2}$. From the Duhamel
		formula, we have for $t\in[T_{0},T)$, 
		\begin{align*}
			\|\td v^{\tau}(t)\|_{L^{2}}\leq\|\td v^{\tau}(T_{0})\|_{L^{2}}+2|T-T_{0}|\sup_{s\in[T_{0},T)}\|[\partial_{xx},\Phi_{r}]v(s)+\Phi_{r}\calN(v(s))\|_{L^{2}}.
		\end{align*}
		If we have 
		\begin{align}
			\sup_{s\in[T_{0},T)}\|[\partial_{xx},\Phi_{r}]v(s)+\Phi_{r}\calN(v(s))\|_{L^{2}}\lesssim_{r,M_{0},E_{0}}1,\label{eq:Outer nonlinear goal pre}
		\end{align}
		then we have 
		\begin{align*}
			\sup_{\tau\in(0,\delta_{2})}\sup_{t\in[T_{0},T-\tau)}\|\td v^{\tau}(t)\|_{L^{2}}\leq\sup_{\tau\in(0,\delta_{2})}\|\td v^{\tau}(T_{0})\|_{L^{2}}+C(r,M_{0},E_{0})|T-T_{0}|.
		\end{align*}
		So, taking $T_{0}$ sufficiently close to $T$ and then choosing $\delta_{2}>0$
		to be small, we deduce \eqref{eq:Outer nonlinear goal}. Here, we
		use the continuity of the flow $\tau\mapsto u(T_{0}+\tau)\in L^{2}$
		at $\tau=0$. Thus, we reduce \eqref{eq:Outer nonlinear goal} to
		\eqref{eq:Outer nonlinear goal pre}.
		
		Now, we prove \eqref{eq:Outer nonlinear goal pre}. If necessary,
		we take $T_{0}(r)=T_{0}<T$ larger so that 
		\begin{align}
			T_{N}<T_{0}<T,\quad|x_{j}(t)-x_{j}(T)|<\tfrac{r}{2}\text{ on }t\in[T_{0},T)\text{ for all }1\leq j\leq N.\label{eq:Outer nonlinear T0 def}
		\end{align}
		Using \eqref{eq:proof outer radiation scale bound}, we have 
		\begin{align}
			\|\Phi_{r}\partial_{x}v\|_{L^{2}} & \lesssim{\sum_{j=1}^{N}}\lambda_{j}^{-1}|\|{\textbf 1}_{|y|\gtrsim r\lambda_{j}^{-1}}Q_{y}\|_{L^{2}}+\|\partial_{x}\eps_{N}\|_{L^{2}}\nonumber \\
			& \lesssim{\sum_{j=1}^{N}}(\lambda_{j}^{\frac{1}{2}}+\lambda_{j}^{-1}\|\chi_{R}\wt{\eps}_{j}\|_{\dot{H}^{1}})+\lambda_{N}^{-1}\|\wt{\eps}_{N}\|_{\dot{H}^{1}}\lesssim_{r,M_{0},E_{0}}1.\label{eq: outer v L2 estimate}
		\end{align}
		Moreover, interpolating \eqref{eq: outer v L2 estimate} and mass
		conservation law, we have 
		\begin{align}
			\|\Phi_{r}v\|_{L^{\infty}}\lesssim_{r,M_{0},E_{0}}1.\label{eq: outer v Linf estimate}
		\end{align}
		For the commutator term, similar to \eqref{eq: outer v L2 estimate},
		we have 
		\begin{align}
			\|[\partial_{xx},\Phi_{r}]v\|_{L^{2}} & \lesssim{\sum_{j=1}^{N}}\left(\|(\partial_{x}\chi_{|x-x_{j}(T)|\leq r})\partial_{x}v\|_{L^{2}}+\|(\partial_{xx}\chi_{|x-x_{j}(T)|\leq r})v\|_{L^{2}}\right)\nonumber \\
			& \lesssim_{r,M_{0},E_{0}}1.\label{eq:Outer nonlinear goal commutator}
		\end{align}
		For the quadratic term in the nonlinear terms, using \eqref{eq: outer v Linf estimate},
		we have 
		\begin{align*}
			\|\Phi_{r}|v|^{4}v\|_{L^{2}}\lesssim_{r,M_{0},E_{0}}1.
		\end{align*}
		The estimate of the nonlocal part $v|D||v|^{2}$ is not as simple
		as above since the Hilbert transform $\mathcal{H}$ may interfere
		with the truncation $\Phi_{r}$. In fact, we need to look inside $v(t)$
		and use the multi-soliton configuration. We will use some special
		relation between $Q$ and $\mathcal{H}$. Moreover, we will use the
		commutation relation with $\mathcal{H}$. We first simplify by \eqref{eq: outer v Linf estimate},
		\begin{align*}
			\|\Phi_{r}v|D||v|^{2}\|_{L^{2}}\lesssim_{r,M_{0},E_{0}}\|\Phi_{r}|D||v|^{2}\|_{L^{2}},
		\end{align*}
		and use the decomposition 
		\begin{align*}
			|v|^{2}={\sum_{j=1}^{N}}|[Q]_{\textrm{g}_{j}}|^{2}+2{\sum_{i<j}^{N}}\Re([Q]_{\textrm{g}_{j}}\ol{[Q]_{\textrm{g}_{i}}})+2{\sum_{j=1}^{N}}\Re([Q]_{\textrm{g}_{j}}\ol{\eps}_{N})+|\eps_{N}|^{2}.
		\end{align*}
		Using the pointwise bound $||D|Q^{2}|\lesssim Q^{2}$ from $\mathcal{H}(Q^{2})=yQ^{2}$,
		we have 
		\begin{align*}
			\big\|\Phi_{r}|D|{\sum_{j=1}^{N}}|[Q]_{\textrm{g}_{j}}|^{2}\big\|_{L^{2}}\lesssim{\sum_{j=1}^{N}}\lambda_{j}^{-\frac{3}{2}}\big\|{\textbf 1}_{|y|\gtrsim r\lambda_{j}^{-1}}Q^{2}\big\|_{L^{2}}\lesssim_{r}1.
		\end{align*}
		Now, we estimate interaction terms, 
		\begin{align}
			\big\|\Phi_{r}|D|{\sum_{i<j}^{N}}\Re([Q]_{\textrm{g}_{j}}\ol{[Q]_{\textrm{g}_{i}}})\big\|_{L^{2}}.\label{eq:Outer two soliton goal}
		\end{align}
		It suffices to estimate 
		\begin{align}
			\big\|\Phi_{r}|D|\Re([Q]_{\textrm{g}_{j}}\ol{[Q]_{\textrm{g}_{i}}})\big\|_{L^{2}}\label{eq:Outer two soliton 1}
		\end{align}
		for $i<j$. Changing the variable with $\frac{x-x_{j}}{\lambda_{j}}=y$,
		we have 
		\begin{align*}
			\eqref{eq:Outer two soliton 1}\leq\lambda_{j}^{-\frac{3}{2}}\big\|{\textbf 1}_{|y|\gtrsim r\lambda_{j}^{-1}}{\textbf 1}_{|y-(x_{i}-x_{j})\lambda_{j}^{-1}|\gtrsim r\lambda_{j}^{-1}}|D|\Re(Q[Q]_{\textrm{g}_{i,j}})\big\|_{L^{2}}.
		\end{align*}
		We recall the definition of $\textrm{g}_{i,j}$, \eqref{eq: g_i,j def}.
		By \eqref{eq:CommuteHilbertDerivative}, we have 
		\begin{align}
			|D|(Q[Q]_{\textrm{g}_{i,j}})=y^{-1}\calH[y\partial_{y}(Q[Q]_{\textrm{g}_{i,j}})].\label{eq:Outer two soliton 2}
		\end{align}
		We decompose \eqref{eq:Outer two soliton 2} into 
		\begin{align}
			\eqref{eq:Outer two soliton 2}= & y^{-1}\calH[yQ_{y}[Q]_{\textrm{g}_{i,j}}]\label{eq:Outer two soliton 2-1}\\
			& +y^{-1}\calH[yQ\partial_{y}([Q]_{\textrm{g}_{i,j}})].\label{eq:Outer two soliton 2-2}
		\end{align}
		For \eqref{eq:Outer two soliton 2-1}, applying \eqref{eq:CommuteHilbert},
		we have 
		\begin{align*}
			\eqref{eq:Outer two soliton 2-1}=y^{-2}\calH[y^{2}Q_{y}[Q]_{\textrm{g}_{i,j}}]+\frac{1}{\pi y^{2}}{\int_{\bbR}}yQ_{y}[Q]_{\textrm{g}_{i,j}}dy.
		\end{align*}
		Using $|y|\gtrsim r\lambda_{j}^{-1}>0$ and Hölder inequality, we
		have 
		\begin{align*}
			\big\|{\textbf 1}_{|y|\gtrsim r\lambda_{j}^{-1}}\eqref{eq:Outer two soliton 2-1}\big\|_{L^{2}}\lesssim_{r}(\lambda_{j}^{2}\|y^{2}Q_{y}\|_{L^{\infty}}+\lambda_{j}^{\frac{3}{2}}\|yQ_{y}\|_{L^{2}})\|[Q]_{\textrm{g}_{i,j}}\|_{L^{2}}\lesssim\lambda_{j}^{\frac{3}{2}}.
		\end{align*}
		For \eqref{eq:Outer two soliton 2-2}, further using \eqref{eq:CommuteHilbert},
		we have 
		\begin{align}
			\eqref{eq:Outer two soliton 2-2}= & \frac{1}{y(y-(x_{i}-x_{j})\lambda_{j}^{-1})}\calH[yQ(y-(x_{i}-x_{j})\lambda_{j}^{-1})\partial_{y}([Q]_{\textrm{g}_{i,j}})]\label{eq:Outer two soliton 3}\\
			& +\frac{1}{y(y-(x_{i}-x_{j})\lambda_{j}^{-1})}{\int_{\bbR}}yQ\partial_{y}([Q]_{\textrm{g}_{i,j}})dy.\label{eq:Outer two soliton 4}
		\end{align}
		We first control \eqref{eq:Outer two soliton 4}. By integrating by
		parts and Hölder inequality, we have 
		\begin{align*}
			|\eqref{eq:Outer two soliton 4}|\leq\frac{1}{|y(y-(x_{i}-x_{j})\lambda_{j}^{-1})|}\|\partial_{y}(yQ)\|_{L^{2}}\|Q\|_{L^{2}}\lesssim\frac{1}{|y(y-(x_{i}-x_{j})\lambda_{j}^{-1})|}.
		\end{align*}
		Therefore, we have 
		\begin{align*}
			\big\|{\textbf 1}_{|y|\gtrsim r\lambda_{j}^{-1}}{\textbf 1}_{|y-(x_{i}-x_{j})\lambda_{j}^{-1}|\gtrsim r\lambda_{j}^{-1}}\eqref{eq:Outer two soliton 4}\big\|_{L^{2}}\lesssim_{r}\lambda_{j}^{\frac{3}{2}}.
		\end{align*}
		For \eqref{eq:Outer two soliton 3}, from $y\partial_{y}Q\in L^{2}$,
		we have 
		\begin{align*}
			&\big\|{\textbf 1}_{|y|\gtrsim r\lambda_{j}^{-1}}{\textbf 1}_{|y-(x_{i}-x_{j})\lambda_{j}^{-1}|\gtrsim r\lambda_{j}^{-1}}\eqref{eq:Outer two soliton 3}\big\|_{L^{2}} 
			\\& \lesssim_{r}\lambda_{j}^{2}\|yQ(y-(x_{i}-x_{j})\lambda_{j}^{-1})\partial_{y}([Q]_{\textrm{g}_{i,j}})\|_{L^{2}}\\
			& \lesssim\lambda_{j}^{2}\|[y\partial_{y}Q]_{\textrm{g}_{i,j}}\|_{L^{2}}\lesssim\lambda_{j}^{2}.
		\end{align*}
		Therefore, we deduce 
		\begin{align*}
			\eqref{eq:Outer two soliton goal}\lesssim_{r}{\sum_{i<j}^{N}}\lambda_{j}^{-\frac{3}{2}}(\lambda_{j}^{\frac{3}{2}}+\lambda_{j}^{2})\lesssim1.
		\end{align*}
		Next, the estimate of $2\sum_{j=1}^{N}\Re([Q]_{\textrm{g}_{j}}\ol{\eps}_{N})$
		is performed in a similar manner. Indeed, we have 
		\begin{align}
			\big\|\Phi_{r}|D|\Re([Q]_{\textrm{g}_{j}}\ol{\eps}_{N})\big\|_{L^{2}}\leq\lambda_{j}^{-\frac{3}{2}}\big\|{\textbf 1}_{|y|\gtrsim r\lambda_{j}^{-1}}|D|\Re(Q[\ol{\eps}_{N}]_{\textrm{g}_{j}}^{-1})\big\|_{L^{2}}.\label{eq:Outer two soliton 5}
		\end{align}
		Again thanks to \eqref{eq:CommuteHilbertDerivative}, we have 
		\begin{align*}
			\eqref{eq:Outer two soliton 5} & =\lambda_{j}^{-\frac{3}{2}}\big\|{\textbf 1}_{|y|\gtrsim r\lambda_{j}^{-1}}y^{-1}\calH[y\partial_{y}(\Re(Q[\ol{\eps}_{N}]_{\textrm{g}_{j}}^{-1})]\big\|_{L^{2}}\\
			& \lesssim_{r}\lambda_{j}^{-\frac{1}{2}}(\|yQ_{y}\|_{L^{2}}\|[\ol{\eps}_{N}]_{\textrm{g}_{j}}^{-1}\|_{L^{\infty}}+\|yQ\|_{L^{\infty}}\|\partial_{y}[\ol{\eps}_{N}]_{\textrm{g}_{j}}^{-1}\|_{L^{2}})\lesssim_{M_{0},E_{0}}1.
		\end{align*}
		Finally, $|\eps_{N}|^{2}$ is estimated as 
		\begin{align*}
			\|\Phi_{r}|D||\eps_{N}|^{2}\|_{L^{2}}\lesssim\|\partial_{x}\eps_{N}\|_{L^{2}}\|\eps_{N}\|_{L^{\infty}}\lesssim_{M_{0},E_{0}}1.
		\end{align*}
		Hence, we have 
		\begin{align}
			\sup_{s\in[t_{0},T)}\|\Phi_{r}\calN(v(s))\|_{L^{2}}\lesssim_{r,M_{0},E_{0}}1,\label{eq:Outer nonlinear goal nonlinear}
		\end{align}
		and then from \eqref{eq:Outer nonlinear goal commutator} and \eqref{eq:Outer nonlinear goal nonlinear},
		we are led to \eqref{eq:Outer nonlinear goal pre}.
		
		Now, let us finish the proof. We first claim that $\Phi_{r}\eps_{N}(t)$
		is Cauchy in $L^{2}$ for any $r>0$. We want to show that for any
		fixed $\delta_{1}>0$ and $r>0$, there exists $\delta_{2}\in(0,T-T_{0})$
		such that 
		\begin{align}
			\sup_{\tau\in(0,\delta_{2})}\sup_{t\in[T_{0},T-\tau)}\|\Phi_{r}(\eps_{N}(t+\tau)-\eps_{N}(t))\|_{L^{2}}<\delta_{1}.\label{eq:Outer L2 Cauchy}
		\end{align}
		By \eqref{eq:proof Outer decomposition} and \eqref{eq:Outer nonlinear goal},
		to prove \eqref{eq:Outer L2 Cauchy}, it suffices to show that $\|\Phi_{r}[Q]_{\textrm{g}_{j}}\|_{L^{2}}\to0$
		for all $j=1,2,\cdots,N$. Given the definition of $\Phi_{r}$, \eqref{eq:proof outer translation bound},
		\eqref{eq:Outer nonlinear T0 def}, and $\lambda_{j}(t)\to0$ as $t\to T$
		for all $j=1,2,\cdots,N$, we have $\|\Phi_{r}[Q]_{\textrm{g}_{j}}\|_{L^{2}}\to0$
		for all $j=1,2,\cdots,N$. Therefore, we conclude \eqref{eq:Outer L2 Cauchy}.
		That is, $\Phi_{r}\eps_{N}(t)$ is Cauchy in $L^{2}$ for any $r>0$.
		Thus, there exist $z_{r}^{\ast}$ for each $r>0$ such that $\Phi_{r}\eps_{N}(t)\to\Phi_{r}z_{r}^{\ast}$
		in $L^{2}$. By the uniqueness of the limit, we conclude that there
		exists $z^{\ast}$ such that $\Phi_{r}\eps_{N}(t)\to\Phi_{r}z^{\ast}$
		in $L^{2}$ for any $r>0$. Since $\|\eps_{N}(t)\|_{L^{2}}$ is uniformly
		bounded in $L^{2}$, we also have $z^{\ast}\in L^{2}$, and we finish
		the proof. 
	\end{proof}
	Now, we show that $\eps_{N}\to z^{\ast}$ in $L^{2}$. We first prove
	that $\eps_{N}\rightharpoonup z^{\ast}$ in $H^{1}$ for some $z^{\ast}\in H^{1}$.
	For any sequence $t_{n}\to T$, we know that $\|\eps_{N}(t_{n})\|_{H^{1}}$
	is bounded. Therefore, there exists a subseqeunce $t_{n^{\prime}}\to T$
	such that $\eps_{N}(t_{n}^{\prime})\rightharpoonup z_{w}^{\ast}$
	for some $z_{w}^{\ast}\in H^{1}$. Moreover, by the Rellich--Kondrachov
	theorem, for any $R>0$, we have $\Phi_{R^{-1}}\chi_{R}\eps_{N}(t_{n}^{\prime})\to\Phi_{R^{-1}}\chi_{R}z_{w}^{\ast}$
	in $L^{2}$. In view of Lemma~\ref{lem:Outer convergence} and the
	uniqueness of the limit, we have $\Phi_{R^{-1}}\chi_{R}\eps_{N}(t_{n}^{\prime})\to\Phi_{R^{-1}}\chi_{R}z^{\ast}$
	in $L^{2}$. Thus, by taking $R$ arbitrary large, we have $z^{\ast}=z_{w}^{\ast}$.
	
	We show further information of $z^{*}$. From \eqref{eq:soliton decoupling}
	with $\psi\equiv1$, we deduce $M(z^{\ast})=M(v_{0})-N\cdot M(Q)$.
	If we further assume $v_{0}\in H^{1,1}$, then from the virial identity,
	we have that $\|xv(t)\|_{L^{2}}$ is bounded for $t\in[0,T)$. Therefore,
	by Lemma~\ref{lem:Outer convergence} and Fatou's lemma, we estimate
	\begin{align*}
		\|(x-x_{j}(T))z^{\ast}\|_{L^{2}} & =\lim_{r\to0}\|(x-x_{j}(T))\Phi_{r}z^{\ast}\|\\
		& \leq\lim_{r\to0}\liminf_{t\to T}\|(x-x_{j}(T))\Phi_{r}v(t)\|_{L^{2}}\\
		& \leq\limsup_{t\to T}\|(x-x_{j}(T))\Phi_{r}v(t)\|_{L^{2}}\lesssim_{T,v_{0}}1,
	\end{align*}
	and we conclude $xz^{\ast}\in L^{2}$.
	
	\textbf{Step 2.} \emph{Pseudo-conformal bound $\lambda_{N}\lesssim T-t$}.
	
	Let $I_{r}(t)=\int_{\bbR}\Phi_{r}|u|^{2}$. From \eqref{eq:nonnegative energy},
	we deduce $|\partial_{t}I_{r}|\lesssim_{M_{0},E_{0}}r^{-1}$ and 
	\begin{align}
		|I_{r}(t)-I_{r}(\tau)|\lesssim_{M_{0},E_{0}}|t-\tau|r^{-1}.\label{eq:lambda pseudo conformal 1}
	\end{align}
	We have 
	\begin{align*}
		{\int_{\bbR}}\Phi_{r}|[Q]_{\textrm{g}_{j}}|^{2}\sim{\int_{\bbR}}{\textbf 1}_{|x-x_{j}|\geq\frac{r}{2}}|[Q]_{\textrm{g}_{j}}|^{2}\sim{\int_{\bbR}}\varphi_{r\lambda_{j}^{-1}}Q^{2}\sim\lambda_{j}r^{-1}.
	\end{align*}
	We also estimate for $j\le i\le N$, 
	\begin{align*}
		2(\Phi_{r}[Q]_{\textrm{g}_{j}},[Q]_{\textrm{g}_{i}})_{r} & =2({\textbf 1}_{|y|\gtrsim r\lambda_{j}^{-1}}Q,{\textbf 1}_{|y-(x_{i}-x_{j})\lambda_{j}^{-1}|\gtrsim r\lambda_{j}^{-1}}[Q]_{\textrm{g}_{i,j}})_{r}\\
		& \lesssim\|{\textbf 1}_{|y|\gtrsim r\lambda_{j}^{-1}}Q\|_{L^{2}}\lesssim(\lambda_{j}r^{-1})^{\frac{1}{2}}.
	\end{align*}
	Similarly, we have for $j<N$, 
	\begin{align*}
		2(\Phi_{r}[Q]_{\textrm{g}_{j}},\eps_{N})_{r}\lesssim\|{\textbf 1}_{|y|\gtrsim r\lambda_{j}^{-1}}Q\|_{L^{2}}\|\eps_{N}\|_{L^{2}}\lesssim_{M_{0}}(\lambda_{j}r^{-1})^{\frac{1}{2}}.
	\end{align*}
	Using the decomposition of $\eps_{N}$, we estimate 
	\begin{align*}
		2(\Phi_{r}[Q]_{\textrm{g}_{N}},\eps_{N})_{r}\lesssim{\sum_{k=1}^{N-1}}\|\chi_{R}\eps_{j}\|_{L^{2}}+|(Q,\wt{\eps}_{N})_{r}|\lesssim{\sum_{k=1}^{N-1}}\lambda_{j}+\lambda_{j}^{\frac{1}{2}-}=o_{t\to T}(1).
	\end{align*}
	Since $\eps_{N}(t)\to z^{\ast}$ in $L^{2}$ , we have 
	\begin{align*}
		{\int_{\bbR}}\Phi_{r}(|\eps_{N}|^{2}-|z^{\ast}|^{2})dx=o_{t\to T}(1).
	\end{align*}
	Using \eqref{eq:lambda pseudo conformal 1}, and then we obtain 
	\begin{align*}
		\bigg|I_{r}(t)-{\int_{\bbR}}\Phi_{r}|z^{\ast}|^{2}dx\bigg|\lesssim|T-t|r^{-1}.
	\end{align*}
	Now, we take $r=T-t$ and estimate 
	\begin{align}
		\bigg|{\sum_{j=1}^{N}}\frac{\lambda_{j}}{T-t}+O_{M_{0}}\bigg({\sum_{j=1}^{N-1}}\bigg(\frac{\lambda_{j}}{T-t}\bigg)^{\frac{1}{2}}\bigg)+o_{t\to T}(1)\bigg|\lesssim1.\label{eq:lambda pseudo conformal 2}
	\end{align}
	On the other hand, by Young's inequality, we obtain 
	\begin{align}
		{\sum_{j=1}^{N-1}}\frac{\lambda_{j}}{T-t}+O_{M_{0}}\bigg({\sum_{j=1}^{N-1}}\bigg(\frac{\lambda_{j}}{T-t}\bigg)^{\frac{1}{2}}\bigg)\geq-C_{M_{0}}.\label{eq:lambda pseudo conformal 2-1}
	\end{align}
	Hence, by \eqref{eq:lambda pseudo conformal 2} and \eqref{eq:lambda pseudo conformal 2-1}
	we have $\frac{\lambda_{N}}{T-t}-C_{M_{0}}+o_{t\to T}(1)\lesssim_{M_{0},E_{0}}1$,
	or equivalently $\la_{N}\lesssim C_{M_{0},E_{0}}(T-t)$. From $\la_{1}\lesssim\la_{2}\lesssim\cdots\lesssim\la_{N}$,
	we have $\la_{j}\lesssim T-t$ for all $j$. This finishes the proof
	of Theorem~\ref{thm:Soliton resolution gauged} for finite time blow-up
	solutions.
	
	\textbf{Step 3.} \emph{Global solution case.}
	
	We use the pseudo-conformal transform to extend results to a global
	solution in $H^{1,1}$. According to the local theory, if $v_{0}\in H^{1,1}$,
	then $v(t)\in H^{1,1}$ for its lifespan. Suppose that $v(t)$ is
	a solution to \eqref{CMdnls-gauged} on the time interval $[1,\infty)$.
	If $v(t)$ scatters forward in time, then there is nothing to prove.
	Suppose that $v(t)$ does not scatter forward in time. Denote $v_{c}(t)\coloneqq[\calC v](t)$.
	Then $v_{c}(t)$ is a solution to \eqref{CMdnls-gauged} on $[-1,0)$.
	We claim that $v_{c}(t)$ does not converge to $z_{c}^{\ast}$ in
	$L^{2}$ as $t\to0$. If $v_{c}(t)$ were to converge, then we would
	have 
	\begin{align*}
		\|v_{c}(t)-z_{c}^{\ast}\|_{L^{2}}=o_{t\to0}(1)=\|v_{c}(t)-e^{it\partial_{xx}}z_{c}^{\ast}\|_{L^{2}}.
	\end{align*}
	Taking the inverse of the pseudo-conformal transform (indeed, $\mathcal{C}^{-1}=\mathcal{C})$,
	we would derive that 
	\begin{align*}
		\|v(t)-e^{it\partial_{xx}}v^{\ast}\|_{L^{2}}=o_{t\to\infty}(1),\qquad v^{\ast}=\tfrac{1}{\sqrt{2\pi}}\calF(z_{c}^{\ast}),
	\end{align*}
	and this is a contradiction. Therefore, $v_{c}(t)$ does not converge
	in $L^{2}$. Let $T\geq0$ be the maximal forward time of existence
	of $v_{c}$. We claim that $T=0$, i.e. $v_{c}$ blows up at $t=0$.
	If $T>0$, then we have $v_{c}(t)\in C_{t}^{1}H^{1}([-1,0+]\times\bbR)$
	by a standard Cauchy theory with $v_{c}(-1)\in H^{1}$. Thus, we have
	$v_{c}(t)\to v_{c}(0)$ as $t\to0$ in $H^{1}$. This is a contradiction
	and thus $v_{c}$ blows up at $t=0$. We can apply Theorem~\ref{thm:Soliton resolution gauged}
	for the finite-time blow-up case. Hence, for some $1\leq N\leq\frac{M(v_{0})}{M(Q)}$,
	$v_{c}(t)$ has the decomposition 
	\begin{align*}
		v_{c}(t)-\sum_{j=1}^{N}[Q]_{\lambda_{j}(t),\gamma_{j}(t),x_{j}(t)}\to z^{\ast}\text{ in }L^{2}\text{ as }t\to0^{-},
	\end{align*}
	where the modulation parameters $(\lambda_{j}(t),\gamma_{j}(t),x_{j}(t))$
	and $z^{*}$ satisfy the properties stated in Theorem~\ref{thm:Soliton resolution gauged}.
	We can rewrite $z^{*}$ by $e^{it\partial_{xx}}z^{*}$ as 
	\begin{align*}
		v_{c}(t)-\sum_{j=1}^{N}[Q]_{\lambda_{j}(t),\gamma_{j}(t),x_{j}(t)}-e^{it\partial_{xx}}z^{\ast}\to0\text{ in }L^{2}\text{ as }t\to0^{-}.
	\end{align*}
	Taking the inverse of the pseudo-conformal transform (indeed, $\mathcal{C}^{-1}=\mathcal{C})$,
	we have 
	\begin{align}
		v(t)-e^{i\frac{x^{2}}{4t}}\sum_{j=1}^{N}[Q]_{\mathring{\lambda}_{j}(t),\mathring{\gamma}_{j}(t),\mathring{x}_{j}(t)}-e^{it\partial_{xx}}v^{\ast}\to0\text{ in }L^{2}\text{ as }t\to+\infty,\label{eq:sol resol global pre}
	\end{align}
	where $(\mathring{\lambda}_{j}(t),\mathring{\gamma}_{j}(t),\mathring{x}_{j}(t))\coloneqq(t\lambda_{j}(-t^{-1}),\gamma_{j}(-t^{-1}),tx(-t^{-1}))$
	and $v^{\ast}=\frac{\calF(z^{\ast})}{\sqrt{2\pi}}$. Here, we denote
	the velocity of soliton by $2c_{j}(t)\coloneqq\tfrac{1}{t}\mathring{x}_{j}(t)=x_{j}(-t^{-1})$,
	in accordance with Galilean transform notation. Moreover, from the
	properties stated in Theorem~\ref{thm:Soliton resolution gauged},
	we have $\mathring{\lambda}_{1}(t)\lesssim\mathring{\lambda}_{2}(t)\lesssim\cdots\lesssim\mathring{\lambda}_{N}(t)=t\lambda_{N}(-t^{-1})\lesssim1$,
	and 
	\begin{align*}
		\lim_{t\to+\infty}|t|\cdot\left|2\cdot\frac{c_{i}(t)-c_{j}(t)}{\mathring{\lambda}_{i}(t)}\right|=\lim_{t\to+\infty}\left|\frac{\mathring{x}_{i}(t)-\mathring{x}_{j}(t)}{\mathring{\lambda}_{i}(t)}\right|=\infty,
	\end{align*}
	for all $1\leq i\neq j\leq N$. Moreover, we have 
	\begin{align*}
		\|v(t)\|_{\dot{H}^{1}} & \sim\tfrac{1}{|t|}\|v_{c}(-t^{-1})\|_{\dot{H}^{1}}+O(\tfrac{1}{|t|}\|xv_{c}(-t^{-1})\|_{L^{2}})\\
		& \sim\mathring{\lambda}_{1}(t)^{-1}+O(\tfrac{1}{|t|}\|xv_{c}(-t^{-1})\|_{L^{2}}).
	\end{align*}
	Using $xv_{c}\in L^{2}$, and the virial identity \eqref{eq:Virial identity},
	we have $\limsup_{t\to0^{-}}\|xv_{c}(t)\|<\infty$. Thus, we obtain
	$\|v(t)\|_{\dot{H}^{1}}\sim\mathring{\lambda}_{1}(t)^{-1}$ as $t\to\infty$.
	
	Now, we rewrite $e^{i\frac{x^{2}}{4t}}[Q]_{\mathring{\lambda}_{j}(t),\mathring{\gamma}_{j}(t),2tc_{j}(t)}$ in \eqref{eq:sol resol global pre} into a canonical form via a Galilean boost. Set $y=\frac{x-2tc_{j}}{\mathring{\lambda}_{j}}$, and we compute
	\begin{align*}
		e^{i\frac{x^{2}}{4t}}[Q]_{\mathring{\lambda}_{j},\mathring{\gamma}_{j},2tc_{j}}(x)
		& =e^{i\frac{\mathring{\lambda}_{j}^{2}y^{2}}{4t}}e^{ic_{j}\cdot x-itc_{j}^{2}}\frac{e^{i\mathring{\gamma}_{j}}}{\mathring{\lambda}_{j}^{\frac{1}{2}}}Q\bigg(\frac{x-2tc_{j}}{\mathring{\lambda}_{j}}\bigg)\\
		& =e^{i\frac{\mathring{\lambda}_{j}^{2}y^{2}}{4t}}\textnormal{Gal}_{c_{j}}([Q]_{\mathring{\lambda}_{j}(t),\mathring{\gamma}_{j}(t),0})(x).
	\end{align*}
	Since $\mathring{\lambda}_{j}(t)\lesssim\mathring{\lambda}_{N}(t)\lesssim1$,
	and using DCT in the asymptotics as $t\to\infty$, we can replace
	the pseudo-conformal factor $e^{i\frac{\mathring{\lambda}_{j}^{2}y^{2}}{4t}}$
	with $1$. Finally, from $v^{\ast}=\frac{1}{\sqrt{2\pi}}\calF(z^{\ast})$
	and $\partial_{x}z^{\ast},xz^{\ast}\in L^{2}$, we have $\partial_{x}v^{\ast},xv^{\ast}\in L^{2}$.
	Note that $\lim_{t\to\infty}c_{j}(t)$ exists. On the blow-up side,
	$x_{j}(t)$ is a position parameter, but on the pseudo-conformal side,
	$c_{j}(t)=\frac{1}{2}x_{j}(-t^{-1})$ is a velocity of each soliton.
	We rename $\mathring{\lambda}_{j}(t),\mathring{\gamma}_{j}(t)$ as
	$\lambda_{j}(t),\gamma_{j}(t)$ in the theorem statement. This finishes
	the proof of global solution case. 
\end{proof}
Next, we provide the proof of Theorem~\ref{thm:Soliton resolution}.
This is accomplished using the gauge transform $\mathcal{G}$ and
its inverse $\mathcal{G}^{-1}$. 
\begin{proof}[Proof of Theorem~\ref{thm:Soliton resolution}]
	Since the gauge transform $\mathcal{G}$ is a fixed-time nonlinear transform
	and a diffeomorphism on $H^{1}(\R)$, it suffices to justify
	the multi-soliton configuration under $\mathcal{G}^{-1}$. In the
	procedure, since $\mathcal{G}^{-1}$ is a nonlinear transform, we
	need to handle the sum of solitons with care. Recall the inverse gauge
	transform, 
	\begin{align*}
		\mathcal{G}^{-1}(f)(x)=f(x)e^{\frac{i}{2}\int_{-\infty}^{x}|f(y)|^{2}dy}.
	\end{align*}
	In view of $v=\sum_{j=1}^{N}[Q]_{\mathrm{g}_{j}}+\eps_{N}$, and by
	the local mass decoupling \eqref{eq:soliton decoupling} with $\psi(y)=\textbf{1}_{(-\infty,x]}(y)$,
	we decompose the phase 
	\[
	\int_{-\infty}^{x}|v(y)|^{2}dy=\int_{-\infty}^{x}\sum_{\ell=1}^{N}|[Q]_{\textrm{g}_{\ell}}|^{2}+|\eps_{N}|^{2}dy+o_{t\to T}(1).
	\]
	Here, $o_{t\to T}(1)$ does not depend on $x$. Then, using $\mathcal{G}^{-1}([Q]_{\textrm{g}_{j}})=-[\calR]_{\textrm{g}_{j}}$
	we obtain 
	\begin{align}
		u(t,x)= & -\mathcal{G}^{-1}(v)\nonumber \\
		= & \sum_{j=1}^{N}[\calR]_{\textrm{g}_{j}}\exp\bigg(\frac{i}{2}\int_{-\infty}^{x}\sum_{\ell=1,\ell\ne j}^{N}|[Q]_{\textrm{g}_{\ell}}|^{2}+|\eps_{N}|^{2}dy+o_{t\to T}(1)\bigg)\label{eq:description u^*2}\\
		& -\mathcal{G}^{-1}(\eps_{N})\exp\bigg(\frac{i}{2}\int_{-\infty}^{x}\sum_{\ell=1}^{N}|[Q]_{\textrm{g}_{\ell}}|^{2}dy+o_{t\to T}(1)\bigg)\label{eq:description u^*1}
	\end{align}
	We now compute the asymptotics of the phase functions in \eqref{eq:description u^*2}
	and \eqref{eq:description u^*1}. From $\lambda_{j}(t)\to0$, $x_{j}(t)\to x_{j}(T)$,
	and $\eps_{N}(t)\to z^{\ast}$ in $L^{2}$, we have 
	\[
	{\frac{1}{2}\int_{-\infty}^{\lambda_{j}x+x_{j}}|\eps_{N}|^{2}dy}\to\frac{1}{2}{\int_{-\infty}^{x_{j}(T)}}|z^{*}|^{2}dy\eqqcolon\gamma_{j}^{*}.
	\]
	On the other hand, we have 
	\begin{align*}
		[\calR]_{\textrm{g}_{j}}{\exp\bigg(\frac{i}{2}\int_{-\infty}^{x}\sum_{\ell=1,\ell\ne j}^{N}|[Q]_{\textrm{g}_{\ell}}|^{2}\bigg)}=\bigg[\calR\exp\bigg(\frac{i}{2}\sum_{\ell=1,\ell\ne j}^{N}\int_{-\infty}^{[\cdot-\frac{x_{\ell}-x_{j}}{\lambda_{j}}]\frac{\lambda_{j}}{\lambda_{\ell}}}Q^{2}\bigg)\bigg]_{\textrm{g}_{j}}.
	\end{align*}
	Using $\frac{1}{2}\int_{-\infty}^{x}Q^{2}=\arctan x+\frac{\pi}{2}$,
	\begin{align*}
		\eqref{eq:description u^*2}=\sum_{j=1}^{N}[\calR\cdot e^{i(\td{\gamma}_{j}(t,\cdot)+\gamma_{j}^{*}+o_{t\to T}(1))}]_{\lambda_{j},\gamma_{j},x_{j}},
	\end{align*}
	where 
	\begin{align*}
		\td{\gamma}_{j}(t,\cdot)\coloneqq{\sum_{\ell=1,\ell\neq j,}^{N}}\theta_{\ell,j}(t,\cdot),\quad\theta_{\ell,j}(t,\cdot)\coloneqq\arctan\bigg(\bigg[\cdot-\frac{x_{\ell}-x_{j}}{\lambda_{j}}\bigg]\frac{\lambda_{j}}{\lambda_{\ell}}\bigg)(t)+\frac{\pi}{2}.
	\end{align*}
	Now, we show that $\td{\gamma}_{j}(t,x)$ converges to a constant,
	\emph{independent of $x$} as $t\to T$. For $\ell>j$, using $\frac{\lambda_{j}}{\lambda_{\ell}}\lesssim1$
	and , we have 
	\begin{align}
		\theta_{\ell,j}(t,\cdot)\to\arctan\bigg(\lim_{t\to T}\bigg(\frac{x_{j}-x_{\ell}}{\lambda_{\ell}}\bigg)(t)\bigg)+\frac{\pi}{2}\label{eq:i>j point converge}
	\end{align}
	as $t\to T$. By Proposition~\ref{prop:no bubble tree}, we have $(x_{j}-x_{\ell})\lambda_{\ell}^{-1}(t)\to+\infty$
	or $-\infty$, the right side of \eqref{eq:i>j point converge} is
	$0$ or $\pi$. If $\ell<j$, we use $1\lesssim\frac{\lambda_{j}}{\lambda_{\ell}}$
	and $|(x_{\ell}-x_{j})\lambda_{j}^{-1}|\to\infty$, we have 
	\begin{align*}
		\theta_{\ell,j}(t,\cdot)\to-\arctan\bigg(\lim_{t\to T}\bigg(\frac{x_{\ell}-x_{j}}{\lambda_{j}}\bigg)(t)\bigg)+\frac{\pi}{2},
	\end{align*}
	and this is also $0$ or $\pi$. Therefore, again by DCT, $\td{\gamma}_{j}(t,\cdot)$
	converges to a constant, which we denote by $\td{\gamma}_{j}^{\prime}$,
	and we have 
	\begin{align}
		\eqref{eq:description u^*2}-\sum_{j=1}^{N}[\calR]_{\lambda_{j},\gamma_{j}+\gamma_{j}^{*}+\td{\gamma}_{j}^{\prime},x_{j}}\to0\text{ in }L^{2}.\label{eq:i<j point converge}
	\end{align}
	
	Next, we consider \eqref{eq:description u^*1}. From $\eps_{N}\to z^{\ast}$
	in $L^{2}$, we have $\mathcal{G}^{-1}(\eps_{N})\to\mathcal{G}^{-1}(z^{\ast})$
	in $L^{2}$. By a similar argument, we have 
	\[
	\exp\bigg(\frac{i}{2}\int_{-\infty}^{x}\sum_{\ell=1}^{N}|[Q]_{\textrm{g}_{\ell}}|^{2}dy+o_{t\to T}(1)\bigg)\to\prod_{j=1}^{N}({\textbf 1}_{x<x_{j}(T)}-{\textbf 1}_{x>x_{j}(T)})
	\]
	in pointwise sense. Hence, we have 
	\begin{align*}
		\eqref{eq:description u^*1}\to\mathcal{G}^{-1}(z^{\ast})\prod_{j=1}^{N}({\textbf 1}_{x<x_{j}(T)}-{\textbf 1}_{x\geq x_{j}(T)})\eqqcolon\td z^{\ast}\text{ in }L^{2}.
	\end{align*}
	Thus, we have \eqref{eq:soliton resol ungauge}. Since the modulation
	parameters $\lambda_{j}$ and $x_{j}$ remain unchanged, we also have
	\eqref{eq:asymptotic orthogonality}, $\lim_{t\to T}x_{j}(t)=x_{j}(T)$
	exists with $|x_{j}(T)|<\infty$, and $\lambda_{1}\lesssim\lambda_{2}\lesssim\cdots\lesssim\lambda_{N}\lesssim T-t$.
	We also have $|\td z^{\ast}|=|z^{\ast}|$, which implies $\|\td z^{\ast}\|_{L^{2}}^{2}=\|z^{\ast}\|_{L^{2}}=M(v_{0})-N\cdot M(Q)$.
	If $xu_{0}\in L^{2}$, then $|xv_{0}|=|x\calG(u_{0})|=|xu_{0}|\in L^{2}$,
	which means that $x\td z^{\ast}\in L^{2}$. However, due to discontinuities
	at $x_{j}(T)$, we do not have control of $\partial_{x}\td z^{*}$.
	This completes the proof for finite-time blow-up solutions.
	
	For global solutions in $H^{1,1}$, we argue similarly to the proof
	of Step 3 of Theorem~\ref{thm:Soliton resolution gauged} to obtain
	the scattering or the soliton decomposition 
	\begin{align}
		u(t)-\sum_{j=1}^{N}\textnormal{Gal}_{c_{j}(t)}([\calR]_{{\lambda}_{j}(t),{\gamma}_{j}(t),0})-e^{it\partial_{xx}}u^{\ast}\to0\text{ in }L^{2}\text{ as }t\to+\infty,\label{eq:sec5 ungauge global decom}
	\end{align}
	for some modulation parameters $({\lambda}_{j}(t),{\gamma}_{j}(t),c_{j}(t))$
	and $u^{\ast}=\frac{1}{\sqrt{2\pi}}\calF(\td z^{\ast})$. Using $x\td z^{\ast}\in L^{2}$,
	we have $\partial_{x}u^{\ast}\in L^{2}$.
	
	Now, we show that for chiral solutions $u(t)\in L_{+}^{2}$, we can
	maintain the chirality in the multi-soliton configuration. For the
	case $T<\infty$, since we know $[\calR]_{\lambda_{j},\gamma_{j},x_{j}}\in L_{+}^{2}$
	in \eqref{eq:soliton resol ungauge}, and thus $\td z^{\ast}\in L_{+}^{2}$,
	there is nothing to prove. For the case $T=\infty$, the solitons
	contain the Galilean transforms $\textnormal{Gal}_{c_{j}}([\calR]_{{\lambda}_{j},{\gamma}_{j},0})$.
	Therefore, we might need to adjust the soliton configurations. Note
	that $\calF(\calR)(\xi)=2\sqrt{2}\pi ie^{-\xi}{\textbf 1}_{\xi\geq0}$.
	So, we obtain 
	\begin{align}
		\calF(\textnormal{Gal}_{c_{j}}([\calR]_{{\lambda}_{j},{\gamma}_{j},0}))(\xi)=\lambda_{j}^{\frac{1}{2}}e^{i(\gamma_{j}-tc_{j}^{2}-2tc_{j}\xi)}2\sqrt{2}\pi ie^{-\lambda_{j}(\xi-c_{j})}{\textbf 1}_{\lambda_{j}(\xi-c_{j})\geq0}.\label{eq:soliton fourier side}
	\end{align}
	In view of \eqref{eq:soliton fourier side}, if $c_{j}<0$ for some
	$1\le j\le N$, then $\textnormal{Gal}_{{c}_{j}}([\calR]_{{\lambda}_{j},{\gamma}_{j},0})\notin L_{+}^{2}$.
	On the other hand, if $\lim_{t\to\infty} c_j(t)\eqqcolon c_j(\infty)>0$ for all $1\le j\le N$, then the multi-soliton configuration becomes asymptotically chiral.
	Therefore, by choosing the valid time $\tau$ large enough, we are
	fine. If $c_{J}(\infty)\le0$ for some $1\le J\le N$, we might need
	to make a suitable change of the configuration. We first claim that
	\begin{equation}
		\lim_{t\to\infty}(\lambda_{J}c_{J})(t)=0.\label{eq:claim galilean parameter}
	\end{equation}
	Assume not, i.e. $\liminf_{t\to\infty}(\lambda_{J}c_{J})(t)\eqqcolon(\lambda_{J}c_{J})(\infty)<0$.
	We have for $c_{J}(t)\le0$, from \eqref{eq:soliton fourier side},
	\begin{align}
		\Pi_{+}\{[\textnormal{Gal}_{{c}_{J}}(\calR)]_{{\lambda}_{J},{\gamma}_{J},0}\}=e^{\lambda_{J}c_{J}}[\calR]_{{\lambda}_{J},\gamma_{J}-tc_{J}^{2},2tc_{J}}.\label{eq:5.33-1}
	\end{align}
	This means that, applying $\Pi_{+}$ to \eqref{eq:sec5 ungauge global decom},
	since $u_{0},u(t)\in L_{+}^{2}$, we have the following decoupled
	$L_{+}^{2}$ norm, 
	\begin{align*}
		M_0=M(u(t))=\|u(t)\|_{L_{+}^{2}}^{2}=\sum_{j=1}^{N}e^{\text{min}((\lambda_{j}c_{j})(t),0)}\cdot2\pi+\|u^{*}\|_{L_{+}^{2}}^{2}+o_{t\to\infty}(1).
	\end{align*}
	However, we have 
	\begin{align}
		M(u) & =\liminf_{t\to+\infty}\bigg(\sum_{j=1}^{N}e^{\text{min}((\lambda_{j}c_{j})(t),0)}\cdot2\pi+\|u^{*}\|_{L_{+}^{2}}^{2}\bigg)\nonumber \\
		& \le e^{(\lambda_{J}c_{J})(\infty)}2\pi+2(N-1)\pi+\|u^{*}\|_{L_{+}^{2}}^{2}<2N\pi+\|u^{*}\|_{L^{2}}^{2}=M(u),\label{eq:5.34}
	\end{align}
	which leads to a contradiction, thus proving \eqref{eq:claim galilean parameter}.
	Moreover, \eqref{eq:5.34} also shows that $\|u^{*}\|_{L_{+}^{2}}=\|u^{*}\|_{L^{2}}$
	and $u^{\ast}\in L_{+}^{2}$. Under the condition $c_{J}(\infty)\le0$,
	using \eqref{eq:claim galilean parameter} and \eqref{eq:5.33-1},
	we are able to replace $\textnormal{Gal}_{{c}_{J}}([\calR]_{{\lambda}_{J},{\gamma}_{J},0})$
	with a chiral soliton $[\calR]_{{\lambda}_{J},\gamma_{J}-tc_{J}^{2},2tc_{J}}$
	in the configuration. This finishes the proof. 
\end{proof}
\begin{rem}
	\label{rem:phase correction} One can observe that the \emph{no bubble
		tree property}, as stated in Proposition~\ref{prop:no bubble tree},
	is natural. Indeed, in the proof of Theorem~\ref{thm:Soliton resolution},
	if there were to be a bubble tree, i.e., $\lim_{t\to T}|(x_{\ell}-x_{j})\lambda_{j}^{-1}|\not\to\infty$,
	we could not derive \eqref{eq:i<j point converge} for $\ell<j$.
	Instead, we would have 
	\begin{align*}
		\exp(i\theta_{\ell,j}(t,y))-({\textbf 1}_{y<(x_{\ell}-x_{j})\lambda_{j}^{-1}(t)}-{\textbf 1}_{y>(x_{\ell}-x_{j})\lambda_{j}^{-1}(t)})\to0.
	\end{align*}
	This implies that $\theta_{\ell,j}$ converges to a step function
	with a discontinuity. This means that in the multi-soliton configuration,
	some solitons have discontinuities in phase. We believe this is quite
	unnatural. Note that, so far, there is no finite-time bubble tree
	construction in any models. Therefore, the nonexistence of a bubble
	tree in \eqref{CMdnls} provides sharper information on the soliton
	resolution. 
\end{rem}

\appendix

\section{Decomposition}

\label{sec:appendix variation} We provide the proofs of Lemmas~\ref{lem:Proximity}
and~\ref{lem:Decomposition pre} . They are consequences of variational
structure of the ground state $Q$. 
\begin{lem}[\cite{GerardLenzmann2022}]
	\label{lem:variational argument} Suppose that $\{f_{n}\}_{n\bbN}\subset H^{1}$
	is a sequence such that 
	\begin{align*}
		\sup_{n\in\bbN}\|f_{n}\|_{L^{2}}<\infty,\quad\lim_{n\to\infty}E(f_{n})=0,\quad\|f_{n}\|_{\dot{H}^{1}}=1\text{ for all }n\in\bbN.
	\end{align*}
	Then, there exist a sequence $x_{n}$, and $\lambda,\gamma$ such
	that 
	\begin{align*}
		f_{n}(\cdot-x_{n})\rightharpoonup[Q]_{\lambda,\gamma,0}\text{ weakly in }H^{1},\quad\liminf_{n\to\infty}\|f_{n}\|_{L^{2}}^{2}\geq M(Q)=2\pi.
	\end{align*}
\end{lem}

Now, we prove the tube stability. 
\begin{proof}[Proof of Lemma~\ref{lem:Proximity}]
	By scaling, we may assume $\wt{\lambda}=1$. We suppose the tube
	stability fails. Then, there exist $\delta^{\prime}>0$ and a sequence
	$f_{n}^{\prime}$ in $H^{1}$ such that 
	\begin{align*}
		\sup_{n\in\bbN}\|f_{n}^{\prime}\|_{L^{2}}<\infty,\quad\lim_{n\to\infty}E(f_{n}^{\prime})=0,\quad\|f_{n}\|_{\dot{H}^{1}}=1\text{ for all }n\in\bbN,
	\end{align*}
	and 
	\begin{align}
		\inf_{\wt{\lambda},\wt{\gamma},\wt x}\|f_{n}^{\prime}-[Q]_{\wt{\lambda},\wt{\gamma},\wt x}\|_{\dot{\calH}^{1}}>\delta^{\prime}\text{ for any }n\in\bbN.\label{eq:tube fail}
	\end{align}
	By Lemma~\ref{lem:variational argument}, we have $f_{n}^{\prime}(\cdot-x_{n})\rightharpoonup[Q]_{\lambda,\gamma,0}$
	weakly in $H^{1}$. We write $f_{n}^{\prime}(\cdot-x_{n})=[Q+\td f_{n}]_{\lambda,\gamma,0}$.
	So, we have $\td f_{n}\rightharpoonup0$ weakly in $H^{1}$, and we
	also have $\td f_{n}\to0$ in $L_{\text{loc}}^{2}$ and $L^{p}$ for
	$2<p<\infty$. We have 
	\begin{align}
		E(Q+\td f_{n})=\tfrac{1}{2}\|L_{Q}\td f_{n}\|_{L^{2}}^{2}+O(\|N_{Q}(\td f_{n})\|_{L^{2}}^{2}),\label{eq:appendix energy estimte}
	\end{align}
	where 
	\begin{align*}
		N_{Q}(\td f_{n})=\td f_{n}\mathcal{H}(\Re(Q\td f_{n}))+\tfrac{1}{2}(Q+\td f_{n})\mathcal{H}(|\td f_{n}|^{2}).
	\end{align*}
	Since $\|\td f_{n}\|_{L^{\infty}}\lesssim1$ and $\|Q\td f_{n}\|_{L^{2}}\to0$,
	we have 
	\begin{align*}
		\|\td f_{n}\mathcal{H}(\Re(Q\td f_{n}))\|_{L^{2}}\lesssim\|\td f_{n}\|_{L^{\infty}}\|Q\td f_{n}\|_{L^{2}}\to0.
	\end{align*}
	Using \eqref{eq:HilbertProductRule} with $f=g=|\td f_{n}|^{2}$,
	we deduce 
	\begin{align*}
		\|(Q+\td f_{n})\mathcal{H}(|\td f_{n}|^{2})\|_{L^{2}}^{2} & \lesssim\|Q+\td f_{n}\|_{L^{\infty}}^{2}\|\mathcal{H}(|\td f_{n}|^{2})\|_{L^{2}}^{2}\\
		& =\|Q+\td f_{n}\|_{L^{\infty}}^{2}\bigg(\|\td f_{n}\|_{L^{4}}^{4}+{\int_{\bbR}}|\td f_{n}|^{2}\mathcal{H}(|\td f_{n}|^{2})dx\bigg).
	\end{align*}
	We have $\|\td f_{n}\|_{L^{4}}\to0$. For any $f\in L^{2}$, from
	the isometric property of $\calH$ with $\calH^{2}=-I$, we can know
	that $\int_{\bbR}f\calH(f)dx=0$, and this means that ${\int_{\bbR}}|\td f_{n}|^{2}\mathcal{H}(|\td f_{n}|^{2})dx=0$.
	Therefore, we have $\|N_{Q}(\td f_{n})\|_{L^{2}}\to0$ as $n\to\infty$.
	This leads to $\eqref{eq:appendix energy estimte}\to0$. Now, thanks
	to the subcoercivity of $L_{Q}$ (\cite{KimKimKwon2024arxiv}, Lemma
	A.3), we have $\td f_{n}\to0$ in $\dot{\calH}^{1}$, and we also
	have $f_{n}^{\prime}(\cdot-x_{n})\to[Q]_{\lambda,\gamma,0}$ in $\dot{\calH}^{1}$.
	This contradicts \eqref{eq:tube fail}. 
\end{proof}
\begin{proof}[Proof of Lemma~\ref{lem:Decomposition pre}]
	For a notational convenience, we adopt the notation $\textrm{g}\coloneqq(\lambda,\gamma,x)$
	and $[f]_{\textrm{g}}\coloneqq[f]_{\lambda,\gamma,x}$ in \eqref{eq: g def}
	and \eqref{eq: g def renormal}. We denote by $\textrm{g}_{0}\coloneqq(1,0,0)$.
	
	(1) We equip $\bbR_{+}$ with the metric $\text{dist}(\lambda_{1},\lambda_{2})=|\log(\frac{\lambda_{1}}{\lambda_{2}})|$,
	and equip $\bbR/2\pi\bbZ$ with the induced metric from $\bbR$. We
	first decompose $v$ by $[Q+\wt{\eps}]_{\textrm{g}}$. We define 
	\begin{align*}
		\mathbf{F}(\textrm{g};v)=\mathbf{F}(\lambda,\gamma,x;v)\coloneqq((\wt{\eps},\mathcal{Z}_{1})_{r},(\wt{\eps},\mathcal{Z}_{2})_{r},(\wt{\eps},\mathcal{Z}_{3})_{r})^{T},
	\end{align*}
	where $w\coloneqq\lambda^{1/2}e^{-i\gamma}v(\lambda\cdot+x)$ and
	$\wt{\eps}\coloneqq w-Q$. Then one checks that 
	\begin{align*}
		\partial_{\textrm{g}}\mathbf{F}(\textrm{g}_{0};v)\coloneqq\partial_{\lambda,\gamma,x}\mathbf{F}(1,0,0;v)=\partial_{\textrm{g}}\mathbf{F}(\textrm{g}_{0};Q)+O(\|v-Q\|_{\dot{\calH}^{1}}),
	\end{align*}
	and that the matrix $(\partial_{\textrm{g}}\mathbf{F}(\textrm{g}_{0};Q))_{i,j}$ is diagonal and nonsingular. Thus, by the implicit function theorem, $\mathbf{F}$
	is $C^{1}$ for $(\lambda,\gamma,x)$ and invertible at $(1,0,0,Q)$.
	Let $B_{\delta_{2}}^{\dot{\calH}^{1}}(Q)\coloneqq\{v\in\dot{\calH}^{1}:\|v-Q\|_{\dot{\calH}^{1}}<\delta_{2}\}$.
	For given $R_{0}\gg1$, there exist $0<\delta_{2}\ll\delta_{1}$ and
	$C^{1}$ maps $\textrm{g}=(\lambda,\gamma,x):B_{\delta_{2}}^{\dot{\calH}^{1}}(Q)\to B_{\delta_{1}}(\textrm{g}_{0})$
	such that for given $v\in B_{\delta_{2}}^{\dot{\calH}^{1}}(Q)$, $\textrm{g}(v)$
	is unique solution to $\textbf{F}(\textrm{g};v)=0$ in $B_{\delta_{1}}(\textrm{g}_{0})$.
	We also have that $\text{dist}(\textrm{g},\textrm{g}_{0})\lesssim\|v-Q\|_{\dot{\calH}^{1}}$
	and $\|\wt{\eps}\|_{\dot{\calH}^{1}}\lesssim\|v-Q\|_{\dot{\calH}^{1}}$.
	
	Now, we want to prove the uniqueness of $\textrm{g}$ on $\mathcal{T}_{\delta}$
	for some $\delta$. From the $L^{2}$-scaling, phase rotation, and
	translation invariances, we may assume $\|v-Q\|_{\dot{\calH}^{1}}<\delta$.
	We choose $\eta$ and $\delta$ so that $\eta\ll\delta_{1}$ and $\delta\ll\min(\eta,\delta_{2})$.
	We assume that $v=[Q+\wt{\eps}]_{\textrm{g}}=[Q+\wt{\eps}^{\prime}]_{\textrm{g}^{\prime}}$
	with $\|\wt{\eps}\|_{\dot{\calH}^{1}},\|\wt{\eps}^{\prime}\|_{\dot{\calH}^{1}}<\eta$
	and $(\wt{\eps},\mathcal{Z}_{k})_{r}=(\wt{\eps}^{\prime},\mathcal{Z}_{k})_{r}=0$
	for $k=1,2,3$. Then, by the scaling, rotation, and translation symmetries
	and changing the roles of $(\textrm{g},\eps)$ and $(\textrm{g}^{\prime},\eps^{\prime})$
	if necessary, we may assume $\lambda\geq1$, $\lambda^{\prime}=1$,
	$\gamma^{\prime}=0$, and $x^{\prime}=0$. Then, we have $Q-[Q]_{\textrm{g}}=[\wt{\eps}]_{\textrm{g}}-\wt{\eps}^{\prime}$
	and we deduce 
	\begin{align*}
		\|Q-[Q]_{\textrm{g}}\|_{\dot{H}^{1}}=\|[\wt{\eps}]_{\textrm{g}}-\wt{\eps}^{\prime}\|_{\dot{H}^{1}}\leq\lambda^{-1}\|\wt{\eps}\|_{\dot{\calH}^{1}}+\|\wt{\eps}^{\prime}\|_{\dot{\calH}^{1}}\lesssim\eta\ll\delta_{1}.
	\end{align*}
	So, we have $\text{dist}(\textrm{g},\textrm{g}^{\prime})<\frac{\delta_{1}}{2}$.
	Since $\text{dist}(\textrm{g},\textrm{g}_{0})\lesssim\|v-Q\|_{\dot{\calH}^{1}}<\delta\ll\delta_{1}$,
	we have $\text{dist}(\textrm{g},\textrm{g}_{0})<\frac{\delta_{1}}{2}$,
	and thus we derive $\text{dist}(\textrm{g}^{\prime},\textrm{g}_{0})<\delta_{1}$.
	By the uniqueness which comes from the implicit function theorem,
	we conclude $\textrm{g}=\textrm{g}^{\prime}$.
	
	(2) Let $\wt{\lambda}\coloneqq\frac{\|Q\|_{\dot{H}^{1}}}{\|v\|_{\dot{H}^{1}}}$.
	Then, we have $\|Q\|_{\dot{H}^{1}}=\wt{\lambda}\|w\|_{\dot{H}^{1}}=\wt{\lambda}\|Q+\wt{\eps}\|_{\dot{H}^{1}}$.
	Since $\|\wt{\eps}\|_{\dot{H}^{1}}<\eta\ll1$, we have $\frac{\lambda}{\wt{\lambda}}\sim1$
	and \eqref{eq:Decompose lambda estimate}. This finishes the proof. 
\end{proof}

 \bibliographystyle{abbrv}
\bibliography{referenceCM}

\end{document}